\documentclass{article}
\usepackage{amsmath,amsthm,amssymb}
\usepackage{authblk}
\theoremstyle{plain}
\newtheorem{theo}{Theorem}
\newtheorem*{theo*}{Theorem}
\newtheorem{assume}[theo]{Assumption}
\newtheorem{coro}[theo]{Corollary}
\newtheorem{proposition}[theo]{Proposition}
\newtheorem*{theoremA}{Theorem A}
\newtheorem*{theoremB}{Theorem B}
\newtheorem*{theoremC}{Theorem C}
\newtheorem*{theoremD}{Theorem D}

\newtheorem{lemma}[theo]{Lemma}
\theoremstyle{definition}
\newtheorem{define}[theo]{Definition}
\newtheorem{remark}[theo]{Remark}
\usepackage{hyperref}
\usepackage{color}

\usepackage[T1]{fontenc}

\newcommand{\ab}{\text{ab}}

\newcommand{\ep}{\varepsilon}
\newcommand{\IR}{\mathbb{R}}
\newcommand{\IZ}{\mathbb{Z}}
\newcommand{\IC}{\mathbb{C}}
\newcommand{\IN}{\mathbb{N}}
\newcommand{\p}{\partial}
\newcommand{\tr}{\mathrm{tr}}
\newcommand{\Tr}{\mathrm{Tr}}

\newcommand{\pp}{\rho}

\title{Instantons on multi-Taub-NUT Spaces I:\\
 Asymptotic Form and  Index Theorem}

\author[*]{Sergey A. Cherkis}
\author[**]{Andr\'es Larra\'in-Hubach}
\author[$\star$]{Mark Stern}
\affil[*]{\small Department of Mathematics, University of Arizona, Tucson, AZ 85721-0089}
\affil[ ]{\small\tt cherkis@math.arizona.edu}
\affil[**]{\small Department of Mathematics, University of Dayton, Dayton, OH 45469}
\affil[ ]{\small\tt alarrainhubach1@udayton.edu}
\affil[$\star$]{\small Department of Mathematics, Duke University, Durham, NC 27708-0320}
\affil[ ]{\small\tt stern@math.duke.edu}

\begin{document}

\begin{titlepage}

\renewcommand{\thepage}{ }
\date{}  

\maketitle

\abstract{
We study finite action anti-self-dual Yang-Mills connections on the multi-Taub-NUT space.   
 Under a technical assumption of generic asymptotic holonomy, we establish the curvature and the harmonic spinor decay rates and compute the index of the associated Dirac operator.   

This is the first in a series of papers proving the completeness of the bow construction of instantons on multi-Taub-NUT spaces and exploring it in detail. 
}

\end{titlepage}

\tableofcontents

\section{Introduction}
This paper establishes core analytic results needed for proving the completeness of the bow construction \cite{Cherkis:2010bn} of  instantons on Asymptotically Locally Flat (ALF) spaces. In general, for an oriented Riemannian four-manifold $M,$ we call a connection $A$ on a rank $n$ Hermitian bundle $\mathcal{E}\rightarrow M$
  an  {\em instanton} if it has square integrable, anti-self-dual (ASD) curvature $F_A.$  Our focus is on the case when $M$ is a prototypical ALF space: the multi-Taub-NUT space.  The bow construction, just as the AHDM-Nahm construction \cite{Atiyah:1978ri,Nahm1979-2,Nahm:1983sv}, relates instantons on ALF spaces to solutions of a system of nonlinear ordinary differential equations on a collection of line segments and additional linear data on their boundary.  In \cite{Cherkis:2010bn} this data is conveniently organized in terms of a `bow'.  In comparison, the equivariant version of the ADHM construction, studied by Kronheimer and Nakajima \cite{KN}, relates instantons on Asymptotically Locally Euclidean (ALE) spaces to solutions of algebraic equations originating from a quiver.   In both cases, the rank of the bow or quiver system is determined by the dimension of the space of $L^2$ harmonic spinors twisted by an instanton connection on the ALF or ALE space.  In this paper we compute this dimension via an index theorem, as a first step in proving the above correspondence.  Establishing the asymptotic form of the connection is crucial both for the index calculation and for finding the asymptotic form of these harmonic spinors.

In order to establish the asymptotic form of the connection, we  first prove  quadratic decay of  the curvature.  The curvature of the known $U(1)$ instantons on ALF spaces decays quadratically  \cite{Rubak}, so the decay rate we establish is sharp. 
There has been extensive work by numerous authors analyzing the decay rates of the curvature of Yang-Mills and instanton connections. Uhlenbeck \cite{Uhlenbeck1,Uhlenbeck2} showed that a finite energy Yang-Mills connection on $\mathbb{R}^4$ has quartic curvature decay. This result has been generalized in many directions, including \cite{FU}, \cite{IN}, \cite[Sec.4.4.3]{DK}, and \cite{Rade}. In particular, Groisser and Parker \cite{Groisser} extend the quartic decay to  asymptotically flat spaces (as defined in \cite[Sec.1]{Groisser}), and their proof readily extends to ALE spaces.

The decay rate is sensitive to the asymptotic form of the metric, especially to the asymptotic volume growth.   For example, Mochizuki \cite{Mochizuki} studied the case of doubly periodic instantons, proving $|F|=O\left(\frac{1}{r^{1+\epsilon}}\right).$  Here $r$ denotes the distance to the origin in $\mathbb{R}^2$ pulled back to $\IR^2\times S^1\times S^1.$  A case  very closely related to the one we consider here is that of a monopole on $\mathbb{R}^3$, for which  Jaffe and Taubes  \cite[Thm. 10.5]{JT} and Jarvis \cite[Thm. 19]{Jarvis}\footnote{We thank  \'Akos Nagy for pointing out this reference.} proved quadratic decay of the curvature.

Proving quadratic decay for ALF spaces is more delicate than proving quartic decay in the ALE case.  To see this, consider the ordinary differential equation on $[1,\infty)$:
$$\left(-\frac{d^2}{dr^2} + \frac{p}{r^2}\right)y = 0.$$
The solutions of this equation are $c_1r^{\frac{1}{2}+\sqrt{p+\frac{1}{4}}}+ c_2r^{\frac{1}{2}-\sqrt{p+\frac{1}{4}}},$ and the decay rate is determined by the coefficient $p$. The analogous equation for Yang-Mills equations is the Bochner formula
$$\nabla^*\nabla F_A - \epsilon^i\epsilon_j^*{\cal R}_{ij}F_A - \epsilon^i\epsilon_j^*[F_{ij},F_A]=0.$$
Here ${\cal R}$ denotes the Riemann curvature, which decays cubically for ALF spaces, and therefore will not critically affect our estimates. The $\mathrm{ad}(F_{ij})$ term, however, is analogous to the $\frac{p}{r^2}$ term, with $p$ completely unknown. Hence, the unknown magnitude of the curvature, which we wish to bound, appears to play an important role in establishing bounds.  
Terms with faster than quadratic decay are negligible in such analyses, making the ALE case significantly simpler. In particular, the Sobolev and Hardy inequalities for ALE spaces are stronger than for ALF spaces. These inequalities, coupled to a Moser iteration argument, readily imply that instantons on ALE spaces decay faster than $r^{-q}$ for some $q>2$. Once such decay is demonstrated, the ALE decay problem is effectively linearized and readily solved. 

 As proved in \cite{Minerbe10,ChenChen} ALF spaces whose Riemann curvature decays faster than quadratically are asymptotic to the multi-Taub-NUT metric described below. In this paper, we therefore focus our attention on the case of the manifold $M$ a $k$-centered Taub-NUT~(TN$_k$) space \cite{TNUT,Newman:1963yy,HAWKING197781,Gibbons:1979zt}; the general ALF (with faster than quadratic  Riemann curvature decay) case then follows from essentially the same argument. TN$_k$ is a hyperk\"ahler four-manifold with a triholomorphic isometric circle action, which has $k$ fixed points, $\{\nu_1,\ldots,\nu_k\}$.  The quotient of TN$_k$ by this circle action is $\IR^3$, and the quotient map  
$\pi_k: \mathrm{TN}_k\to \IR^3$ defines a principal circle fibration in the complement of the fixed points:
\begin{equation}\label{fiber}S^1\rightarrow \mathrm{TN}_k\setminus\{ \nu_1,\ldots, \nu_k\}\xrightarrow{\pi_k} \mathbb{R}^3\setminus\{\nu_1,\ldots,\nu_k\}.\end{equation}
This  $S^1$ fibration has Hopf number $-1$ over any small sphere centered at any fixed point $\nu_\sigma.$  Over spheres of large radius, this fibration restricts to a circle bundle of Hopf  number $-k$.  
The  circle bundle (\ref{fiber}) admits a connection one-form $\varpi$ with curvature $d\varpi=\pi_k^*(*_{3} dV),$ where $\ast_3$ denotes the Hodge star operator on Euclidean $\IR^3$ and $V$ is the following function on 
$\mathbb{R}^3\setminus\{\nu_1,\ldots,\nu_k\}:$
\begin{align}\label{metricpotential}
V(x)=\ell+\sum_{\sigma=1}^k\frac{1}{2\left|x-\nu_\sigma\right|},\end{align}
here $\ell$ is a real positive constant fixed throughout this paper. 
  The hyperk\"ahler metric on TN$_k$ has the following form found by Gibbons and Hawking \cite{Gibbons:1979zt}: 
\begin{align}\label{metric}
g = Vdx^2+\frac{\varpi^2}{V}.
\end{align}
In local coordinates $(x,\tau)\in U \times S^1$ with $U\subset\mathbb{R}^3\setminus\{\nu_j\}_{j=1}^k$  an open contractible neighborhood  and $\tau\in[0,2\pi)$ a homogeneous coordinate along the circle fiber, we have  
\begin{align}\label{varpi}
\varpi=d\tau+\pi_k^*\omega,
\end{align} with $\omega$ a one-form   
on $U$ satisfying 
\begin{align}\label{tndef}d\omega=*_{3}dV =-*_{3}\frac{kdr}{2r^2}+O(r^{-3}).
\end{align} 
Here $r=|x|.$ 
(For the sake of brevity, in the rest of the paper we shall use $\omega$ to denote $\pi^*_k(\omega)$ as well, thus writing $\varpi=d\tau+\omega,$ just as we used $V$ for $\pi_k^*(V)$ in Eq.~\eqref{metric} and $\nu_\sigma$ for $\pi_k^*(\nu_\sigma)$ in \eqref{fiber}, hoping this will cause no confusion). 
The most important properties of $\mathrm{TN}_k$ that we use  in this paper are that it is a complete, hyperk\"ahler Riemannian manifold (with an ASD Ricci flat metric) with cubically decaying Riemannian curvature and an asymptotic  tri-holomorphic circle action.  
 
Our analysis assumes a generic asymptotic holonomy condition (see Definition~\ref{genconn} on page~\pageref{genconn}):  there exists a ray in the base $\mathbb{R}^3$ such that the holonomy of the instanton connection around every circle fiber over that ray is generic (i.e. belongs to a common compact subset of  the regular $U(n)$ adjoint orbit). In other words, we assume that the  eigenvalues of the holonomy over some ray in the base $\mathbb{R}^3$ are distinct with the distance between distinct eigenvalues bounded below by a positive constant.  
There are known instantons, such as e.g. those found by Etesi and Hausel in \cite{Etesi:2002cc}, that do not satisfy our assumption.  In fact, these latter have asymptotic holonomy approaching identity at infinity.  We limit ourselves here to the generic case and intend to address the general case in the future. 

Under the generic asymptotic holonomy assumption we prove in Section~\ref{Sec:Instantons} the following:
\begin{theoremA} 
Let $M$ be an ALF space and $A$  an instanton connection with generic asymptotic holonomy. Let $o\in M$, and let $F_A$ denote the curvature of $A$. Then there exists $C>0$ so that $|F_A|(p)\leq \frac{C}{\mathrm{dist}(p,o)^{2}}$ for any $p\in M.$
\end{theoremA}
 
Using this quadratic decay property, we prove in Section~\ref{Sec:Asymp} that the asymptotic form of the instanton connection is that of a direct sum of the known $U(1)$ instantons:
\begin{theoremB}
Let $A$ be an instanton on a Hermitian bundle $\mathcal{E}\rightarrow\mathrm{TN}_k$ with generic asymptotic holonomy. 
Then, outside of a compact set, $\mathcal{E}$ splits as a direct sum of line bundles, and the  connection  has the form  
$$A= \, \mathop{\oplus}_a \left(-i (\lambda_a+\frac{m_a}{2 |x|})\frac{\varpi }{V}+\pi_k^* \eta_a\right)+O(|x|^{-2}),$$ 
with $\lambda_a\in\mathbb{R}$, $m_a\in\IZ$, and $\eta_a$ a connection on a line bundle $W(a)$  (defined in Thm.~\ref{asyholo}) over $\IR^n\setminus K$, for some compact set $K$. Moreover $W(a)$ restricted to $S^2$ has  Chern number $m_a$.    
\end{theoremB}
The fact that all $U(1)$ instantons have this form was proved in \cite{HitchinL2} (see also  \cite{HHM}).

\ 

We fix the orientation of TN$_k$ by setting $d\mathrm{Vol} = V d\mathrm{Vol}_{\IR^3}\wedge \varpi $. With respect to this orientation, the K\"ahler forms are  self-dual, and the Riemann curvature is anti-self-dual.
 Under the Clifford action of the volume form the spin bundle $S\rightarrow\mathrm{TN}_k$ splits as $S=S^-\oplus S^+$ with $S^-$ and $S^+$ denoting respectively, the negative and the positive eigenvalue eigen-bundles of the chirality operator  $\gamma^5:=-c(d\mathrm{Vol}).$  (Here $c$ denotes the Clifford action, and this common sign convention ensures that the chirality operator is compatible with the Hodge star action on two-forms: $\gamma^5 c(\eta)=c(*\eta).$)  Due to hyperk\"ahlerity,  the bundle $S^+$ is trivial.
Let $D_A^{\pm}$ denote the restriction of the Dirac operator $D_A$ to sections of $S^\pm\otimes \mathcal{E}.$ 

In Section~\ref{Sec:HarmSpinors} we prove  decay estimates for the $L^2$ harmonic spinors, which will be needed in our subsequent analysis of the bow construction.  

\begin{theoremC}
Let $A$ be an instanton  on TN$_k$ with generic asymptotic holonomy and $\psi\in \mathrm{Ker}(D_A^-)\cap L^2$, then 
$|\psi|$ decays exponentially if the asymptotic holonomy has no invariant vectors, i.e. if for all $a$, $\frac{\lambda_a}{\ell}\notin\mathbb{Z}$; $|\psi|$  decays quadratically otherwise.
\end{theoremC}
With our orientation and chirality conventions, for an instanton $A$ one has $\mathrm{Ker}(D_A^+)\cap L^2 = \{0\}$. Hence, in order to compute the dimension of $\mathrm{Ker}(D_A)\cap L^2$ it suffices to compute the $L^2$ index of $D_A^-.$  We prove the following index theorem in Section~\ref{Sec:Index}: 

\begin{theoremD}
The $L^2-$index of $D_A^-$ is
\begin{align}
\mathrm{ind}_{L^2} D_A^-=\sum_a\left((\{\frac{\lambda_a}{\ell}\}-\frac{1}{2})(m_a-k\lfloor\frac{\lambda_a}{\ell}\rfloor)-\frac{k}{2}\{\frac{\lambda_a}{\ell}\}^2\right)+\frac{1}{8\pi^2}\int \tr\, F\wedge F,
\end{align}
where $\lfloor x\rfloor$ denotes the largest integer not greater then $x$ and $\{x\}=x-\lfloor x\rfloor.$  
\end{theoremD}

Our techniques for analyzing the decay of harmonic spinors in the Fredholm case follow closely the work of Agmon \cite{agmon}. We introduce a new iterated maximum principle  to treat the non-Fredholm case. The decay of harmonic spinors on  spaces with quadratically decaying Green's operator (such as one finds in ALE spaces) was also considered in \cite{Feehan}. Our treatment of the index theorem follows closely the approach of \cite{SternH} (see also \cite{SSZ}). One might also apply the general fibered boundary index formula of \cite{Leichtnam}\footnote{We thank Fr\'ed\'eric Rochon for providing this reference.}, which would also require evaluating the $\eta$-invariant boundary term, as in \cite{Hitchin74}.

\subsubsection*{A Perspective}
Moduli spaces of instantons on multi-Taub-NUT  are gaining significance in both mathematics and physics.  They play a central role in the Geometric Langlands correspondence for complex surfaces.  The original versions of this correspondence \cite{BF1,BF2} focused on instantons on ALE spaces, however, the physics picture   \cite{Tan:2008wp,Witten:2009at} reveals that instantons on multi-Taub-NUT tell a richer story.   
These instanton moduli spaces are also significant in quantum field theory, since they appear  as Coulomb branches of three-dimensional $N=4$ supersymmetric gauge theories \cite{Seiberg:1996nz,deBoer:1996mp,Nakajima:2016guo}, and as  both the Coulomb branches and the Higgs branches of Seiberg-Witten theories with impurities \cite{Gaiotto:2008ak,Cherkis:2011ee}.  A mathematical treatment of these spaces and their relation to bows appeared in \cite{Nakajima:2016guo}.

These interpretations of the instanton moduli space give precise predictions for the dimension of their $L^2$ cohomology.  See, e.g. \cite{Moore:2015qyu} for the case of monopole moduli spaces, \cite{nakajima1994,BF1} for the case of instantons on ALE spaces, and \cite{Tan:2008wp,Witten:2009at,Cremonesi:2013lqa,Nakajima:2015txa} for our case of instantons on ALF spaces.   It  is a challenging problem to verify these predictions by directly computing the $L^2$ cohomology.  In fact, since the direct study of instanton moduli spaces  presents numerous analytic challenges, it is desirable to have a simpler descriptions of these spaces.  The bow construction \cite{Cherkis:2010bn} delivers such a description by suggesting that instantons are in correspondence with bow solutions and that instanton moduli spaces are isomorphic to bow moduli spaces.  The bow moduli spaces are much more amenable to computation.  For example, their asymptotic metric was found in \cite{Cherkis:2010bn} and, for the metric on the moduli space of a low  rank bow representation, it was  computed explicitly in \cite{Cherkis:2008ip}. In this sequence of papers, we set out to prove that, indeed, bow moduli spaces are isometric to the moduli spaces of instantons on multi-Taub-NUT space.

Another significant application of instantons on multi-Taub-NUT space is in string theory where they deliver a description of the effective dynamics of the Chalmers-Hanany-Witten brane configurations \cite{Chalmers:1996xh,Hanany:1996ie}. This relation provides significant information about the instantons themselves, as demonstrated in \cite{Witten:2009xu}.

\section{Analytic Preliminaries: Moser Iteration}
Let $(M,g)$ be a smooth complete n-dimensional manifold with bounded geometry, i.e. its injectivity radius is bounded below, and its Riemann curvature tensor $\mathcal{R}$ has norm bounded above: $\|\mathcal{R}\|_{L^\infty(M)}<\infty$. Let $\delta(M)$ denote the injectivity radius of $M$. In such geometries,  the following local Sobolev embedding theorem for geodesic balls holds with  uniform Sobolev constant. 
\begin{proposition}[{{\cite[Chapter 2, Lemma 2.24]{aubin}}}] Let $M$ be a manifold of bounded geometry. Then there exists 
$ S_M>0$, depending only on $\delta(M)$ and $\|\mathcal{R}\|_{L^\infty}$, such that for all $ p\in M$, $ R<\frac{\delta(M)}{2} $, and all  $ f\in C_c^\infty(B_{R}(p))$,  one has 
\begin{align}\label{sob3}S_M \|d \xi\|^2_{L^2(B_R(p))}\geq \| \xi\|_{L^{\frac{2n}{n-2}}(B_R(p))}^2.
\end{align}
\end{proposition}
We remark that the multicenter Taub NUT spaces have bounded geometry. We will estimate curvature $F_A$ and related quantities using  Moser iteration. For the convenience of the reader and to clarify dependence on various constants, we recall the theorem and standard proof in the form we need. (See \cite[Lemma 1.2, p. 54]{CW}.)

\begin{proposition}\label{moserflex}
Let $M$ be a complete Riemannian n-manifold ($n>2$) of bounded geometry.   Let $\lambda>0$ and $R<\frac{\delta(M)}{2+2\lambda}$. 
Suppose $f$ is a nonnegative function satisfying 
\begin{align}\label{seed}\Delta f\leq w^2f,\end{align}
 for  a nonnegative function  $w\in L^{\infty}(B_{(1+\lambda)R}(p))$. Set $W:=R\|   w\|_{L^{\infty}(B_{(1+\lambda)R}(p))},$ and  
 $\pp_k:=(\frac{n}{n-2})^k.$  
Then 
\begin{align}\label{moserdoneii}
\|f\|_{L^{\infty}(B_{R}(p))}\leq  
 R^{-n} S_M^{-\frac{n}{2}}( \lambda^{-2} + W^2)^{\frac{n}{2}}\frac{2\prod_{k=0}^\infty 4^{\frac{k+1}{  \pp_k}}  }{(1-(\frac{1}{4})^{\frac{1}{n}})^{n}}  \| f \|_{L^1(B_{(1+\lambda)R}(p))}.
\end{align}
\end{proposition}
\begin{proof}
 Set $R_k:= R+2^{-k}\lambda R.$  Let $\eta(s)$ denote a $C^1$ cutoff function such that 
$\eta(s) =1$ for $s\leq \frac{3}{2}$, $\eta(s) = 0$ for $s\geq 2$, and $|d\eta|<4$. Define radial cutoff functions by 
\begin{align}\eta_k(x) = \eta(1+\frac{2^{k-1}}{\lambda}( \frac{\mathrm{dist}(x,p)}{R}-1)).\end{align}
Then $\eta_k$ is identically 1 on $B_{R_k}(p)$, it vanishes identically on the complement of $B_{R_{k-1}}(p)$, 
and
\begin{align}|d\eta_k|(x)\leq \frac{2^{k+1}}{\lambda R}\chi_{B_{R_{k-1}}(p)},\end{align}
where $\chi_B$ denotes the characteristic function of $B$. 

  Let $\rho_k $ be any positive number for now and  multiply \eqref{seed} by $\eta_k^2f^{2\pp_k-1}$ and integrate.  Integration by parts and manipulation yields 
 \begin{align}\label{nextagain}  & \int_M w^2 f^{2\pp_k }  \eta_k^2dv \geq\int_M  f^{2\pp_k-1 }  \eta_k^2\Delta fdv\nonumber\\
=&\frac{2\rho_k-1}{\rho_k^2}\|\eta_kd ( f^{\pp_k})\|^2_{L^2(M)} 
 +\frac{2}{\rho_k}  \langle   f^{\pp_k }d\eta_k,  \eta_k d( f^{\pp_k }) \rangle_{L^2(M)}\nonumber\\
=&\frac{2\rho_k-1}{\rho_k^2}\|  d(   \eta_kf^{\pp_k})\|^2_{L^2(M)} 
-\frac{1}{\rho_k^2}\| f^{\rho_k}d\eta_k\|^2_{L^2(M)}\nonumber\\
&-\frac{2\rho_k-2}{\rho_k^2}\langle d ( \eta_kf^{\pp_k}),f^{\rho_k}d\eta_k\rangle_{L^2(M)}
\nonumber\\
  & \geq \frac{ 1}{ \rho_k }(\|d( \eta_kf^{\pp_k})\|^2_{L^2(M)} 
-  \|   d(\eta_k) f^{\pp_k } \|^2_{L^2(M)}) . 
 \end{align}
By the definition of $W$,
\begin{align}\label{Wdef}\int_M w^2 f^{2\pp_k }  \eta_k^2dv\leq \frac{ W^2}{R^2}\int_M f^{2\pp_k }\eta_k^2dv.\end{align}
Applying \eqref{sob3} and \eqref{Wdef} to \eqref{nextagain} yields 

\begin{align}\label{nextagain2}  
 \| \eta_kf^{\pp_k}\|^2_{L^{\frac{2n}{n-2}}(M)}
\leq   R^{-2}S_M4^{k+1} \int_{B_{R_{k-1}}}(  W^2 +\lambda^{-2})f^{2\pp_k }dv . 
\end{align}
Hence  
\begin{align}\label{nextagainb}    
\|f\|^{2\pp_k}_{L^{\frac{2n}{n-2}\pp_k}(B_{R_k}(p))}
\leq  S_M R^{-2} 4^{k+1}(\lambda^{-2}  + W^2)\| f \|_{L^{2\pp_k}(B_{R_{k-1}}(p))}^{2\pp_k} . 
\end{align}
  Now fix $ \pp_k =  (\frac{n}{n-2})^k.$ We rewrite  \eqref{nextagainb} as
\begin{align}\label{close}
\|f\|_{L^{2\rho_{k+1}}(B_{R_k}(p))}
 \leq  (4^{k+1}S_M)^{\frac{1}{2 \pp_k}}R^{-\frac{1}{ \pp_k}}  
 (\lambda^{-2}   +  W^2  )^{\frac{1}{2 \pp_k}}\| f \|_{L^{2\pp_k}(B_{R_{k-1}}(p))}. 
\end{align}
Iterating gives
\begin{align}\label{moserdone0}
 \|f\|_{L^{\infty}(B_{R}(p))}\leq  
 R^{-\frac{n}{2 }} S_M^{-\frac{n}{4 }}( \lambda^{-2} +  W^2)^{\frac{n}{4}}  \| f \|_{L^{2}(B_{(1+\lambda)R}(p))}  \prod_{k=0}^\infty 2^{\frac{k+1}{  \pp_k}}.
\end{align}
 In order to replace the $L^2$ norm of $f$ by the $L^1$ norm on the right-hand side of \eqref{moserdone0}, we  shall use 
\cite[Lemma 4.1, p. 27]{CW}. Set $\phi(s) := \|f\|_{L^{\infty}(B_{s}(p))}$. Then     
\begin{align}\label{407}&\phi(R)\leq 
 R^{-\frac{n}{2 }} S_M^{-\frac{n}{4 }}( \lambda^{-2} +  W^2)^{\frac{n}{4}}  \| f \|_{L^{2}(B_{(1+\lambda)R}(p))}  \prod_{k=0}^\infty 2^{\frac{k+1}{  \pp_k}}\nonumber\\
&\leq 
 (\lambda R)^{-\frac{n}{2}} S_M^{-\frac{n}{4}}(1 + \lambda^2W^2)^{\frac{n}{4}}   \phi((1+\lambda)R)^{\frac{1}{2}} \| f \|_{L^{1 }(B_{(1+\lambda)R}(p))}^{\frac{1}{2}} \prod_{k=0}^\infty 2^{\frac{k+1}{\pp_k}}\nonumber\\
&\leq  \frac{1}{2}\phi((1+\lambda)R) + (\lambda R)^{-n} S_M^{-\frac{n}{2 }}(1 + \lambda^2W^2)^{\frac{n}{2}}  \| f \|_{L^{1}(B_{(1+\lambda)R}(p))}\prod_{k=0}^\infty 4^{\frac{k+1}{  \pp_k}},
\end{align}
where we have used the arithmetic/geometric mean inequality at the last step. 
Introducing 
$$A=S_M^{-\frac{n}{2 }}(1 + \lambda^2W^2)^{\frac{n}{2}}  \prod_{k=0}^\infty 4^{\frac{k+1}{  \pp_k}},$$  
we write \eqref{407}  for $\lambda\leq 1$ as   
\begin{align}\label{CW27}\phi(s)\leq \frac{1}{2}\phi(t)+(t-s)^{-n}A\| f \|_{L^{1}(B_{t}(p))}.
\end{align}
 \cite[p. 27, Lemma 4.1]{CW} states that the inequality \eqref{CW27} implies 
$$\phi(s)\leq (t-s)^{-n}\frac{2A}{(1-(\frac{1}{4})^{\frac{1}{n}})^{n}}\phi(t).$$
Hence
\begin{align}\|f\|_{L^{\infty}(B_{R}(p))}\leq R^{-n}S_M^{-\frac{n}{2 }}(\lambda^{-2} +  W^2)^{\frac{n}{2}} \frac{2 \prod_{k=0}^\infty 4^{\frac{k+1}{  \pp_k}}}{(1-(\frac{1}{4})^{\frac{1}{n}})^{n}}\| f \|_{L^{1}(B_{(1+\lambda)R}(p))},
\end{align}
as desired. 
\end{proof}
The following corollary is a typical application of Proposition \ref{moserflex}.  Here we set $r$ to be the distance to the origin in $\mathbb{R}^3$.
\begin{coro}\label{moserapp} Suppose for some $\rho >0$, 
$f\in C^2(\IR^3\setminus B_\rho(0))$ is a nonnegative function satisfying $\Delta f\leq w^2f$. Let  $y\in \IR^3\setminus B_\rho(0)$, and $w^2<\frac{c^2}{r^2+1}$ in $B_{\tilde R}(y)$ for some $\tilde R<|y|-\rho$, and some $c>0$. Then there exists $C>0$  so that for every $x\in B_{\tilde R}(y)$,
\begin{align}f(x)\leq \frac{C(1+c^2)^{\frac{3}{2}}}{(\tilde R-|y-x|)^3}\int_{B_{\tilde R}(y)}fdv.
\end{align} 
In particular, if $w^2(x)<\frac{c^2}{r^2+1}, \forall x\in \IR^3\setminus B_\rho(0)$, then for every  $x\in \IR^3\setminus B_\rho(0)$, 
\begin{align}f(x)\leq C(1+c^2)^{\frac{3}{2}}|x|^{-3}\int_{\IR^3}fdv.
\end{align} 
\end{coro}
\begin{proof}Let 
$x\in B_{\tilde R}(y)$. In $B_{\frac{\tilde R-|y-x|}{2}}(x)$, $r>|x|-\frac{\tilde R-|y-x|}{2}\geq \frac{\tilde R-|y-x|}{2}$. 
Hence $w^2(z)<\frac{c^2}{(\frac{\tilde R-|y-x|}{2})^2}$, for $z\in B_{\frac{\tilde R-|y-x|}{2}}(x)$. Now set $R=\frac{\tilde R-|y-x|}{4}$ and $\lambda = 1$ and apply Proposition \ref{moserflex} to estimate 
$$|f(x)|\leq \|f\|_{L^\infty(B_R(x))}\leq  \frac{C(1+c^2)^{\frac{3}{2}}}{(\tilde R-|y-x|)^3}\int\limits_{B_{  2R}(x)}fdv\leq  \frac{C(1+c^2)^{\frac{3}{2}}}{(\tilde R-|y-x|)^3}\int\limits_{B_{\tilde R}(y)}fdv,$$ 
proving the first claim. The second claim is an immediate consequence of the first, choosing $x=y$ and $\tilde R=\frac{1}{2}(|x|-\rho)$. 
\end{proof}

\section{Decay of Yang-Mills Curvature}\label{Sec:YMdecay}
Our preliminary estimates for the curvature decay require only the weaker hypothesis that the connection is smooth, finite action Yang-Mills. In this section we work under this weaker hypothesis. 

We call a smooth connection $A$ {\em a Yang-Mills connection} if its curvature satisfies the Yang-Mills equation $d_A^*F_A=0$.  Note that the Bianchi identity, $d_AF_A=0,$ implies that any connection with anti-self-dual curvature is also a Yang-Mills connection.  Given a local orthonormal frame $\{e_j\}_j$ and coframe $\{e^j\}_j$, let
$\ep^j$ denote exterior multiplication on the left by $e^j$ and let 
$\ep^*_j$  denote the adjoint operation - interior multiplication by the metrically dual vector field. 
(In general, for a differential form $\omega$, $\ep(\omega)$ will denote the exterior multiplication on the left by $\omega$, while $\ep^*(\omega)$ will denote the adjoint operation.)  
With this notation, when $A$ is Yang-Mills, its curvature 
$F=F_A$ satisfies the Bochner formula:
\begin{equation}\label{bochner} 
0 =(d_Ad_A^*+d^*_Ad_A)F= \nabla^*\nabla F - \ep^i\ep^*_j{\cal R}_{ij}F - \ep^i\ep^*_j[F_{ij},F],
\end{equation}
where $\mathcal{R}$ denotes the Riemann curvature operator.

We recall in our context, Uhlenbeck's $\epsilon-$regularity theorem for Yang-Mills connections.   
\begin{lemma}[{{\cite[Theorem 2.2.1]{Tian}}}]\label{epreg} 
Let $A$ be a smooth Yang-Mills connection with $L^2$ curvature on a Hermitian  vector bundle $\mathcal{E}$ over a Riemannian 4-manifold of bounded geometry. There exist constants $\epsilon, C>0$ and $R\in (0,\frac{\delta(M)}{2})$ such that for every $ p\in M$,  
if $\rho\in (0,R)$  is small enough so that 
\begin{align}\label{epsreg}\int_{B_\rho(p)}|F_A|^2dv<\epsilon,
\end{align}
then
\begin{align}\label{epsregc} |F_A|^2(p)\leq\frac{C}{\rho^4}\int_{B_\rho(p)}|F_A|^2dv.
\end{align}
\end{lemma}  
\begin{remark}
The constants $\epsilon$, $C$, and $R$ in this theorem depend only on the injectivity radius, the magnitude of the Riemannian curvature (through the expression for $\Delta|F_A|^2$), and the Sobolev constant $S_M.$     (See \cite{Wong} for an exposition in which the dependence of the constants is made explicit, in the context of harmonic maps).  
\end{remark}
 
Once we have bounded $|F_A|$, we may replace $\epsilon-$regularity arguments with Moser iteration arguments. 
\begin{proposition}\label{Thm:Linf} 
Let $M$ be a Riemannian 4-manifold of bounded geometry. Let $\mathcal{E}\to M$ be a Hermitian vector bundle over $M$ with smooth finite action Yang-Mills connection  $A$. Then  there exists $ \tilde C>0$ depending only on $\delta(M), S_M,$ and $\|\mathcal{R}\|_{L^\infty} $ and the rank of $\mathcal{E}$ 
such that $\|F_A\|_{L^\infty}<\tilde C\|F_A\|_{L^2}$. Moreover,  for any choice of basepoint $o\in M,$ $|F_A(p)|\to 0$ as $\mathrm{dist}(p,o)\to\infty.$ 
\end{proposition} 
\begin{proof}From equation \eqref{epsregc}, we see that 
$\|F_A\|_{L^\infty}\leq \tilde C\|F_A\|_{L^2}$, with $\tilde C = C4^4\delta(M)^{-4}$.   For any 
$\eta$ less than the $\epsilon$ appearing in \eqref{epsreg}, let 
$K_\eta$ be a compact subset of $M$ such that $\int_{M\setminus K_\eta}|F_A|^2dv<\eta.$ Then for all $p$  of distance at least $R<\frac{\delta(M)}{2}$ from $K_\eta$, we have from \eqref{epsregc}
\begin{align} |F_A|^2(p)\leq\frac{C\eta}{R^4},
\end{align}
thus $F_A$ is indeed $L^\infty.$ 
The norm decay follows from selecting an exhausting sequence of compact sets $K_{\eta_k}$ with $\lim_{k\rightarrow\infty}\eta_k=0.$ 
\end{proof}


\section{Instanton Connections}\label{Sec:Instantons}

In the previous section we proved that, on a complete four-manifold of bounded geometry,  Yang-Mills connections with square integrable curvature have curvature vanishing at infinity. In this section we specialize to self-dual (SD) and anti-self-dual (ASD) connections with square integrable curvature on TN$_k$, and we impose the generic asymptotic holonomy assumption.  We first prove that generic asymptotic holonomy around the Taub-NUT fiber, as defined in Section~\ref{Sec:Holonomy}, implies that the asymptotic holonomy exists and its conjugacy class is the same for every direction. 
In Section~\ref{Sec:FirstBound} we show curvature decays at least as fast as $r^{-3/2}$. Section~\ref{Sec:Quadratic} sharpens this result to quadratic curvature decay.

\subsection{Holonomy} \label{Sec:Holonomy}
As in  (\ref{fiber}), the multi-Taub-NUT metric admits an isometric $S^1$ action with $k$ fixed points  $\{\nu_1,\ldots,\nu_k\}$, and the quotient of TN$_k$ by the $S^1$ action is $\IR^3$.  Let $\pi_k:\mathrm{TN}_k\rightarrow\mathbb{R}^3$ again denote the projection to this  quotient. We now consider the holonomy of the instanton 
around the $S^1$ fibers $\pi_k^{-1}(x)$ for $x\in \IR^3\setminus  \{ \nu_1,\ldots,\nu_k\}.$ 

Let $A$ be an instanton, let $p\in \pi_k^{-1}(x)$, and let $H_p$ denote the holonomy of $A$ around $\pi_k^{-1}(x)$ with base point $p$. Thus $H_p$ is the unitary transformation  $H_p:\mathcal{E}_p\to \mathcal{E}_p$ obtained by parallel translation around $\pi_k^{-1}(x)$ (in the direction $\partial_\tau$). Let  $\{e^{ 2\pi i\mu_a(x)}\}_a$ be the eigenvalues of $H_p.$  These eigenvalues depend on $x$ 
and not on  the choice of the point $p\in\pi_k^{-1}(x)$. 

The holonomy of a tensor product of bundles is the tensor product of the holonomies. In particular, if $\mathcal{E}$ has holonomy $H_{p}$ with eigenvalues $e^{2\pi i\mu_a}$, $a=1,\cdots, n,$ then 
$ad(\mathcal{E})\subset \mathcal{E}\otimes \mathcal{E}^*$ has holonomy $Ad(H_p)$ with eigenvalues $e^{2\pi i(\mu_a-\mu_b)}$, $a,b=1,\cdots, n.$ When $H_p$ belongs to the regular adjoint orbit, i.e. its eigenvalues are distinct (equivalently  $\mu_a-\mu_b\not\in\IZ$ for $a\not = b$), then the centralizer of $H_p$ in $ad(\mathcal{E}_p)$ is a Cartan subalgebra, $Z_p\subset ad(\mathcal{E}_p)$. This subspace is invariantly defined and is the fiber of a subbundle $Z$ of $ad(\mathcal{E})$. Equivalently, $Z$ is the holonomy eigen-subbundle of $Ad(H)$ with eigenvalue $1$. 
 
Consider a simple curve $c$ in $\IR^3\setminus  \{ \nu_1,\ldots,\nu_k\}$ that is unit speed in the Euclidean metric. Choose a  trivialization $S^1\times c$  of the $S^1$ bundle $\pi_k^{-1}(c)$.  Choose a frame for $\mathcal{E}$ along the circle $\pi_k^{-1}(c(0))$  such that the eigenvalues of the connection matrix $A( \frac{\p}{\p\tau})$ all have norm $\leq 1$.  Extend  this frame to the 2-dimensional cylinder $\pi_k^{-1}(c)$, by requiring it to be covariant constant along $ \{\tau\}\times c$ for each value $\tau$. 
In such a  frame, the connection matrix $A$ of the connection pulled back to  $\pi_k^{-1}(c)$ satisfies 
\begin{align}
A(\tau,s)(c'(s))=0, \text{   and thus   }  A\wedge A = 0.
\end{align}
 Hence, for all $v$ tangent to the cylinder,
\begin{align}\label{likeabelian}F(c'(s),v)=dA(c'(s),v),
\end{align}
and
\begin{equation}
\label{avar}A(\tau,s) = A(\tau,0) + \int_0^s F(c'(u),\cdot)(\tau,u)du.
\end{equation}
Since  $c$ is unit speed in $\IR^3$,  $|F(c',\partial_\tau)|\leq \sqrt{V}|F|,$  and
\begin{equation}\label{holoest0}
|A(\tau,s) - A(\tau,0)|\leq \int_0^s|F(\tau,c(t))|\sqrt{V}dt.
\end{equation}
In particular, 
\begin{equation}\label{holoest}
|A(\tau,s) - A(\tau,0)|\leq{|s|}\sup_{t\in [0,s]} \sqrt{V}|F(\tau,c(t))|.
\end{equation}
The  (unitary) solutions  of $(\frac{\p}{\p\tau} + A(\tau,s) (\frac{\p}{\p\tau}))\Pi(\tau,s) = 0,$ $\Pi(0,s) = \text{Id}$ satisfy 
\begin{align}|\Pi(\tau,s_1)-\Pi(\tau,s_2)|&\leq \|A(\cdot,s_1)-A(\cdot,s_2)\|_{L^1([0,\tau])}\nonumber\\
&+\int_0^\tau|A(u,s_1)||\Pi(u,s_1)-\Pi(u,s_2)|du.
\end{align}
 Gronwall's inequality (\cite[Lemma 2.7]{teschl}) then yields 
\begin{align}&|\Pi(\tau,s_1)-\Pi(\tau,s_2)| \leq \|A(\cdot,s_1)-A(\cdot,s_2)\|_{L^1([0,\tau])}\exp[\|A(\cdot,s_1)\|_{L^1([0,\tau])}]\nonumber\\
&
\leq e^{2\pi}\|A(\cdot,s_1)-A(\cdot,s_2)\|_{L^1([0,\tau])}\exp[\|A(\cdot,s_1)-A(\cdot,s_2)\|_{L^1([0,\tau])}]
\nonumber\\
&
\leq e^{2\pi}\|\sqrt{V}F(\cdot,c(\cdot))\|_{L^1([0,\tau]\times [s_1,s_2])}\exp[\|\sqrt{V}F(\cdot,c(\cdot))\|_{L^1([0,\tau]\times[s_1,s_2])}],
\end{align}
where we have used the bound on the eigenvalues $A(\frac{\p}{\p\tau})$ on $\pi_k^{-1}(0)$ in the second inequality. 

Since $H_{c(s)} = \Pi(2\pi,s)$, 
we have 
\begin{align}\label{holovar}&\|H_{c(s)}-H_{c(0)}\|_{sup}\nonumber\\
&\leq \lim_{N\to\infty}\sum_{j=1}^N\|H_{c(\frac{js}{N})}-H_{c(\frac{(j-1)s}{N})}\|_{sup}\nonumber\\
&\leq   \lim_{N\to\infty}\sum_{j=1}^Ne^{2\pi}\|\sqrt{V}F(\cdot,c(\cdot))\|_{L^1([0,2\pi]\times [\frac{(j-1)s}{N},\frac{js}{N}])}
e^{\|\sqrt{V}F(\cdot,c(\cdot))\|_{L^1([0,2\pi]\times[\frac{(j-1)s}{N},\frac{js}{N}])}}\nonumber\\
&= e^{2\pi}\|\sqrt{V}F(\cdot,c(\cdot))\|_{L^1([0,2\pi]\times [0,s])} .\end{align}

\begin{define}\label{genconnKappa}   Let 
 $0<\kappa \leq \frac{1}{6}$. 
Let $(\mathcal{E},{A})\to \mathrm{TN}_k$ denote a rank $n$ Hermitian bundle with connection. Let $\mathcal{U}\subset \mathbb{R}^3\setminus\{\nu_1,\cdots,\nu_k\}$ be any set.  We say that $A$ has {\em $\kappa-$generic holonomy} in $\mathcal{U}$ if on each circle fiber $\pi_k^{-1}(x)$ of $\pi_k^{-1}(\mathcal{U})$, the eigenvalues $\{e^{2\pi i\mu_a(x)}\}_a$ (with $\mu_a(x)$ defined mod $\IZ$)  of the holonomy $H(x)$, satisfy  
\begin{align}\label{kappaspace}
&\inf \left\{|\mu_a(x) - \mu_{a'}(x) -m| : m\in \IZ \text{ and }a\neq a'\right\}\geq\kappa,&  &\forall x\in\mathcal{U}. 
\end{align}
\end{define} 
\begin{define}\label{genconn}  
We say that a connection $A$ on TN$_k$ has {\em generic asymptotic holonomy} if there exists a ray $\rho:[0,\infty)\to \IR^3$ and  there exist $ t_0>0$, $0<\kappa \leq\frac{1}{6}$ so that $A$ has $\kappa-$generic holonomy in  $\rho([t_0,\infty)).$ 
\end{define}
In particular, for a connection with generic asymptotic holonomy the stabilizer of its holonomy  over any point of that given ray $\rho([t_0,\infty))$ is the Cartan subgroup of the gauge group.

Observe that the generic asymptotic holonomy condition involves a single ray in $\IR^3$ and imposes no a priori conditions that are global over the manifold's end. We will show in Proposition \ref{r3/2}, however, that this generic asymptotic holonomy condition involving a single ray implies that the holonomy is $\kappa$-generic on the complement of a compact set. Towards this goal we next determine lower bounds  on the distance between holonomy eigenvalues in terms of the curvature.
\begin{lemma}\label{lemma9}
Let $h(t)$ be a continuous family of unitary matrices such that the eigenvalues $\{e^{2\pi i\mu_a(0)}\}_j$ of $h(0)$ satisfy \eqref{kappaspace}, for some $\kappa \leq \frac{1}{6}$. Suppose that $\|h(0)-h(s)\|_{sup}<\epsilon\leq \sin \pi\kappa$. 
Then   the eigenvalues $\{e^{2\pi i\mu_a(s)}\}_a$  of $h(s)$  satisfy  
\begin{align}\inf \left\{|\mu_a(s)-\mu_{a'}(s) -q| : q\in \IZ \text{ and }a\not = a'\right\}
>\kappa-\epsilon. 
\end{align} 
\end{lemma}
\begin{proof}
The eigenvalue condition can also be expressed as 
 $$\frac{1}{2}\left| e^{2\pi i \mu_a(0)}- e^{2\pi i \mu_{a'}(0)} \right| =  \sin(\pi(\mu_a(0)-\mu_{a'}(0)))> \sin (\pi\kappa),$$ for $a\not = a'$. 
 The condition $\|h(0)-h(s)\|_{sup}<\epsilon$ implies $ |e^{2\pi i\mu_a(s)}-e^{2\pi i\mu_a(0)}|<\epsilon$, and 
therefore 
\begin{align*}
&\sin(\pi(\mu_a(s)-\mu_{j'}(s))) > \sin (\pi\kappa) - \epsilon\\
&\geq \sin (\pi(\kappa -\frac{1}{\pi}\arcsin(2\epsilon ))
>\kappa-\frac{1}{\pi}\arcsin(2\epsilon).
\end{align*}
The observation that   
$\frac{1}{\pi}\arcsin(2\epsilon ) <  \epsilon$ when $0<\epsilon<\frac{1}{2}$ completes the proof. 
\end{proof}
\begin{coro}\label{simplestav}If 
$c:[0,s]\to \IR^3\setminus\{\nu_1,\cdots,\nu_k\}$ is a unit speed curve with $H_{c(0)}$ satisfying the $\kappa$-generic condition   \eqref{kappaspace}, and 
$\int_0^s e^{2\pi} \sqrt{2\pi} \sqrt{\int_0^{2\pi} V|F(\tau,c(t))|^2d\tau }dt \leq \epsilon$, with $\epsilon\leq\frac{1}{3}$, then 
the eigenvalues $\{e^{2\pi i \mu_a(s)}\}_a$ of $H_{c(s)}$  satisfy  
\begin{align}\inf \left\{|\mu_a(s)-\mu_{a'}(s) -m| : m\in \IZ \text{ and }a\not = a'\right\}\geq\kappa-  \epsilon. 
\end{align} 
\end{coro} 
\begin{proof} By Lemma \ref{lemma9} and Equation \eqref{holovar} followed by Cauchy-Schwartz, 
\begin{align*}&\inf \left\{|\mu_a(s)-\mu_{a'}(s) -m| : m\in \IZ \text{ and }a\not = a'\right\}\nonumber\\
&\geq\kappa-   e^{2\pi}\|\sqrt{V}F(\cdot,c(\cdot))\|_{L^1([0,2\pi]\times [0,s])}\nonumber\\
&\geq\kappa-   \sqrt{2\pi}e^{2\pi} \int_0^s\sqrt{\int_0^{2\pi}V|F(\tau,c(t))|^2d\tau} dt. 
\end{align*}
\end{proof}

\subsection{First Bound}\label{Sec:FirstBound}
In this subsection, we exploit more features of the geometry of TN$_k$.   
The holonomy enters our analysis as an effective potential as seen in the following lemma. 
\begin{lemma}\label{poincholonomy}
Let $(B,\nabla)$ be a rank $n$ Hermitian vector bundle with Hermitian connection over a  circle $(S^1,d \tau)$ of length $2\pi  $.  Let $\{e^{2\pi i\mu_a}\}_{a=1}^n$  be the eigenvalues of the holonomy $H_p$, for one (and hence every)  base point $p\in S^1 $.  The eigenvalues of $i\nabla_{\p_\tau}$ on $L^2 $ sections of $B$ are 
\begin{align}\label{spec}
\mathrm{Spec}\,(i\nabla_{\p_\tau}) = \{\mu_a-m:m\in\IZ, 1\leq a\leq n\}.\end{align}
Assume further that the eigenvalues  satisfy :  $\inf_{a}\inf_{m\in\IZ}|\mu_a-m|\geq \kappa$.  
Then for every   smooth section $\sigma$ of $B$, 
\begin{align}\int_{S^1}|\nabla\sigma|^2d \tau\geq \kappa^2 \int_{S^1}|\sigma|^2d \tau.
\end{align}
\end{lemma}
\begin{proof}Let 
$\{v_1,\cdots,v_n\}$ be an orthonormal eigenbasis of $B_0$ for 
$H_0$ with eigenvalues $\exp(2\pi i \mu_a)$. Let $v_a(\tau)$ denote the covariant constant extension of $v_a$, for each $a$. Then 
$v_a(\tau+2\pi) = e^{2\pi i\mu_a}v_a(\tau).$ For any section $\sigma$ write 
$$\sigma(\tau) = \sum_a\sigma_a(\tau)e^{-i \tau\mu_a}v_a(\tau),$$
where the coordinate functions $\sigma_a$ are periodic of period $2\pi$. 
Then $\nabla= \partial_\tau$ in this frame, and 
$$\int_0^{2\pi}|\nabla\sigma(\tau)|^2d\tau =   \sum_a\int_0^{2\pi}|\sigma_a'(\tau) -i \mu_a\sigma_a|^2d\tau.$$
Fourier expanding $\sigma_a(\tau) = \sum_{k\in\IZ}\sigma_{ak}e^{i k\tau },$  
\begin{align}\label{thetaL}
\int_{S^1 }|\nabla\sigma(\tau)|^2d \tau& = \sum_a\sum_k \int_0^{2\pi}|\sigma_{ak}|^2(k-\mu_a)^2d\tau\nonumber\\
&\geq \kappa^2 \int_{S^1 }|\sigma(\tau)|^2d \tau.
\end{align}
Moreover, we see that $\{e^{i\tau(m-\mu_a)}v_a:m\in\IZ,1\leq a\leq n\}$ gives a complete eigenbasis for the $L^2$ sections of $B$ with the claimed eigenvalues. 
\end{proof}

 Consider now a set $\mathcal{U}$ with $\kappa-$generic holonomy, as in Definition \ref{genconnKappa}. We again let $Z$ denote the $1$ eigenspace of $Ad(H)$ in $ad(\mathcal{E})|_{\mathcal{U}}$.  Let $B$ denote its orthogonal complement. Then over $\mathcal{U}$, 
\begin{align}ad(\mathcal{E}) = Z\oplus B,
\end{align}
where $B$ is the subbundle on which the logarithms of the eigenvalues of the  holonomy have distance at least $\kappa$ from $\IZ$. This decomposition is preserved by $\nabla_{\frac{\p}{\p\tau}}.$ The generic holonomy hypothesis implies that $Z$ is a maximal abelian subalgebra of $ad(\mathcal{E})$.  This decomposition induces a corresponding decomposition of the curvature and its covariant derivatives  as
\begin{align}\label{zdecomp1}
F_A &= F_A^Z + F_A^B,& 
\nabla^qF_A &= (\nabla^qF_A)^Z + (\nabla^qF_A)^B,
\end{align}
where $F_A^Z$ is a two-form with coefficients in $Z$,  $F_A^B$ is a two-form with coefficients in $B,$ $(\nabla^qF_A)^Z$ is a tensor with coefficients in $Z$, and $(\nabla^qF_A)^B$ is a tensor with coefficients in $B.$ 
\begin{remark}\label{remarkr}We fix an origin for $\IR^3$ and will henceforth let $r$ denote both the radial function on $\IR^3$ and its lift to $M$=TN$_k$. 
\end{remark}
\begin{lemma}\label{verygeneric} Let 
$\mathcal{V}$ be either the spin bundle or 
$(T^*M)^{\otimes q}\otimes \Lambda^2T^*M$ for some $q\geq 0$.  For each $\kappa>0$ and each $\delta\in (0,\kappa)$, 
there exists a compact set $K_\delta\subset \IR^3$ such that $\forall x\in K_\delta^c$ 
 if  
$(\mathcal{E},A)$ has $\kappa-$generic holonomy at $x$,
 then 
for every section $\sigma$ of $\mathcal{V}\otimes ad(\mathcal{E})$, 
\begin{align}
\int_{\pi_k^{-1}(x)}|\nabla_{\p_\tau} \sigma|^2d\tau \geq   (\kappa-\delta)^2  \int_{\pi_k^{-1}(x)}| \sigma^B|^2d\tau.
\end{align}  
 Here $\sigma^B$ denotes the unitary projection of $\sigma$ onto $\mathcal{V}\otimes B$.
\end{lemma}
\begin{proof}
For any open set $W\subset\mathbb{R}^3\setminus\{\nu_1,\ldots\nu_k\},$ over which the $S^1$ bundle is trivial, 
the tangent bundle has a local coordinate frame defined in $\pi_k^{-1}(W)$.
 The coordinates can be chosen with $O(\frac{1}{r^2})$ Christoffel symbols. (For example, $(x,\tau)$ of \eqref{metric} and \eqref{varpi} are such coordinates, since the ($\IR^3-$)gradient of the harmonic function $V$ of Eq.~\eqref{metricpotential} is $O(r^{-2})$ on $\IR^3$.) In this ($S^1$-periodic) frame,  a tangent  vector field $v$  which is parallel  around the $S^1$ fiber satisfies an equation of the form $\frac{d v}{d\tau}= \mathcal{A} v ,$ where $\mathcal{A}= O(r^{-2})$. The fundamental theorem of calculus then implies the holonomy of the Levi-Civita connection is $I+O(r^{-2})$.  This in turn implies the holonomy of $\mathcal{V}$ is $I+O(r^{-2})$. Hence the logarithm of the eigenvalues of the holonomy of $\mathcal{V}\otimes B$ have distance at least $\xi= \kappa-O(r^{-2})$ from $\IZ$. Choose $K_\delta\supset B_R$, for the ball $B_R$ of  sufficiently large radius $R$ to ensure that on $B_R^c$, $\xi > \kappa -\delta$. Then  Lemma \ref{poincholonomy} yields the desired inequality. 
\end{proof}
 
In order to take advantage of holonomy information and the geometry of TN$_k$, it will be convenient for many estimates to first integrate quantities over the TN$_k$ circle fiber, and then compute on the $\IR^3$ base.  Given  sections $\psi_1,\psi_2$ of a Hermitian bundle over TN$_k$, define  functions 
$\Phi(\psi_1)$ and $Q(\psi_1,\psi_2)$ on $\mathbb{R}^3$ to be 
\begin{align}\label{defPhi}
\Phi(\psi_1)(x) &:=  \int_{\pi_k^{-1}(x)}|\psi_1|^2d\tau&
&\text{ and   }&
Q(\psi_1,\psi_2)(x) &:=  \int_{\pi_k^{-1}(x)}\langle \psi_1,\psi_2\rangle d\tau.
\end{align}
In order to obtain decay estimates for  $|F_A|^2$, it suffices to obtain decay estimates for $\Phi(F_A)$ as  the following lemma shows. 
\begin{lemma}\label{rote} 
Let $\psi$ be a smooth section of a Hermitian bundle over TN$_k$. Suppose that for some $W\geq 0$, 
\begin{align}\Delta |\psi|^2\leq W^2|\psi|^2.\end{align}
Let $\alpha\in [0,\infty).$ Then  $r^\alpha|\psi|^2$ is bounded if $r^\alpha\Phi(\psi)$ is bounded.
\end{lemma}
\begin{proof} For any $p\in M=\mathrm{TN_k},$    
by  Ineq.~\eqref{moserdone0} of the proof of Proposition \ref{moserflex},
\begin{align}\label{rote1}
\|\psi\|^2_{L^\infty\left(B_{\frac{\delta(M)}{4}}(p)\right)}\leq c_4 S_M^{-2}\frac{1}{\delta(M)^4} (1+W^2)^2\|\psi\|^2_{L^2\left(B_{\frac{\delta(M)}{2}}(p)\right)},
\end{align}
where $c_4$ depends only on dimension (and not on geometry).  We now estimate, outside of a compact set,  
\begin{align}
\|\psi\|^2_{L^2\left(B_{\frac{\delta(M)}{2}}(p)\right)}& \leq \|V\Phi(\psi)\|_{L^1\left(B_{\frac{\delta(M)}{2}}(\pi_k(p))\right)}\nonumber\\
&
\leq (r(p)-\frac{\delta}{2})^{-\alpha}\|Vr^{\alpha}\Phi(\psi)\|_{L^1\left(B_{\frac{\delta(M)}{2}}(\pi_k(p))\right)}.
\end{align}
Hence by \eqref{rote1}, $r^\alpha|\psi|$ is bounded. 
\end{proof} 
 
We next seek estimates  for $\Phi(F_A)$. Using
\begin{align}\label{laplace}\Delta_{TN}(\pi_k^*f)=\pi_k^*(V^{-1}\Delta_{R^3}f),
\end{align}
for $f:\IR^3\to \IR$. 
We compute:
\begin{align}\label{philap}V^{-1}\Delta_{\IR^3} \frac{1}{2}\Phi(\psi)  = -\Phi(\nabla\psi) +Q(\nabla^*\nabla\psi,\psi).
\end{align}
In particular, for Yang-Mills connections, 
\begin{align}\label{phiboch}V^{-1}\Delta_{\IR^3} \frac{1}{2}\Phi(F_A)  = -\Phi(\nabla F_A) +Q(\epsilon^i\epsilon_j^*{\cal R}_{ij} F_A+\epsilon^i\epsilon_j^*[F_{ij}, F_A],F_A).
\end{align}
 
Equation \eqref{phiboch}  implies for some computable $c_1, c_2 >0$, 
\begin{multline}\label{phiboch2}  
\frac{1}{2V}\Delta_{\IR^3}\Phi(F_A)(x)  \leq -\Phi(\nabla F_A)(x)
 + c_1\|{\cal R}\|_{L^\infty(\pi_k^{-1}(x))}\Phi(F_A)(x)\\
+ c_2\|F_A\|_{L^\infty(\pi_k^{-1}(x))}\Phi(F_A^B)(x).
\end{multline}
If we further  assume that $x\in K_\delta$ (introduced in Lemma \ref{verygeneric}) and that $A$ has $\kappa-$generic holonomy at $x$, then  Lemma \ref{verygeneric} implies 
\begin{align}\label{phibochk}  \frac{1}{2V}\Delta_{\IR^3}\Phi(F_A)(x)  \leq & -((\kappa-\delta)^2-c_2\|F_A\|_{L^\infty(\pi_k^{-1}(x))})\Phi( F_A^B)(x)\nonumber\\
 &+ c_1 \|{\cal R}\|_{L^\infty(\pi_k^{-1}(x))}\Phi(F_A)(x).
\end{align}

\begin{proposition}\label{r3/2} Let  $A$ be a Yang-Mills connection with  generic asymptotic holonomy. Then 
$r^{\frac{3}{2}}|F_A|$ is bounded, and for some $ \kappa>0 $, $A$ has $\kappa-$generic holonomy in the complement of a compact set.  
\end{proposition}
\begin{proof} 
By Proposition \ref{Thm:Linf}, given $\delta \in (0,\frac{1}{12})$, there exists a compact set 
$K_{2, \delta}\supset K_\delta$ so that $c_2\|F_A\|_{L^\infty(\pi_k^{-1}(x))} <\frac{\delta^2}{2}$, for $x\in K_{2, \delta}^c$.  When  $A$ has  $2\delta-$generic holonomy at  $x\in K_{2,\delta}^c$,  we use the cubic decay of  the Riemann curvature   $\cal R$ to bound the right-hand side of the Bochner estimate \eqref{phibochk} by
\begin{align}\label{phiboch22} \frac{1}{2V} \Delta_{\IR^3} \Phi(F_A)(x)  \leq \hat cr^{-3}\Phi(F_A)(x).
\end{align}
By Corollary \ref{moserapp}, if $A$ has $2\delta-$generic holonomy in $B_R(y)\subset K_{2,\delta}^c $, for some $R>1$ and some $y\in K_{2,\delta}^c$, then for some $c_3>0$ depending only  on  $\hat c$, 
\begin{align}\label{moserkap}\Phi(F_A)(x)&\leq \frac{c_3^2 }{  (R-|x-y|)^3}\int_{B_{R-|x-y|}(x)}\Phi(F_A) dv 
\nonumber\\
&\leq \frac{c_3^2 }{  (R-|x-y|)^3}\int_{B_{R }(y)}\Phi(F_A) dv,
\end{align}
 for any $x\in B_R(y).$  
Hence, for any unit vector $u$, 
\begin{align}\label{decline}\int_0^{R-1} \sqrt{V\Phi(F_A)(y+tu)}dt\leq 2c_3 (1-R^{-\frac{1}{2}})\|V \|_{L^\infty(B_{R }(y))}^{1/2}\| \Phi(F_A)\|_{L^1(B_{R }(y))}^{1/2}.\end{align}

Set
$$\epsilon :=  2e^{2\pi}c_3\sqrt{2\pi}  \|V \|_{L^\infty(B_{R }(y))}^{1/2}\| \Phi(F_A)\|_{L^1(B_{R }(y))}^{1/2}.$$
Corollary \ref{simplestav} and equation \eqref{decline} imply that if $A$ has  $\xi-$generic holonomy at $y$ and $2\delta $-generic holonomy in $B_{R}(y)$, then if $\epsilon<\xi$, $A$ has  
$(\xi- \epsilon)-$generic holonomy at $y+(R-1)u$.

In particular, if 
\begin{align}\label{condl1}    2 e^{2\pi} c_3 \sqrt{2\pi}   \|V \|_{L^\infty(B_{R }(y))}^{1/2}\| \Phi(F_A)\|_{L^1(B_{R }(y))}^{1/2}<\delta,
\end{align}
then $A$ has ($\xi-\delta)-$generic holonomy in $B_{R-1}(y)$. 
If in addition 
\begin{align}\label{condl2} e^{2\pi} \sqrt{2\pi}   \|\sqrt{V\Phi(F_A)}\|_{L^\infty(B_R(y))}< \delta ,\end{align}
then the resulting upper bound on $\int_{R-1}^R\sqrt{\Phi (F_A)(y+tu)}\,dt$ combined with  \eqref{condl1}  implies that $A$ has $(\xi-2\delta)-$generic holonomy in $B_{R}(y)$. By Proposition \ref{Thm:Linf}, there exists a compact set $K_{3,\delta}\supset K_{2,\delta}$ such that \eqref{condl1} and \eqref{condl2} hold whenever $B_R(y)\subset K_{3,\delta}^c$.   
Hence  
\begin{enumerate}
\item[(i)] if $B_{R}(y)\subset K_{3,\delta}^c$,
\item[(ii)] if $A$ has $2\delta-$generic holonomy in $B_{R}(y)$, and 
\item[(iii)] if $A$ has  $ \xi-$generic holonomy at $y$ with $\xi>4\delta$,
\end{enumerate}
then 
\begin{align}\label{Alow}
A \text{ has } ( \xi-2\delta)-\text{generic holonomy in } B_{R}(y). 
\end{align}
Given $y\in K_{3,\delta}^c$ so that $A$ has  $\xi-$generic holonomy at $y$ with $ \xi\geq 8\delta$,  set 
$$R_{\delta}(y):=\sup\{R:B_R(y)\subset K_{3,\delta}^c,\text{ and } A  \text{ has }2\delta\text{--generic holonomy in }B_R(y) \}.$$
Then (i)-(iii) are satisfied for $R= R_\delta(y)$ and therefore $A$ has ($\xi-2\delta)-$generic holonomy on $B_{R_{\delta}(y)}(y).$ Hence $A$ has ($\xi-3\delta)-$generic holonomy on $B_{R_{\delta}(y)+\beta}(y),$ for some $\beta>0$.   By the assumption of maximality of $R_\delta(y)$, we conclude that $\forall \beta>0,$  $B_{R_{\delta}+\beta}(y)\cap K_{3,\delta } \not =\emptyset$. We conclude that if $A$ is   $8\delta-$generic at $y$, then $A$ is $6\delta-$generic in every $B_R(y)\subset K_{3,\delta }^c$. 

By hypothesis, $A$ has  generic asymptotic holonomy, so there is  a ray $\rho$ in $\mathbb{R}^3$, which without loss of generality is assumed to start at the origin, and a closed ball $B_N(0)$, that contains all the  $\{\nu_j\}$,  such that, for some $\kappa>0$,  $A$ has $\kappa-$generic holonomy on  $\rho([0,\infty))\cap B^c_{N}(0)$.  We fix coordinates so that $\rho$ is the positive $z$ axis.  Set $\delta = \frac{\kappa}{16}$. By increasing $N$ if necessary, we may assume $K_{3,\delta}\subset B_N(0)$.    

Take $t_0\in\mathbb{R}$ such that $\rho(t)\in B^c_N(0)$ for $t\geq t_0$.  By the argument above (see \eqref{Alow}), $A$ has $\kappa-2\delta=\frac{7}{8}\kappa$-generic holonomy on  balls centered at points of $\rho$ if these balls are contained in $B^c_N(0)$.  These balls cover the open half-space $z>N$. By continuity of the holonomy, $A$ has  $\frac{7}{8}\kappa-$generic holonomy on the closed halfspace. In particular, $A$ has $\frac{7}{8}\kappa- $generic holonomy on the four rays  $(t,0,N)$, $(-t,0,N)$,   $(0,t,N)$, and $(0,-t,N)$, $t\geq 0$. Iterating the argument, we deduce that $A$ has $\frac{7}{8}\kappa-2\delta=\frac{6}{8}\kappa-$generic holonomy on the four closed half spaces $x\geq N$, $x\leq -N$, $y\geq N$, and $y\geq -N$. The first of these half spaces contains the ray $(N,0,-t)$, $t\geq 0$, on which $A$ therefore has $\frac{6}{8}\kappa-$generic holonomy. One more application of the above argument proves $z\leq -N$ has $\frac{6}{8}\kappa-2\delta=\frac{5}{8}\kappa-$generic holonomy. Hence, outside a cube of side length $2N$, $A$ has $\frac{5}{8}\kappa-$generic holonomy. 

We now apply \eqref{moserkap} to deduce $r^3\Phi(F_A)$ is bounded. Lemma \ref{rote} then implies $r^3|F_A|^2$ is bounded. 
\end{proof}

Next, we obtain estimates for $\nabla^kF_A$.
Given a one-form $w$, let $l(w)$ denote left tensor multiplication by $w$.  In an orthonormal frame we have 
\begin{align}[\nabla^*\nabla,\nabla] = -2l(e^b)(F_{ab}+{\cal R}_{ab})\nabla_a. 
\end{align}
Here we have used that both $F_A$ and ${\cal R}$ are Yang-Mills on TN$_k$. Hence 
\begin{align}\label{dog1}\nabla^*\nabla \nabla^{k}F_A&=    \nabla^{k}(\ep^i\ep^*_j{\cal R}_{ij}F_A - \ep^i\ep^*_j[F_{ij},F_A]) \nonumber\\
&   -2\sum_{m=0}^{k-1}\nabla^m[l(\ep^b)(F_{ab}+{\cal R}_{ab})\nabla_a\nabla^{k-1-m}F_A], 
\end{align} 
implying
\begin{align}\label{rabbit}
\frac{1}{2}\Delta| \nabla^{k}F_A|^2&= -|\nabla^{k+1}F_A|^2+\langle   \nabla^{k}(\ep^i\ep^*_j{\cal R}_{ij}F_A - \ep^i\ep^*_j[F_{ij},F_A]),\nabla^{k}F_A\rangle \\
&   -2\sum_{m=0}^{k-1}\langle\nabla^m[l(e^b)(F_{ab}+{\cal R}_{ab})\nabla_a\nabla^{k-1-m}F_A],\nabla^{k}F_A\rangle . \nonumber
\end{align} This yields the inequality for some $C_k>0$, 
\begin{align}\label{ksub1} \frac{1}{2}\Delta| \nabla^{k}F_A|^2 &\leq -|\nabla^{k+1}F_A|^2+C_k\sum_{m=0}^k |\nabla^{m} {\cal R}||\nabla^{k-m}F_A||\nabla^{k}F_A|\nonumber\\
&+C_k\sum_{m=0}^k |\nabla^{m}F_A||(\nabla^{k-m}F_A)^B||\nabla^{k}F_A |.
\end{align}
By \cite[Proposition A.2]{Minerbe10}, for all integers $q\geq 0$, $\exists C_{q,\mathcal{R}}<\infty$ such that
\begin{align}\|r^{3+q}\nabla^q\mathcal{R}\|_{L^\infty}<C_{q,\mathcal{R}}. 
\end{align}
We may similarly bound the covariant derivatives of $F_A$ as follows. 
 \begin{proposition}\label{killfourier} Let $A$ be a Yang-Mills connection with generic asymptotic holonomy, satisfying
$\|r^{\alpha}F_A\|_{L^2(M)}<\infty$, for some $\alpha\geq 0$.  
 Then for all integers  $ m\geq 0$,
\begin{align}\label{claimm}
\|r^{ \alpha+m}\nabla^mF_A\|_{L^2}<\infty, 
\end{align}
\begin{align}\label{claimm2} \|r^{\alpha+m+\frac{3}{2}}\nabla^mF_A\|_{L^\infty} &<\infty,& 
&\text{ and }& 
\|r^{ \alpha+m+1}(\nabla^mF_A)^B\|_{L^2}&<\infty.
\end{align}
\end{proposition}
\begin{proof} 
The proof is similar to \cite[Proposition A.2]{Minerbe10}. We induct on $m$. Suppose that we have shown \eqref{claimm} holds for all $m<q$ and if $q>1$  \eqref{claimm2} holds for all $m<q-1$. 

To obtain the pointwise bound of \eqref{claimm2} for $m=q-1$, we use  Corollary \ref{moserapp}. Define
$$f:=\sum_{m=0}^{q-1}\frac{1}{2}r^{2m}\Phi(\nabla^mF_A).$$
Then for $r$ large, using \eqref{rabbit}, we compute that there exist $\hat C_{q-1},\tilde C_{q-1}>0$ such that 
$$ \Delta f \leq  -\sum_{m=0}^{q-1}(\Phi(r^m\nabla^{m+1}F_A)- \frac{2m}{V}r^{2m-1}Q(\nabla_{\p_r}\nabla^mF_A,\nabla^mF_A))+\frac{\hat C_{q-1}}{r^{2}}f\leq \frac{\tilde C_{q-1}}{r^2}f,$$
via the Cauchy-Schwarz inequality. 
Hence, Corollary \ref{moserapp} yields for some constants $C_{q-1,1}, C_{q-1,2}, C_{q-1,3}>0$ and for all $x\in \IR^3$ with $R=\frac{|x|}{2}$ large, 
$$\|f\|_{L^\infty\left(B_{\frac{R}{2}}(x)\right)}\leq   \frac{C_{q-1,1}}{R^{3}}\int_{B_R(x)}fdv\leq  \frac{C_{q-1,2}}{R^{3+2\alpha}}\int_{B_R(x)}r^{2\alpha}fdv <\frac{C_{q-1,3}}{R^{3+2\alpha}}.$$
Hence 
\begin{align}\label{cold1}\|r^{2(q-1) +3+2\alpha}\Phi(\nabla^{q-1}F_A)\|_{L^\infty}<\infty.
\end{align}
 Lemma \ref{rote}  then implies $\|r^{q-1 +\frac{3}{2}+\alpha} \nabla^{q-1}F_A\|_{L^\infty}<\infty$,
and the pointwise estimate of \eqref{claimm2} for $m=q-1$ follows. 

Let $0\leq \phi\leq 1$ be a smooth function identically $1$ on $[0,1]$ and supported on $(-\infty,2)$. Let $\eta_{n,N}:= \min\{r^{q+\alpha },n^{q+ \alpha}\}\phi(\frac{r}{N})$.  
Then taking the inner product of  $\eta^2\nabla^{q-1}F_A$ with \eqref{dog1} (for $k=q-1$) yields :
\begin{align}
&\| \eta_{n,N}\nabla^{q}F_A\|^2_{L^2}  +2\langle  \eta_{n,N}\nabla^{q}F_A,d\eta_{n,N}\otimes \nabla^{q-1}F_A\rangle_{L^2}  \nonumber\\
&=  \Big\langle \nabla^{q-1}(\ep^i\ep^*_j{\cal R}_{ij}F_A - \ep^i\ep^*_j[F_{ij},F_A])\nonumber\\
  & -2\sum_{m=0}^{q-2}\nabla^m[l(e^b)(F_{ab}+{\cal R}_{ab})\nabla_a\nabla^{q-2-m}F_A,\eta_{n,N}^2\nabla^{q-1}F_A\Big\rangle_{L^2}.
\end{align}
Hence 
\begin{align}\label{dog2.1}
&\frac{1}{2}\| \eta_{n,N}\nabla^{q}F_A\|^2_{L^2}  \leq 2\||d\eta_{n,N}| \nabla^{q-1}F_A\|^2_{L^2} \nonumber\\
 &+C_{q-1}\sum_{m=0}^{q-1}\|r\eta_{n,N}|\nabla^{m} {\cal R}||\nabla^{q-1-m}F_A\|_{L^2}\|r^{-1}\eta_{n,N}\nabla^{q-1}F_A\|_{L^2}\nonumber\\
&+C_{q-1}\sum_{m=0}^{q-1} \|r\eta_{n,N}(|\nabla^{m}F_A||(\nabla^{q-1-m}F_A)^B\|_{L^2}\|r^{-1}\eta_{n,N}\nabla^{q-1}F_A\|_{L^2}\nonumber\\
&\leq 2\|(q+\alpha)|dr|r^{q+\alpha-1} \nabla^{q-1}F_A\|^2_{L^2} \nonumber\\
 &+C_{q-1}\sum_{m=0}^{q-1}C_{q,\mathcal{R}}\|r^{q+\alpha-2-m} \nabla^{q-1-m}F_A\|_{L^2}\|r^{q+\alpha-1}\nabla^{q-1}F_A\|_{L^2}\nonumber\\
&+C_{q-1}\sum_{m=0}^{q-1} \|r^{-m-\alpha-\frac{1}{2}}\eta_{n,N}(\nabla^{q-1-m}F_A)^B\|_{L^2} \|r^{q+\alpha-1}\nabla^{q-1}F_A\|_{L^2}.
 \end{align}
By Proposition \ref{r3/2}, for some $\kappa >0$, $\mathcal{E}$ has $\kappa-$generic holonomy in the complement of a compact set. Choosing $\delta = \frac{\kappa}{2}$ in Lemma \ref{poincholonomy}, there exists a compact set $K_\delta\subset \IR^3$ such that 
\begin{align}\label{poincap} \frac{1}{4}\| \eta_{n,N}\nabla^{q}F_A\|^2_{L^2} \geq \frac{1}{4}\|\sqrt{V}\eta_{n,N}\nabla_{\p_\tau}  \nabla^{q-1}F_A\|^2_{L^2} \geq \frac{\kappa^2}{16}\|\sqrt{V}\eta_{n,N} (\nabla^{q-1}F_A)^B\|^2_{L^2(K_\delta^c)}. 
\end{align}
Hence 
\begin{align}\label{dog2}
&\frac{1}{4}\| \eta_{n,N}\nabla^{q}F_A\|^2_{L^2}   \leq 2\|(q+\alpha)|dr|r^{q+\alpha-1} \nabla^{q-1}F_A\|^2_{L^2} \nonumber\\
 &+C_{q-1}\sum_{m=0}^{q-1}C_{q,\mathcal{R}}\|r^{q+\alpha-2-m}|\nabla^{q-1-m}F_A\|_{L^2}\|r^{q+\alpha-1}\nabla^{q-1}F_A\|_{L^2}\nonumber\\
&+C_{q-1}\sum_{m=1}^{q-1} \|r^{-m-\alpha-\frac{1}{2}}\eta_{n,N}(\nabla^{q-1-m}F_A)^B\|_{L^2}\|\|r^{q+\alpha-1}\nabla^{q-1}F_A\|_{L^2}\nonumber\\
&+\frac{\kappa^2}{16} \|r^{-\alpha-\frac{1}{2}}\eta_{n,N}(\nabla^{q-1}F_A)^B\|_{L^2(K_\delta^c)}^2-\frac{1}{4}\| \sqrt{V}\eta_{n,N}\nabla_{\p_\tau}(\nabla^{q-1}F_A)^B\|^2_{L^2}\nonumber\\
& +\frac{4C_{q-1}^2}{\kappa^2} \|r^{q+\alpha-1}\nabla^{q-1}F_A\|_{L^2}^2.
 \end{align}

The inductive hypothesis and \eqref{poincap} imply that every term on the right-hand side of \eqref{dog2}  is bounded above, uniformly for all $N$ and $n$.
 Hence, we may take the limit as $N,n\to\infty$ to deduce $r^{q+\alpha}\nabla^qF_A\in L^2,$  and by an additional application of \eqref{poincap}, $r^{q+\alpha}(\nabla^{q-1}F_A)^B\in L^2.$
The proposition now follows by induction. 

\end{proof}
\begin{coro}\label{cold2}
Let $A$ be a Yang-Mills connection with generic asymptotic holonomy, satisfying
$\|F_A\|_{L^2(M)}<\infty$.  
Then for all $j\in \IN$, and $r$ outside a compact set,
\begin{align}
|(\nabla^{ j} F_A)^B| < \frac{C(j,m)}{r^m},
\,\forall m\in \IN. 
\end{align}
\end{coro}
\begin{proof} For $r$ outside a compact set,  
$\Phi(r^{m+j+1+\frac{3}{2}} (\nabla^{ m+j+1} F_A)^B)$ is bounded by \eqref{cold1} (for $\alpha = 0$). Applying  Lemma \ref{poincholonomy} $m$ times yields $\Phi(r^{m+j +\frac{5}{2}} \nabla_{\frac{\p}{\p\tau}}^{j+1} F_A^B)$ is bounded. By the Sobolev embedding theorem for the circle, $r^{m+j +\frac{5}{2}} |(\nabla^{j} F_A)^B|$ is bounded. 
\end{proof}

\subsection{Quadratic Curvature Decay}\label{Sec:Quadratic} 
In this subsection we sharpen our curvature decay to quadratic decay. We first recall the Kato-Yau inequality for closed self-dual and anti-self-dual forms and a Hardy inequality for TN$_k$. 
\begin{lemma}[{{Kato-Yau inequality \cite{IN,Calderbank}}}]
Let $h$ be a closed anti-self-dual form, over a four-manifold, with coefficients in a Hermitian bundle, then for any unit vector $u$ 
\begin{equation}\label{sdkato}\frac{3}{2}|\nabla_uh|^2\leq |\nabla h|^2.\end{equation}
\end{lemma}
\begin{proof}
 Let $\{e_1,e_2,e_3,e_4\}$ be an oriented orthonormal frame. The bundle-valued form $h$ is closed, therefore $d_Ah=0$ and this Bianchi identity, combined with anti-self-duality, imply 
\begin{align*}
-h_{12;1}+ h_{13;4}+h_{41;3} &= 0,& 
h_{12;2}+ h_{13;3}+h_{14;4} &= 0,\\
h_{12;3}+ h_{31;2}-h_{14;1} &= 0,&
h_{12;4}- h_{31;1}+h_{41;2} &= 0.
\end{align*}
Hence 
$$|h_{12;1}|^2 + |h_{13;1}|^2 + |h_{14;1}|^2 \leq 2(|h_{13;4}|^2+|h_{41;3}|^2 + |h_{12;4}|^2+|h_{41;2}|^2+|h_{12;3}|^2+ |h_{31;2}|^2).$$
Hence 
$$\frac{3}{2}|\nabla_1h|^2\leq |\nabla h|^2.$$
Choose $e_1 = u$, and the lemma follows.  
\end{proof}
Let, as usual, $H^2_1(\mathrm{TN}_k)=\{f\in L^2(\mathrm{TN}_k):df\in L^2(\mathrm{TN}_k)\}.$  Similarly define $H^2_1(\mathrm{TN}_k,W)$ for any Hermitian vector bundle equipped with a connection. We will simply write $H^2_1$ when the bundle and connection are clear. 
 \begin{lemma}[Hardy inequality for TN$_k$]
For all $ f\in H^2_1(\mathrm{TN}_k),$ 
\begin{equation}\label{hardy2}
\frac{1}{4}\left\|\frac{f}{r\sqrt{V}}\right\|^2\leq \left\|\frac{\p f}{\p r}\right\|^2.
\end{equation}
\end{lemma}
\begin{proof}
The proof is the same as for $\IR^3$ and extends to any ALF space by a minor modification of the proof of Proposition 3.7 in \cite{DegStern}. In the notation of \eqref{defPhi}, we write 
$$\int_M\left|\frac{f}{r\sqrt{V}}\right|^2dv = \int_{\IR^3}\Phi\left(\frac{f}{r}\right)dx.$$
Hardy's inequality for $\IR^3$ followed by Kato's inequality for $L^2(S^1)$ gives 
$$\frac{1}{4}\int_{\IR^3}\Phi\left(\frac{f}{r}\right)dx\leq \int_{\IR^3}\left|\frac{\p\sqrt{\Phi(f)}}{\p r}\right|^2dx\leq \int_{\IR^3} \Phi\left(\frac{\p f}{\p r}\right)Vdx =\left\|\frac{\p f}{\p r}\right\|^2.$$
\end{proof}

\begin{theo}\label{lucky}
Let $A$ be an instanton with generic asymptotic holonomy. Then $ \|r^2  F_A \|_{L^\infty}<\infty. $ 
\end{theo}
\begin{proof}
We begin  with an integrated Bochner formula: 
\begin{equation}\label{again}0 = \|\nabla (\eta F_A) \|^2 -\||d\eta|F_A\|^2    - \langle \ep^i\ep^*_j{\cal R}_{ij}\eta F,\eta F\rangle
- \langle [F_{ij},\eta F_{jk}],\eta F_{ik}\rangle. \end{equation}
This equation  holds for all $\eta$ with  $\eta F_A\in H^2_1.$   
Choosing  $\eta =\eta_n(r) = r_n^{p}r^{\frac{1}{2}}$ with $p\leq 1$ yields 
\begin{equation}\label{logagain}  
\|\nabla (r_n^{p}r^{\frac{1}{2}} F_A) \|^2 =\left\||d(r_n^{p}r^{\frac{1}{2}})|F_A\right\|^2  + \langle \ep^i\ep^*_j{\cal R}_{ij}r_n^{2p}r  F,  F\rangle
+ \langle [F_{ij},r_n^{2p}r F_{jk}],  F_{ik}\rangle. 
\end{equation}
Write 
$$ \|\nabla  (\eta F_A) \|^2 =  \|\nabla^0 (\eta F_A) \|^2 +  \|\nabla_{\hat{e}_1} (\eta F_A) \|^2,$$
where $\hat{e}_1$ is a unit vector in the radial direction and $\nabla^0$ denotes the summand of the covariant derivative in the directions orthogonal to $\hat{e}_1$. The Hardy inequality (\ref{hardy2}) and the Kato inequality give  
\begin{equation}\label{hardy} \|\nabla_{\hat{e}_1} (\eta F_A) \|^2\geq \frac{1}{4} \left\|\frac{\eta F_A}{r\sqrt{V}} \right\|^2    .\end{equation}
 By Proposition \ref{r3/2} and  Corollary \ref{cold2},  we have 
$$
\langle [F_{ij},\eta_n F_{jk}],\eta_n F_{ik}\rangle(x)\leq 
C_1|F_A|\left|F^B\right|^2 \eta_n^2 \in L^1,
$$ 
with $L^1$ bound independent of $n$ since $|F^B|= O(r^{-m}),\forall m$.  Combining this estimate and (\ref{hardy}) with (\ref{logagain}) gives 
\begin{equation}\label{hreduct}\frac{1}{4} \|\frac{r_n^{p}}{\sqrt{rV}} F_A  \|^2+\|r_n^{p}r^{\frac{1}{2}}\nabla^0   F_A \|^2 \leq   \||d(r_n^{p}r^{\frac{1}{2}})|F_A\|^2  + C_2. \end{equation}
Here we have used the cubic decay of the Riemann curvature ${\cal R}_{ij}$.
 
Now we use the Bianchi identity to estimate $\|r_n^{p}r^{\frac{1}{2}}\nabla^0   F_A \|^2$ from below. 
{Let $\hat{e}_1,\hat{e}_2,\hat{e}_3,\hat{e}_4$ be an oriented orthonormal frame around a point $x\in \mathrm{TN}_k$, with $\hat{e}_1$ the unit radial vector field. The  Bianchi identity implies }
\begin{equation}\label{bianchi}0 = \langle  F_{12;3}- F_{13;2} - F_{14;1},F_{14}\rangle  . \end{equation}
Summing over cyclic permutations and applying Cauchy-Schwartz yields
\begin{equation}\label{bianchicyc}|\nabla^0F_A||F_{A}| \geq  -\frac{1}{2}\hat e_1|F_A|^2  . \end{equation}
  Multiply this inequality by $r_n^{2p}V^{-\frac{1}{2}}$ and integrate to obtain, for some  $c>0$ independent of $n$ and $p$,
\begin{equation}\label{nocrosspn} \sqrt{2}\|r_n^{p}r^{1/2}  \nabla^0 F_{A}\|  \| r_n^{p}r^{-1/2}V^{-1/2}F_{A}\|  \geq  
 \int |F_{A}|^2 (p\, t_{n,p} +1)\frac{r_n^{2p }}{rV}dv - c, \end{equation}
where, for $\chi_n$ the characteristic function of $\{x:r(x)\leq n\},$ 
\begin{align}\label{tnpdef}t_{n,p} = \frac{\int|F_A|^2\chi_nr_n^{2p }r^{-1}V^{-1}dv}{\int|F_A|^2r_n^{2p }r^{-1}V^{-1}dv}.
\end{align}
In integrating \eqref{bianchicyc}, we have used 
\begin{align}\label{intbyparts}
\mathrm{div} \left(\frac{r_n^{2p}}{2\sqrt{V}}|F_{1m}|^2\hat{e}_1\right) =   r_n^{2p}V^{-\frac{1}{2}} \frac{1}{2}\hat e_1|F_A|^2
&+ (p\chi_n+1)\frac{r_n^{2p }|F_{1m}|^2}{rV}\\
    &+ O\left(|F_{1m}|^2  r_n^{2p }r^{-2}\right).\nonumber
\end{align}

For fixed $n$, the integral of the divergence term in \eqref{intbyparts} vanishes since $F_A\in H^2_1$ (and can be seen even more readily from the bound for $r^{\frac{3}{2}}F_A$ obtained in Proposition \ref{r3/2}). 
Squaring both sides yields 
\begin{equation}\label{bianchiestpn}  
\| r_n^{p}r^{1/2} \nabla^0 F_{A}\|^2     \geq  
 \frac{(p\, t_{n,p}  +1)^2}{2  }\int |F_{A}|^2  r_n^{2p }r^{-1}V^{-1}dv - c_2. 
 \end{equation}
Inserting (\ref{bianchiestpn}) into (\ref{hreduct}) gives, for some  $c_3>0$ independent of $n$ and $p$,
\begin{equation}\label{hreductp}  
[1-p^2  + p^2 (t_{n,p} -  1)^2]\int |F_{A}|^2  r_n^{2p }r^{-1} V^{-1} dv \leq   c_3. 
\end{equation}
Hence, for all $p<1$, we may take the limit as $n\to \infty$ in (\ref{hreductp}) to deduce 
\begin{equation}\label{unifp}
  \int |F_{A}|^2  r^{2p-1}V^{-1} dv \leq   \frac{c_3}{1-p}. 
\end{equation}
 Let $p = 1 - \frac{1}{N}$ in Equation (\ref{unifp}), then
\begin{equation}\label{loggrow}\int_{r\leq e^N} |F_{A}|^2r   V^{-1} dv \leq   c_3e^{2}N.\end{equation}
Hence
\begin{equation}\label{Nloggrow}\frac{1}{N}\sum_{m=0}^{N-1}\int\limits_{e^m\leq r\leq e^{m+1}} |F_{A}|^2r   V^{-1} dv \leq   c_3e^{2} .\end{equation}
Thus, the average value of 
$\int_{e^m\leq r\leq e^{m+1}} |F_{A}|^2r V^{-1}   dv$ is less than $c_3e^{2}$. Set 
$$J:=\left\{m \in\IN:  \int_{e^m\leq r\leq e^{m+1}} |F_{A}|^2r  V^{-1}  dv\leq 2c_3e^{2}\right\}.$$ 
Since \eqref{Nloggrow} holds for all $N$ large, $J$ is an infinite set. For each $m\in J$, $m$ sufficiently large,  inequality \eqref{moserkap} implies for each $(x,y)\in \IR^3\times \IR^3$  with $|y| = \frac{e^{m+1}+e^m}{2}$ and 
$|x-y| < \frac{e^{m+1}-e^m}{2}=:R$ that 
\begin{align}\label{Tbound}\Phi( F_A )(x)&\leq \frac{C}{(\frac{e^{m+1}-e^m}{2}- |x-y|)^3}\int_{\pi_k^{-1}(B_R(y))} |F_{A}|^2 V^{-1}dv\nonumber\\
&\leq \frac{Ce^{-m}}{(\frac{e^{m+1}-e^m}{2}- |x-y|)^3} 2c_3e^2.
\end{align}  
In particular, there exists $\tilde C>0$ independent of  $m\in J$ such that for all $x$ with $|x|\in [\frac{e^{(m+1)}+3e^m}{4}, \frac{3e^{(m+1)}+e^m}{4}]$, $\Phi( F_A )(x) \leq \frac{\tilde C }{|x|^4} .$
Lemma \ref{rote} now implies 
\begin{align}\label{cold3}|F_A|^2(x)\leq \frac{\tilde C_3 }{|x|^4},
\end{align} 
$\forall x, m$ with $|x|\in [\frac{e^{(m+1)}+3e^m}{4}+\frac{\delta(M)}{2}, \frac{3e^{(m+1)}+e^m}{4}-\frac{\delta(M)}{2}] $, and 
$m\in J$.
Because (\ref{Nloggrow}) holds for all large $N$,  there is a uniform pointwise upper bound for $r^2|F_A|$ on infinitely many spherical annuli.  If $r^{1+p}|F_A|$ is not uniformly bounded for all  $p\leq 1,$  then it must achieve local maxima between some of these annuli, for some $p$. At  a critical point of $r^{1+p}|F_A|$ we have 
\begin{equation}\label{crit}0 = (1+p)|F_A|  + r\frac{\p |F_A|}{\p r}.\end{equation}
By the  the anti-self-dual Kato-Yau inequality (\ref{sdkato}) and the Kato inequality, \eqref{crit} implies that at a critical point
\begin{equation}\label{crit2}|\nabla F_A|^2  \geq  \frac{3}{2V}|\nabla_{\frac{\p}{\p r}} F_A|^2 \geq 
  \frac{3(1+p)^2|F_A|^2}{2Vr^2}  .\end{equation} 
At  a maximum point of $r^{1+p}|F_A|$ we have 
\begin{align}\label{maxprinc}
0\geq &-\frac{1}{2}\Delta(|F_A|^2r^{2+2p}) = r^{2+2p}|\nabla F_A|^2 +\frac{(2p^2+5p+3)}{V}r^{2p}|F_A|^2\nonumber\\
 &+ \frac{(2+2p)r^{2p+1}}{V}\frac{\p|F_A|^2}{\p r} + O(r^{2p-1}|F_A|^2) - r^{2+2p}\langle [F_{ij},F_{jk}],F_{ik}\rangle\nonumber\\
&= r^{2+2p}|\nabla F_A|^2 - \frac{(2p^2+3p+1)r^{2p}}{V}  |F_A|^2 + O(r^{2p-1}|F_A|^2) \nonumber\\
&\geq  r^{ 2p}\frac{ (1-p^2)|F_A|^2}{2V }   + O(r^{2p-1}|F_A|^2) ,
\end{align}
by (\ref{crit2}). Dividing by $r^{2p}$ we find that  $\exists \lambda>0$ such that 
\begin{equation}\label{maxprinc2}  
(1-p)|  F_A|^2 \leq  \lambda r^{-1}|F_A|^2. 
\end{equation}
Hence, at a local maximum, we must have $1-p \leq  \lambda r^{-1}$. 
Let $1<N_0\in J$ such that $ \lambda r^{-1}< \frac{1}{\ln(r)}, \forall r>e^{N_0}$.
Let  $N-1\in J$, with $N-1>N_0$. 
For $p = 1-\frac{1}{N}$, $(1-p) \geq \frac{1}{\ln(r)}$ when $r\leq e^{N}$. 
Therefore $r^{2+2p} |  F_A|^2$ has no local maximum  for $e^{N_0}\leq r\leq e^{N}$.  
Hence  
\begin{multline} 
\| r^{1+p} F_A\|_{L^\infty\left(\pi_k^{-1}(\{x:e^{N_0}\leq |x|\leq \frac{e^N+e^{N+1}}{2}\})\right)}\\
\leq \max\left\{ \| r^{1+p} F_A\|_{L^\infty\left(\pi_k^{-1}(|x|=e^{N_0}\right)}, \sqrt{\tilde C_3}\right\}.
\end{multline}
For $r\leq e^N,$ $r^{1+p}\geq r^2e^{-1}.$ Hence we deduce 
\begin{align} &\| r^{2}  F_A \|_{L^\infty(M)}\leq e\max\left\{ \| r^{2} F_A\|_{L^\infty(\pi_k^{-1}(|x|=e^{N_0})}, \sqrt{\tilde C_3}\right\},
\end{align}
and the theorem follows. 
\end{proof}

\section{Asymptotic Form of an Instanton}\label{Sec:Asymp} 

Consider a finite energy Yang-Mills connection $A$ with generic asymptotic holonomy. By Proposition \ref{r3/2}, there exists a compact $K\subset \IR^3$,  with $  \{\nu_\sigma\}_{\sigma=1}^k\subset K$, such that $\forall x\in K^c$, the eigenvalues  $\{e^{ 2\pi i\mu_a(x)}\}_a$  of $H_p$, with $p\in \pi_k^{-1}(x)$ are distinct. 
 Let $U\subset\IR^3\setminus K$ be an open contractible set, and let $(x,\tau)\in U\times [0,2\pi  )$ be local coordinates for $\pi_k^{-1}(U)$.
Let $\{v_a\}_a$ be a smooth unitary holonomy eigenframe of $\mathcal{E}$ over the section of the circle bundle defined by
$$\Sigma:= \{(x,0):x\in U\}.$$ 
We extend this eigenframe by parallel translation (over one period in positive $\tau$ direction) to a frame over $\pi_k^{-1}(U)$, albeit discontinuous at $\Sigma.$ 
Define
$$\{e^{-i \tau\mu_a(x)}v_a(\tau)=:w_a(\tau)\}_a,$$ 
thus obtaining a smooth and, of course, continuous  frame of $\mathcal{E}$ over $\pi_k^{-1}(U).$ In the frame $\{w_a\}$, the connection matrix, $A\left(\frac{\partial}{\partial\tau}\right)$, is diagonal, with 
\begin{align}\label{Eq:parmu}
A\left(\frac{\partial}{\partial\tau}\right)=-i\, \mathrm{diag}(\mu_a).
\end{align} 
Note, that this choice of the frame depends on the choice of the log branch in defining $\mu_a$ for each holonomy eigenvalue $e^{2\pi i\mu_a}.$ We reconsider this choice at the end of this section.

It is  useful to observe that for $a\not = b$, 
\begin{align}\label{asrequested}\langle F_A(\p_\tau,\p_j)w_a,w_b\rangle=\langle F_A^B(\p_\tau,\p_j)w_a,w_b\rangle = O(r^{-m}),
\end{align}
for all $m\in \IN$, by Corollary \ref{cold2}.

Define the projection operator acting on $L^2(\pi_k^{-1}(x),\Lambda^2T^*M\otimes ad(\mathcal{E}))$ (for $|x|$ large) by 
\begin{align}\label{pinought}
\Pi_0:= \frac{1}{2\pi i}\oint_{C}(z-i\nabla_{\frac{\p}{\p \tau}})^{-1}dz,\end{align}
where $C$ is a small circle around $0$ in $\mathbb{C}$ with radius $\rho<\frac{1}{2}\min\{  |\mu_a-\mu_{a'}+m| : m\in\IZ \text{ and } a\neq a'\}$. For sufficiently large $|x|$,  $\Pi_0$ is the projection onto the $O(r^{-2})$ eigenspace of the self-adjoint operator $i\nabla_{\frac{\p}{\p \tau}}$. To see this, we recall that as noted in Lemma \ref{verygeneric}, the holonomy of $\Lambda^*T^*M$  around the circle fibers is $I+ {O}(r^{-2})$. On the other hand, the holonomy eigenvalues on $ad(\mathcal{E})$ are $\{e^{2\pi i(\mu_a-\mu_b)}, 1\leq a,b\leq n\} .$ The tensor product $\Lambda^*T^*M\otimes ad(\mathcal{E})$ then has  holonomy eigenvalues $\{e^{2\pi i(\mu_a-\mu_b+O(r^{-2}))}, 1\leq a,b\leq n \}.$ Hence  by  Lemma \ref{poincholonomy}, the eigenvalues of $i\nabla_{\p_\tau}$ are $\{\mu_a-\mu_b+m+O(r^{-2}):m\in\IZ, 1\leq a,b\leq n\}.$ For $r$ sufficiently large, we see that only $O(r^{-2})$ eigenvalues are contained within the contour $C$.  
Set 
$$\Pi_1 = I-\Pi_0.$$
The subbundle  $\Lambda^*T^*M\otimes B$ is holonomy invariant with holonomy eigenvalues 
$\{e^{2\pi i(\mu_a-\mu_b+O(r^{-2}))}, a\not = b\} .$ Thus all the eigenvalues of $i\nabla_{\p_\tau}$ on this summand have norm larger than $\rho$ for $r$ large. Hence 
\begin{align}\label{analog}\Pi_1F_A = F_A^B + \Pi_1F_A^Z.\end{align}
\begin{proposition}\label{polydecay}
\begin{align}
|\nabla_{\p_\tau}\Pi_1F_A|+|\Pi_1F_A| < \frac{C}{r^m}, 
\forall m\in \IN
\end{align}
\end{proposition}
\begin{proof}  The proof is essentially the same as that of Corollary \ref{cold2}.   
\end{proof}
 
\begin{coro}\label{mpolydecay}
\begin{align}
|\nabla_{\p_\tau}F_A|  = O(r^{-4}).
\end{align}
\end{coro}
\begin{proof}  By construction, 
$|\nabla_{\p_\tau}\Pi_0F_A | =O(r^{-2}|F_A|)=O(r^{-4}).$ Hence \\$|\nabla_{\p_\tau}F_A|\leq |\nabla_{\p_\tau}\Pi_0F_A |+|\nabla_{\p_\tau}\Pi_1F_A |= O(r^{-4}).$
\end{proof}

Now we consider the variation of the $\mu_a$ as a function of $x\in\mathbb{R}^3$.  Let 
$ f_k:=\frac{\p}{\p x^k}-\omega(\frac{\p}{\p x^k})\frac{\p}{\p \tau},$  $1\leq k\leq 3,$ with $\omega$ of Eq.~\eqref{varpi}, denote the horizontal lifts of the coordinate vector fields of $\IR^3$. Using  $\nabla_{f_k}\nabla_{\p_\tau}w_a = F(f_k,\p_\tau)w_a + \nabla_{\p_\tau}\nabla_{f_k}w_a,$ \eqref{Eq:parmu}, and integration by parts, we have
\begin{align}\label{c1}
&\p_k\, \mu_a(x) =  \p_k\frac{i}{2\pi} \int_{\pi_k^{-1}(x)}\langle \nabla_{{\p_\tau}} w_a,w_a\rangle d\tau \nonumber\\
&=  \frac{i}{2\pi} \int_{\pi_k^{-1}(x)}[\langle F(f_k,\p_\tau) w_a,w_a\rangle+\langle\nabla_{{\p_\tau}} \nabla_{f_k} w_a,w_a\rangle +\langle \nabla_{{\p_\tau}} w_a,\nabla_{f_k}w_a\rangle]d\tau \nonumber\\
&=  \frac{i}{2\pi} \int_{\pi_k^{-1}(x)}[\langle F(f_k,\p_\tau) w_a,w_a\rangle+i\mu_a\langle \nabla_{f_k} w_a,w_a\rangle +i\mu_a\langle  w_a,\nabla_{f_k}w_a\rangle]d\tau \nonumber\\
&=  \frac{i}{2\pi} \int_{\pi_k^{-1}(x)} \langle F(f_k,\p_\tau) w_a,w_a\rangle d\tau .
\end{align}
Hence \begin{equation}\label{eibound}|d\mu_a|<\frac{c}{r^2}.\end{equation} 
The final integrand in \eqref{c1} is almost constant in $\tau$ since 
\begin{align}&\frac{\p}{\p\tau}\langle F_A(\p_\tau,\p_j)w_a,w_a\rangle 
 = \langle (\nabla_{\frac{\p}{\p\tau}}F_A)(\p_\tau,\p_j)w_a,w_a\rangle\nonumber\\
&
+\langle F(\nabla_{\p_\tau}\p_\tau,\p_j)w_a,w_a\rangle+\langle F(\p_\tau,\nabla_{\p_\tau}\p_j)w_a,w_a\rangle = O(r^{-4}),\end{align}
by Corollary \ref{mpolydecay}.

Hence     
\begin{align}\label{monopole}id\mu_a = \langle F(\p_\tau,\cdot)w_a,w_a\rangle + O(r^{-4}).
\end{align}  
 We also compute 
\begin{multline}\label{deltaest1}
2\pi i \Delta_{\IR^3}\mu_a(x) =  \sum_{k=1}^3\p_k\int_{\pi_k^{-1}(x)} \langle F(f_k,\p_\tau) w_a,w_a\rangle d\tau  \\
=  \sum_{k=1}^3\left(  \int_{\pi_k^{-1}(x)} \langle (\nabla_{f_k}F)(f_k,\p_\tau) w_a+F(\nabla_{f_k}f_k,\p_\tau) w_a+F(f_k,\nabla_{f_k}\p_\tau) w_a,w_a\rangle d\tau\right.\\
+\left. \int_{\pi_k^{-1}(x)}[\langle F(f_k,\p_\tau) \nabla_{f_k}w_a,w_a\rangle+\langle F(f_k,\p_\tau) w_a,\nabla_{f_k}w_a\rangle] d\tau\right). 
\end{multline}
The Yang-Mills equation implies the vanishing of $\sum_{k=1}^3(\nabla_{f_k}F)(f_k,\p_\tau)$. Since $|w_a|^2 = 1$, we can write $\nabla_{f_k}w_a = \gamma_{ka}^bw_b$ with $\gamma_{ka}^a\in i\IR.$  
Using Prop.~\ref{polydecay} or \eqref{asrequested}, for  $m=6$ gives
\begin{align}
&\int_{\pi_k^{-1}(x)}[\langle F(f_k,\p_\tau) \nabla_{f_k}w_a,w_a\rangle+\langle F(f_k,\p_\tau) w_a,\nabla_{f_k}w_a\rangle] d\tau\nonumber\\
&= \int_{\pi_k^{-1}(x)}[\langle(\Pi_0F)(\p_k,\p_\tau) \gamma_{ka}^aw_a,w_a\rangle+\langle (\Pi_0F)(\p_k,\p_\tau) w_a,\gamma_{ka}^aw_a\rangle] d\tau+O(r^{-6})\nonumber\\
&=  O(r^{-6}).
\end{align}

Computing covariant derivatives in the remaining two terms in \eqref{deltaest1} we conclude
\begin{align}\label{deltaest2}
&\Delta_{\IR^3}\mu_a(x) =    \frac{i}{2\pi} \int_{\pi_k^{-1}(x)} \langle F( f_k 
    ,\frac{V_k}{V}\p_\tau+\frac{1}{2}V^{-2}(\ast dV)_{jk} f_j ) w_a,w_a\rangle  d\tau +O(r^{-6}). 
\end{align}
Rotating coordinates (preserving orientation)  at $x$ so that  $\nabla V$ is in the direction opposite to $\p_1$, we  rewrite \eqref{deltaest2} as 
\begin{align}\label{deltaest3}
&-\Delta_{\IR^3}\mu_a(x) =    \frac{-i}{2\pi} \int_{\pi_k^{-1}(x)} \langle\frac{V_1}{V}[F( f_1 
    , \p_\tau )-F( f_2 
    , V^{-1 } f_3 )] w_a,w_a\rangle  d\tau +O(r^{-6}). 
\end{align}
 If the connection is not only Yang-Mills, but is self-dual, then this yields $\Delta_{\IR^3}\mu_a = O(r^{-6})$. Since $V_1$ is $O(r^{-2})$, we have in the anti-self-dual case :  
\begin{align}\label{deltaestasd}
& \Delta_{\IR^3}\mu_a(x) =    \frac{ i}{\pi} \int_{\pi_k^{-1}(x)} \langle(\frac{V_1}{V}F( f_1 
    , \p_\tau ) w_a,w_a\rangle  d\tau +O(r^{-6})=O(r^{-4}). 
\end{align}

We next use the bounds \eqref{deltaest2} and \eqref{deltaestasd} to obtain information about the asymptotic behavior of $\mu_a$. 
Let $x\in\IR^3$, with $|x|= 2R$. Let $\eta$ satisfy $\eta(s) = 1$ for $s\leq \frac{1}{2}$, $\eta(s) = 0$ for $s\geq 1$,
 and $|\eta'|\leq 4.$ Set 
\begin{align}\label{etaR}\eta_R(y) := \eta\left(\frac{|y|}{R}\right).
\end{align} Let $u$ be a $C^2$ function on $\mathbb{R}^3$ satisfying 
$|\Delta u| = O(r^{-4})$. Then we have  
\begin{multline}\label{presub}\int_{B_{R}(x)}\frac{\eta_R(x-y)\Delta u(y)}{4\pi |x-y|}dv_y = u(x)\\
 + 
\int_{B_{R}(x)}(\frac{\Delta\eta_R(x-y)}{4\pi |x-y|}+  \frac{\nabla\eta_R(x-y)\cdot\nabla|x-y|}{2\pi |x-y|^2}) u(y)dv_y.
\end{multline}
Hence for some $C_1,C_2>0$, 
\begin{align}\label{subharmonic}|u(x)|\leq C_1R^{-2} + 
C_2R^{-3/2}\sqrt{\int_{B_{R}(x)}|u|^2dv}.
\end{align}
Differentiating \eqref{presub} yields 
\begin{align}\label{dsubharmonic}|\nabla u(x)|\leq \tilde C_1R^{-3} + 
\tilde C_2R^{-5/2}\sqrt{\int_{B_{R}(x)}|u|^2dv}.
\end{align}
In order to apply these results to $\mu_a$, we first decompose $\mu_a$ in spherical harmonics:
\begin{align}\label{SperDec}
\mu_a = \sum\mu_a^k(r)Y_k\left(\frac{x}{r}\right),
\end{align}
 where $Y_k$ is a spherical harmonic corresponding to a harmonic homogeneous polynomial  of degree $k$ on $\mathbb{R}^3$, and $Y_0 =1$. 
Then we still have $|\Delta\mu_a^k | = O(r^{-4})$ for all $k$. For $\mu_a^0$ this gives us 
\begin{align}\label{orsq}
\p_r (r^2\p_r\mu_a^0) = O(r^{-2}).
\end{align}
Integrating \eqref{orsq}, we see $\lim_{r\to\infty}r^2\p_r\mu_a^0=:-\frac{\vartheta_a}{2}$ exists and is finite. Hence 
$\lim_{r\to\infty}r^2\p_r(\mu_a^0-\frac{\vartheta_a}{2r})=0$. Moreover 
\begin{align}\label{ode}\p_rr^2\p_r(\mu_a^0-\frac{\vartheta_a}{2r}) = O(r^{-2}),
\end{align}
since $-\frac{\vartheta_a}{2r}$ is harmonic. Integrating \eqref{ode} from $\infty$ to $r$ now yields
\begin{align} \p_r\left(\mu_a^0-\frac{\vartheta_a}{2r}\right) = O(r^{-3}).
\end{align}
Let $\frac{\lambda_a}{\ell}:=\limsup_{r\to\infty}\mu_a^0(r)$. Then we may integrate 
the equality\\ $\p_r(\mu_a^0-\frac{\lambda_a}{\ell}-\frac{\vartheta_a}{2r}) = O(r^{-3})$ from $\infty$ to $r$ to deduce 
\begin{align}\mu_a^0 = \frac{\lambda_a}{\ell} +\frac{\vartheta_a}{2r} + O(r^{-2}). 
\end{align}

Now we consider $\mu_a^1$. It satisfies 
\begin{align} \label{2harm} \p_r r^{-2}\p_r (r^2\mu_a^1)    = O(r^{-4}).
\end{align}
Integrating  \eqref{2harm} from $\infty$ to $r$ yields 
\begin{align} \label{2harm2}   \p_r (r^2\mu_a^1)    = O(r^{-1}).
\end{align}
Hence integrating from any fixed $O(1)$ $r_0$ to $r$ yields 
$$|\mu_a^1|=O(\frac{\ln(r)}{r^2}).$$

Set $\tilde \mu_a = \mu_a-\mu_a^0-\mu_a^1Y_1$. 
Taking the $L^2$ inner product of the equation $\Delta\tilde \mu_a  = O(r^{-4})$ with $\phi^2\tilde \mu_a$, for some compactly supported radial function $\phi$  gives 
\begin{multline}\label{aughardy1}
\|\nabla (\phi\tilde \mu_a)\|^2 - \||d\phi|\tilde \mu_a\|^2 \leq C\int\frac{\phi^2|\tilde \mu_a|}{(r^2+1)^2}dv\\
\leq C \epsilon\int\frac{\phi^2|\tilde \mu_a|^2}{(r^2+1) }dv+C\epsilon^{-1}\int\frac{\phi^2 }{(r^2+1)^3}dv. 
\end{multline}
For functions with vanishing $Y_0$ and $Y_1$ coefficients in their spherical harmonic expansion, Hardy's inequality strengthens in dimension 3 to 
\begin{align}\label{aughardy2}\|\nabla (\phi\tilde \mu_a )\|^2\geq \frac{25}{4} \| \frac{\phi \tilde \mu_a}{r}\|^2. 
\end{align}
Choose $\phi= \phi_{n } = \frac{r^{\frac{3}{2}}}{\ln(r+1)}$ for $ r\leq n$, $\phi_{n} = \frac{n^{3}r^{-\frac{3}{2}}}{\ln(n+1)}$ for $n\leq r\leq n^4$,
and $\phi_{n} = \frac{n^{-3}\eta_{n^4}}{\ln(n+1)}$ for $r\geq n^4$. ($\eta_R$ is defined in \eqref{etaR}.)
 Observe that 
$$\int_{n^4\leq |x|} |d\phi_{n}|^2\tilde \mu_a^2dv = O(n^{-2}\ln(n+1)^{-2}). $$
Hence  \eqref{aughardy1} and \eqref{aughardy2} imply 
\begin{align}\label{aughardy3}
\left(\frac{25}{4}-\frac{9}{4}-C\epsilon\right)\left\|\frac{\phi_n\tilde \mu_a}{r}\right\|^2 - C\int\frac{\phi_n^2}{(r^2+1)^3}dv= o(n^{-2}). 
\end{align}
Since $\int\frac{r^3dv}{(r^2+1)^3\ln^2(r+1)}<\infty$, we may take the limit of \eqref{aughardy3} as $n\to \infty$ to deduce 
\begin{align}\label{cyborg}\left\|\frac{r^{\frac{1}{2}}\tilde \mu_a}{\ln(r+1)}\right\|^2 <\infty. 
\end{align}
Inserting \eqref{cyborg} into \eqref{subharmonic} we deduce 
\begin{align}\label{lnest}\mu_a = \frac{\lambda_a}{\ell} + \frac{\vartheta_a}{2r}+O(\frac{\ln(r)}{r^2}).
\end{align}
From \eqref{dsubharmonic} and \eqref{cyborg}  we see   that 
\begin{align}\label{nablnest}d\mu_a = -\frac{\vartheta_adr}{2r^2}+O(\frac{\ln(r)}{r^3}).
\end{align}
 
\begin{theo}\label{asyholo}  Let $A$ be a  finite action  self-dual or anti-self-dual connection  on a Hermitian bundle 
$\mathcal{E}$ over TN$_k$. Assume that $A$ has generic asymptotic holonomy.  Then there are real constants 
$\vartheta_a$ and 
$\lambda_a$, with $\frac{\lambda_a-\lambda_b}{\ell}\not\in\mathbb{Z}$ for $a\neq b,$ such that 
\begin{equation}\mu_a =\frac{\lambda_a}{\ell} + \frac{\vartheta_a}{2r} + O\left(\frac{1}{r^2}\right),
\label{sdsharpholeq}
\end{equation}
where  $O(\frac{1}{r^2})$ is used in the $C^1$ sense: the derivative is $O(\frac{1}{r^3})$.
There exists a compact set $K\subset\mathbb{R}^3$, such that outside of $\pi_k^{-1}(K)$,  $\mathcal{E}$ splits as a direct sum of line bundles, 
$\mathcal{E}= \oplus_a\pi_k^*W(a)$, where each $W(a)$ is a line bundle over $\IR^3\setminus K$. With respect to this splitting,   $A$ has the form  
\begin{align}\label{asyconnect}
A=\,\mathop{\oplus}_a\left(-i(\lambda_a+\frac{m_a}{2r})\frac{d\tau+\omega}{V}+\pi_k^* \eta_a\right)+O(r^{-2}),
\end{align}
with $\eta_a$ a  connection on  $W(a)$, and $m_a:=l\vartheta_a+\frac{\lambda_a}{\ell} k$. Moreover, when $A$ is anti-self-dual, $m_a\in \IZ$. 
With respect to the splitting,  
\begin{equation}F_A = i\mathrm{diag} \left( \frac{\vartheta_adr\wedge (d\tau +\omega)}{2r^2} +  \epsilon\frac{V\vartheta_a }{2}dVol_{S^2}\right)+ O\left(\frac{1}{r^3}\right),
\label{faasymp}
\end{equation}
with $\epsilon = -1$ in the ASD case and $+1$ in the $SD$ case.   

\end{theo}
\begin{proof}
Let $A$ denote the connection matrix in the $\{w_a\}$ frame. Let $\vartheta$ denote the matrix with diagonal entries $\vartheta_a$. Then by \eqref{monopole} and \eqref{nablnest}
\begin{align}\label{goodf}F(\p_r,\p_\tau) = \frac{i\vartheta }{2r^2}+O(\frac{\ln(r)}{r^3}). \end{align}
 Hence \eqref{deltaestasd} yields 
\begin{align}\label{deltaestasd2}
& \Delta_{\IR^3}\mu_a(x) =    i   \frac{V_1}{V} \frac{\vartheta_a }{r^2} +O(\frac{\ln(r)}{r^5}). 
\end{align}
In particular,  $ \Delta_{\IR^3}\mu_a(x)$ is $O(\frac{\ln(r)}{r^5})+\text{ radial}$. Hence
we may sharpen \eqref{2harm2} to 
\begin{align} \label{2harm3}   \p_r (r^2\mu_a^1)    = O(\frac{\ln(r)}{r^2}),
\end{align}
and $\mu_a^1    = O(\frac{1}{r^2}).$
Hence 
\begin{align} \label{2harm4}    
\mu_a =\frac{\lambda_a}{\ell} + \frac{\vartheta_a}{2r}+ O(\frac{1}{r^2}).
\end{align}
In \eqref{dsubharmonic},  let  $u = \mu_a -\frac{\lambda_a}{\ell}-\frac{\vartheta_a}{2r} =  O(\frac{1}{r^2}).$ Then $\int_{B_{2R}}|u|^2dv =O(R^{-1})$, and $|\nabla u|\leq \tilde C R^{-3}$. Hence 
\begin{align} \label{d2harm4}    d\mu_a = -\frac{\vartheta_adr}{2r^2}+ O(\frac{1}{r^3}).
\end{align}
Equation \eqref{sdsharpholeq} (in the $C^1$ sense) follows from Equations \eqref{2harm4} and \eqref{d2harm4}. Equation \eqref{faasymp} follows from \eqref{monopole}. 

Consider now the connection matrix $A(\frac{\p}{\p x^j})=\left[A(\frac{\p}{\p x^j})_a^b\right]$. We compute 
\begin{align}&\frac{\p}{\p\tau}A(\frac{\p}{\p x^j})_a^b = \frac{\p}{\p\tau}\langle \nabla_{\frac{\p}{\p x^j}}w_a,w_b\rangle\nonumber\\
&
= -i(\mu_a-\mu_b)A(\frac{\p}{\p x^j})_a^b +\langle F(\frac{\p}{\p\tau},\frac{\p}{\p x^j})w_a,w_b\rangle - i\mu_{a,j}\delta_{ab}.\end{align}
Hence 
\begin{align}\label{odeab}\frac{\p}{\p\tau}(e^{i\tau(\mu_a-\mu_b)}A(\frac{\p}{\p x^j})_a^b) =
  e^{i\tau(\mu_a-\mu_b)}\langle F(\frac{\p}{\p\tau},\frac{\p}{\p x^j})w_a,w_b\rangle  - i\mu_{a,j}\delta_{ab}.\end{align}
When $a\not = b$, $\langle F(\frac{\p}{\p\tau},\frac{\p}{\p x^j})w_a,w_b\rangle = O(r^{-N})$ for all $N\in\mathbb{N}$, by \eqref{asrequested}. Hence, integrating \eqref{odeab} from $0$ to $2\pi$, the periodicity of $A(\frac{\p}{\p x^j})_a^b$ implies 
\begin{align}\label{offdiag}
A(\frac{\p}{\p x^j})_a^b = O(r^{-N}), \text{ for }a\not = b.
\end{align}
When $a=b$ we have 
\begin{align}\label{odeabc}\frac{\p}{\p\tau} A(\frac{\p}{\p x^j})_a^a  =
   \langle F(\frac{\p}{\p\tau},\frac{\p}{\p x^j})w_a,w_a\rangle -i\mu_{a,j}= O(r^{-4}).\end{align}
Here we have used \eqref{monopole} to obtain the $O(r^{-4})$ estimate in \eqref{odeabc}. 

The quadratic curvature decay, and Corollary \ref{simplestav} guarantees that there exists a compact set $K\subset\mathbb{R}^3$ such that the holonomy eigenvalues are distinct in $K^c$ and such that we can choose  a logarithm for the holonomy eigenvalues continuously in $K^c$. On $K^c$,  $\mathcal{E}$ splits  into an orthogonal sum of the holonomy eigenline bundles $l_a$:
\begin{align}
\mathcal{E}|_{\mathrm{TN}_k\setminus K}=l_1\oplus l_2\oplus\ldots\oplus l_n.
\end{align}
The $l_a$ can be obtained as the pullback of bundles from $\IR^3\setminus \pi_k(K)$,  as we now show. 
Define $  W(a)$ to be the line bundle on $\IR^3\setminus K$, whose fiber at  $x$ is 
\begin{align}\label{Eq:LinBs}  
W(a)_x := \text{Ker}(\nabla_{\p_\tau}+i\mu_a(x))\subset C^1(\pi_k^{-1}(x),\mathcal{E}).
\end{align} 
This definition depends on the choice of the branch of the log of the holonomy eigenvalues $\ln(e^{2\pi i\mu_a})$ on $K^c$.  
    $ W(a)$ inherits a connection from $\mathcal{E}$ as follows. A choice of  local holonomy eigenframe $\{w_b\}_b$ defines  local sections $\tilde w_a$ of each $  W(a)$ by $\tilde w_a(x)= w_a(x,\cdot)$. Define 
 \begin{align} \label{connWdef}\nabla^{  W(a)}_{X}\tilde w_a(x) =   \frac{ \langle \nabla_{X^h}w_a, w_a\rangle_{L^2(\pi_k^{-1}(x))} }{\|w_a\|^2_{L^2(\pi_k^{-1}(x))}}\tilde w_a(x) ,
\end{align}
where $X^h$ denotes the horizontal lift of $X$,  with
$X^h = X -\omega(X)\frac{\p}{\p\tau} $ in a local trivialization of the circle bundle. Write 
\begin{align} \nabla^{  W(a)}_{X}\tilde w_a =     (A^0(X)_a^a+i\mu_a\omega(X))\tilde w_a ,
\end{align}
where $A^0(X)_a^a(x) $ denotes the average value of $A(X)_a^a $ on $\pi_k^{-1}(x)$. Let $\eta_a$ denote the connection defined by $\nabla^{  W(a)}$ on $W(a)$. Let $\nabla^{\pi_k^*\eta_a}$ denote the covariant derivative defined by $\pi_k^*\eta_a$. By definition, 
$$ \nabla^{\pi_k^*\eta_a}_{\frac{\p}{\p x^j}}w_a = (A^0(\frac{\p}{\p x^j})_a^a+i\mu_a\omega(\frac{\p}{\p x^j}))w_a,$$
and 
$$\nabla_{\frac{\p}{\p \tau}}^{\pi^*\eta_a}w_a = 0.$$
Hence 
\begin{align}\label{deltaconn}(\nabla- \nabla^{\pi_k^*\eta_a})_a^a &= (A(\frac{\p}{\p x^j})_a^a-A^0(\frac{\p}{\p x^j})_a^a-i\mu_a\omega(\frac{\p}{\p x^j})dx^j-i\mu_ad\tau\nonumber\\
&=  -i\mu_a(d\tau+\omega)+O(r^{-4}),\end{align}
where we have used  \eqref{odeabc} to obtain the second equality. Expand 
$$\mu_aV =\mu_a(l+\frac{k}{2r})+O(r^{-2}) = \lambda_a+\frac{l\theta_a+\frac{\lambda_ak}{l}}{2r}+O(r^{-2})$$
to deduce \eqref{asyconnect}. 
By \eqref{deltaconn}, we have 
\begin{align}\label{dumrel}\pi_k^*F_{\eta_a}= (F_A)_a^a +id\mu_a\wedge (d\tau+\omega)+i\mu_ad\omega+O(r^{-5}). 
\end{align}
Restrict now to the case where $A$ is ASD. Choosing $\{X,Y\}$ to be an oriented orthonormal basis of $T_pS^2$ (in the radius one metric), \eqref{dumrel}  and \eqref{tndef} imply 
\begin{align}\pi_k^*F_{\eta_a}(X^h,Y^h)&= (F_A)_a^a(X^h,Y^h)+i\mu_ad\omega(X,Y)\nonumber+O(r^{-3})\\
&= -r^2VF_A(\frac{\p}{\p r}, \frac{\p}{\p \tau})_a^a -\frac{ik\mu_a}{2 }+O(r^{-3})\nonumber\\
&=  -i  \frac{l\theta_a+\frac{\lambda_ak}{l}}{2}+O(r^{-1}).
\end{align}

Integrating this equality over $S^2$, we see that 
$$  l\vartheta_a+k\frac{\lambda_a}{\ell}= \frac{i}{2\pi}\int_{S^2} F_{\eta_a}=\int_{S^2}c_1(W(a)).$$
Therefore $m_a = l\vartheta_a+k\frac{\lambda_a}{\ell}\in\IZ$ as claimed.  
\end{proof}

As mentioned in the beginning of this section, the identification of $l_a$ as pullbacks $l_a=\pi_k^*W(a)$ of a line bundle $W(a)$ over $\mathbb{R}^3\setminus K$ depends on the choice of the branch of the log of the holonomy eigenvalue.  This is clear from their definition in Eq.~\eqref{Eq:LinBs}.  Making a different choice changes $\lambda_a/\ell$ by an integer and changes $m_a$ by that integer multiple of $k.$    This choice is indeed significant for the bow construction and will be discussed at length in the third paper in this series.  Let us mention here two natural choices:
\begin{enumerate}
\item
One can choose to have all $\lambda_a\in[0,\ell).$ In this case the monopole charges $m_a$ take any integer values. Renumbering the line bundles, one can choose $0\leq\lambda_1<\lambda_2<\ldots<\lambda_n<\ell.$
\item
One can instead choose to have $0\leq m_a<k$.  In this case, $\lambda_a$ are any real numbers.  Renumbering, one can choose $0\leq m_1<m_2<\ldots<m_n<k.$
\end{enumerate}
In fact, since the end of the $k$-centered Taub-NUT, TN$_k\setminus\pi_k^{-1} K,$ is contractible to the lense space $S^3/(\mathbb{Z}/k\mathbb{Z})$, there are $k$ types of line bundles over it.  It is $m_a \mod k$ that distinguishes the topological type of the line bundle $l_a.$


\section{Asymptotic Decay of Harmonic Spinors}\label{Sec:HarmSpinors}
In order to recover the bow data from an instanton on TN$_k$, it is necessary to understand properties of the $L^2$-kernel of the coupled Dirac operator. Of particular importance is the dimension of its kernel and the decay rates of $L^2$ harmonic spinors. 
\subsection{The Dirac Operator}\label{subSec:Dirac}
Let $A$ be an instanton connection with generic asymptotic holonomy on TN$_k$. According to Theorem \ref{asyholo} (equation \eqref{asyconnect}), outside a compact set we have a splitting with respect to which  
\begin{align}\label{refA}
A = \mathop{\oplus}\limits_{a=1}^n \left(-i\left(\lambda_a+\frac{m_a}{2 r}\right)\frac{d\tau+\omega}{V} + \pi_k^*\eta_a \right)+O\left(\frac{1}{r^2}\right),\end{align} 
where the values $\lambda_a/l$ depend on a choice of gauge in a neighborhood of $\infty$ and are pairwise  distinct mod $\IZ$.  By a choice of gauge at infinity, they can be chosen to lie in the interval $[0,1).$ 
We recall that the spinor bundle $S$  of TN$_k$ splits into eigen-sub-bundles of  Clifford multiplication by the volume element: $S =S^-\oplus S^+$. We set $\gamma^5:=-c(Vdx^1\wedge dx^2\wedge dx^3\wedge d\tau),$ with $\gamma^5|_{S^\pm}=\pm1,$ where $c$ denotes the Clifford action. (Note: In our conventions the defining relation of the Clifford algebra is $c(\alpha)c(\beta)+c(\beta)c(\alpha)+2 g(\alpha,\beta)=0$ for any one-forms $\alpha$ and $\beta.$). The connection $A$ induces a coupled Dirac operator  $D=D_A$ acting on $\Gamma(S\otimes \mathcal{E})$ with the chiral split:
\begin{equation}\label{Dirac1}
D=\begin{pmatrix}
0&D^+\\D^-&0\end{pmatrix},
\end{equation}
where 
$D^-:\Gamma(S^-\otimes \mathcal{E})\to \Gamma(S^+\otimes \mathcal{E})$ and
$D^+:\Gamma(S^+\otimes \mathcal{E})\to \Gamma(S^-\otimes \mathcal{E}).$
\subsection{Harmonic Spinors: the Fredholm Case}
In Section \ref{Sec:Index} we compute the $L^2$-index of $D_A$. However, this operator is not always Fredholm. The following lemma characterizes which instanton connections produce Fredholm Dirac operators. 

\begin{lemma} \label{Fred}
Let $X$ be a spin manifold equipped with an exhaustion by nested compact sets $\{X_j\}_{j=1}^\infty$.  Suppose that $X_1^c$ is a Riemannian circle bundle, $\pi:X_1^c\to Y$. Set $T_j:=\inf_{ X_j^c}\frac{1}{| \p_\tau|^2}$, with the circle fiber locally parameterized by $\tau\in[0, 2\pi)$. Assume $T_\infty:=\liminf_{j\to\infty}T_j >0.$
Let $A$ be a connection on a bundle $\mathcal{E}$ over $X$. Assume $\exists \kappa>0$ so that the eigenvalues $\{e^{2\pi i\mu_a}\}_a$ of the holonomy $\tilde H_p$ of $S\otimes\mathcal{E} $ satisfy $|\mu_a-m|>\kappa$, for all $m\in \IZ$, $\forall p\in X_1^c$. Let $\Sigma$ denote the scalar curvature and $\Sigma^-:= \min\{0,\Sigma\}$. Assume further that $\|F_A\|_{L^\infty(X_j^c)}+\|\Sigma^-\|_{L^\infty(X_j^c)}\leq \epsilon_j$, with $\lim_{j\to\infty}\frac{\epsilon_j}{T_j} = 0$. Then 
$D_A:H_1^2(S\otimes \mathcal{E})\to L^2(S\otimes \mathcal{E}) $ is Fredholm. In particular, if $A$ is an instanton connection on TN$_k$ with generic asymptotic holonomy and $\frac{\lambda_a}{\ell}\not\in \IZ$, $\forall a$, then $D_A$ is Fredholm.  
\end{lemma}
\begin{proof}
It is well known  (see, e.g., \cite{anghel}) 
that $D$ is Fredholm if and only if there is a compact set $K\subset X$  and a constant $C_K$, such that $\| Dh\|^2_{L^2}\geq C_K\| h\|^2_{L^2}$ for all  $h\in C_c^\infty(S\otimes \mathcal{E})$ with compact support  $\mathrm{supp}(h)\subset  K^c$. For such $h$, the Lichnerowicz formula gives     
 $$ \| Dh \|^2_{L^2}= \| \nabla h \|^2_{L^2}+\langle h,( \frac{\Sigma}{4}+c(F_A))h\rangle_{L^2},$$
where $c(F_A)$ denotes Clifford multiplication by the curvature form $F_A$.
By hypothesis, if $K = X_j$, then   by Lemmas \ref{poincholonomy} and \ref{verygeneric},
\begin{align*}
\| Dh\|^2_{L^2}&\geq \| \nabla h\|^2_{L^2}-C\epsilon_j\|h\|^2_{L^2} \geq T_j\| \nabla_{\p_\tau} h\|^2_{L^2}-C\epsilon_j\|h\|^2_{L^2}\\
&\geq T_j\left(\kappa^2 -C\frac{\epsilon_j}{T_j}\right)\|  h\|^2_{L^2}.
\end{align*}
 Choosing $j $ sufficiently large yields the desired result.  From Lemma \ref{verygeneric} and Theorem \ref{lucky} it is immediate that an instanton on TN$_k$ with $\frac{\lambda_a}{\ell}\not\in \IZ$, $\forall a$ satisfies the conditions on $\kappa,$ $\epsilon_a$, and $T_a$. 
\end{proof} 

Estimates implying Fredholmness usually imply exponential decay of $L^2$ zero modes.  
\begin{proposition} \label{spindecay}
Let $A$ be an instanton with generic asymptotic holonomy and splitting \eqref{refA}, with $\frac{\lambda_a}{\ell}\not\in \IZ$, $\forall a$, and choose a positive $\alpha<\inf\left\{|\frac{\lambda_a}{\ell}+n|: n\in\mathbb{Z}\right\}$. Let $h\in \mathrm{Ker}(D_A)$ with $e^{-\beta r}h\in L^2$ for some $\beta<\alpha$. Then  $e^{b r}h\in L^2$ for all $b<\alpha$, and $h$ decays pointwise exponentially. 
\end{proposition}
\begin{proof}
 Let $\eta_n = e^{(b+\beta)r_n}e^{-\beta r}$,  where $r_n = \min\{r,n\},$ and $b\geq\beta.$
Then, using $D^2=\nabla^*\nabla
+c(F_A)$ and the quadratic decay of $F_A$, we get  
$$ 0 = \| D(\eta_n h)\|^2 - \|c(d\eta_n) h\|^2  \geq  \| \nabla(\eta_n h)\|^2 - b^2\| \eta_n  h\|^2 - C\|\eta_n r^{-1} h\|^2. $$
Lemmas \ref{poincholonomy} and \ref{verygeneric} and the additional hypothesis on the $\lambda_a$ imply that, for some compact set $K_\alpha$,
$$\| \nabla(\eta_n h)\|^2\geq \| \alpha  \eta_n h\|^2 - \alpha^2\int_{K_\alpha}|\eta_n h|^2dv.  $$ 
Hence  for some $\tilde C>0$, 
$$ 0  \geq   \alpha^2 \| \eta_n h \|^2 - b^2\| \eta_n  h\|^2 - \tilde C \|\eta_n r^{-1} h\|^2. $$
Choosing $b<\alpha$, we have 
$$ \tilde C \|\eta_n r^{-1} h\|^2  \geq   (\alpha^2-b^2) \| \eta_n h \|^2. $$
Taking the limit as $n\to\infty$, we see that $e^{br}h\in L^2$. To pass to pointwise estimates, we note 
$$\frac{1}{2}\Delta|h|^2 = -|\nabla h|^2 +\langle c(F_A)h,h\rangle\leq c|F_A|^2|h|^2.$$
Hence, applying Proposition~\ref{moserflex} to $f = |h|^2$ with $R=\frac{\delta(M)}{8}$, we  
 deduce $e^{br}h$ is pointwise bounded. 
\end{proof}
\subsection{Harmonic Spinors: the Non-Fredholm Case}\label{Sec:Spinors} 
In this section, $A$ is an instanton connection with generic asymptotic holonomy and therefore admitting the decomposition \eqref{refA}. The results of this section will not be used in the rest of this paper, but they will play a role in the third paper in this series.  
There, the reconstruction of the bow representation from an instanton requires understanding the decay rates of $L^2$-zero modes of $D_A,$ whether it is Fredholm or not. 
\begin{proposition}\label{needscoro} Suppose that $\lambda_1 = 0$ and $h$ is an $L^2$ harmonic spinor. Then 
$\|r^{m}\nabla^m h\|_{L^2}<\infty,$ $\forall m$. 
\end{proposition}
\begin{proof} Assume that 
$\|r^{j}\nabla^j h\|_{L^2}<\infty,$ $\forall j<m$. Replacing $F_A$ and its Bochner formula by $h$ with its Lichnerowicz formula in \eqref{ksub1} and using the decay estimates for the curvature of Proposition \ref{killfourier}, yields the inequality for some $ C_m >0$, 
\begin{align}\label{ksub3}\frac{1}{2}\Delta| \nabla^{m}h|^2&\leq -|\nabla^{m+1}h|^2+C_m\sum_{i=0}^m (1+r)^{-2-i}|\nabla^{m-i}h||\nabla^{m}h|. 
\end{align}
Let $\eta_n \in C_c^2(B_{2n}(0))$ be a bounded sequence of cutoff functions satisfying $\eta_n(x) = 1$ for $x\in B_n(0)$,   and $\|r d\eta_n\|_{L^{\infty}}+\|r^2 \Delta\eta_n\|_{L^{\infty}}<\alpha$, for some $\alpha>0$ independent of $n$. Multiplying \eqref{ksub3} by $\eta_nr^{2m+2 }$ and integrating by parts yields 
\begin{align}\label{ksub4}&\int\eta_n|r^{m+1 }\nabla^{m+1}h|^2dv \leq  \tilde C_m\int\sum_{i=0}^m |r^{m-i}\nabla^{m-i}h||r^{m}\nabla^{m}h|dv. 
\end{align}
Take the limit as $n\to\infty$ to deduce $|r^{m+1}\nabla^{m+1}h|\in L^2$. By induction, this is true for all $m$, since it is true for $m=0$ and for $\epsilon >0$ arbitrary. 
\end{proof}
Assume for the remainder of this section that $0 = \lambda_1<|\lambda_{j}| , \forall j>1$. Let $C$ be a circle in $\IC$ with center $0$ and radius $\frac{1}{4}\min\{|\mu_a-m|:m\in\IZ\text{ and }a>1\}.$  Once again define a projection operator 
$\Pi_0:= \frac{1}{2\pi i}\oint_C\left(z-i\nabla_{{\p_\tau}}\right)^{-1} dz$, now acting on bundle valued spinors, and set $\Pi_1 := 1-\Pi_0$.  Observe that outside a compact set, the decomposition $\mathcal{E} = \oplus_a l_a$ into holonomy eigenline bundles induces a decomposition of $S\otimes \mathcal{E}$ as $\oplus_a(S\otimes l_a)$. Let $h = \sum_a h_a$ denote the corresponding decomposition of a spinor. Then the analog of equation \eqref{analog} is  
$$\Pi_0h = \Pi_0h_1.$$  
 As with  Proposition \ref {polydecay}, we immediately obtain the following corollary to Proposition \ref{needscoro}. 
\begin{coro}\label{hdecay} For all $N,j\in\IN$, 
$\| r^N\nabla^j_{\p_\tau}\Pi_1h\|_{L^\infty}<\infty$.
\end{coro} 

\begin{lemma}\label{degencurv}
If $\lambda_1 =m_1 =0$,  there exist $c_A>0$ and $c_{h,N}>0$   $\forall N$ such that   
$|\nabla_{\p_\tau}\Pi_0h| \leq c_Ar^{-2}|h|,$ $|c(F_A)\Pi_0h|\leq c_Ar^{-3}|h|,$ and $|c(F_A) h|\leq c_Ar^{-3}|h|+c_{h,N}r^{-N}.$
\end{lemma} 
\begin{proof}By  equation \eqref{spec}, the eigenvalue of $i\nabla_{\frac{\p}{\p \tau}}$ on the image of $\Pi_0$ is $\mu_1$. By \eqref{sdsharpholeq}, 
$\mu_1 = O(r^{-2})$, proving the first inequality of the lemma.     By the generic asymptotic holonomy condition, only $\lambda_1$ vanishes; therefore the image of $\Pi_0$ is contained in the holonomy $\mu_1$ eigenspace;   $m_1 =\lambda_1 = 0$ implies 
$d\mu_1 = O(r^{-3})$, by \eqref{sdsharpholeq}.  The ASD condition implies all curvatures are determined by  $F(\p_\tau,\cdot)$. Equation \eqref{monopole} now implies that these curvature components are $O(r^{-3})$ on the image of $\Pi_0$, giving the second inequality. Writing $|c(F_A) h|\leq |c(F_A) \Pi_0h|+|c(F_A) \Pi_1 h|$, the third inequality now follows from the second and Corollary \ref{hdecay}. 
 
\end{proof}
\begin{theo}\label{NFharmonic}
Let $A$ be an instanton on TN$_k$ with generic asymptotic holonomy and such that  $\lambda_1=0$ and $m_1=0$. If  $h\in Ker (D_A)\cap L^2$, then $\|r^2h\|_{L^\infty}<\infty$. 
\end{theo}
\begin{proof} We follow the proof of Theorem \ref{lucky} with a few changes.  We are now aided by the linearity of the equation $D^2h=0$ and the prior knowledge of the quadratic curvature decay. On the other hand, we need a substitute for the Kato-Yau inequality. For the final maximum principle argument, we introduce a new iterated maximum principle.  

Let  $\{\hat{e}_j\}_{j=1}^4$ be an orthonormal tangent frame  with $\hat{e}_1=\frac{1}{\sqrt{V}}\p_r$ and $\hat{e}_4=\sqrt{V}\p_\tau$ and coframe  $\{\hat{e}^j\}_{j=1}^4$. Let $c^j$ denote Clifford multiplication by $\hat{e}^j$.

From  Lemma \ref{degencurv} and Corollary \ref{hdecay}, 
\begin{align}\label{derivest}|\nabla_{\hat{e}_4}h|^2 \leq 2|\nabla_{\hat{e}_4}\Pi_0h|^2+2|\nabla_{\hat{e}_4}\Pi_1h|^2
\leq O(r^{-4}|h|^2)+  O(r^{- N}).
\end{align}
Equation \eqref{derivest} gives a refined  Kato inequality as follows. First rewrite the Dirac equation $Dh = 0$ as 
$$ \nabla_{\hat e_1}h = c^1c^2\nabla_{\hat e_2}h
+c^1c^3\nabla_{\hat e_3}h+c^1c^4\nabla_{\hat e_4}h.$$ 
Hence  
\begin{align}\label{katoh} |\nabla_{\hat e_1}h|^2 &\leq 2(|  \nabla_{\hat e_2}h|^2+ |\nabla_{\hat e_3}h|^2)+|\nabla_{\hat e_4}h|^2+4|\nabla_{\hat e_4}h|\sqrt{|  \nabla_{\hat e_2}h|^2+ |\nabla_{\hat e_3}h|^2}\nonumber\\
&\leq (2+r^{-1})(|  \nabla_{\hat e_2}h|^2+ |\nabla_{\hat e_3}h|^2)+4r|\nabla_{\hat e_4}h|^2\nonumber\\
&\leq (2+ r^{-1})(|  \nabla_{\hat e_2}h|^2+ |\nabla_{\hat e_3}h|^2)+O( r^{-3}|h|^2)+O(  r^{1- N}) .
\end{align}
We rewrite this as 
\begin{align}\label{katoh2}  |\nabla_{\hat e_1}h|^2 \leq \frac{(2+ r^{-1})}{3+r^{-1}} |  \nabla h|^2
+O( r^{-3}|h|^2)+O( r^{1- N}) .
\end{align} 
Take the $L^2$ inner product $0 = \langle D^2h,\eta^2h\rangle_{L^2}$, use the Lichnerowicz formula, integrate by parts,   apply  \eqref{katoh}  and  Lemma \ref{degencurv}  to obtain
\begin{align}\label{intelichn}
0 =&\|\nabla (\eta  h)\|^2_{L^2} -  \||d\eta | h\|^2_{L^2} + \langle c(F)\eta h,\eta h\rangle_{L^2}\nonumber\\
\geq& \|\nabla_{\hat e_1} (\eta  h)\|^2_{L^2} +\|\frac{\eta \nabla_{\hat e_1}  h}{\sqrt{2+r^{-1}}}\|^2_{L^2} -  \||d\eta | h\|^2_{L^2}\nonumber \\
&-\int O(r^{-3}\eta ^2|h|^2)dv-\int O(\eta ^2r^{1- N})dv,
\end{align} 
valid for $\eta $ such that $\eta h\in H^2_1,$ including $\eta $ with $\|\frac{\eta  }{1+r}\|_{L^{\infty}}+\|d\eta \|_{L^{\infty}}<\infty.$  We choose $\eta = \eta_T = r_T^pr^{1/2},$ with $r_T = \min\{r,T\}$, and use  Kato's inequality and \eqref{hardy2}  to deduce : 
\begin{multline}\label{intelichnest}
\frac{1}{4}\|\frac{\eta_T h}{r\sqrt{V}}\|^2_{L^2} +\|\frac{\eta_T\nabla_{\hat e_1}  h}{\sqrt{2+r^{-1}}}\|^2_{L^2} \leq  (p^2t_{T,p,h}+pt_{T,p,h}+\frac{1}{4}) \|\frac{\eta_T}{r\sqrt{V}} h\|^2_{L^2} \\
+\int O( r^{-3}\eta_T^2|h|^2)dv+\int O(\eta_T^2r^{1- N})dv,
\end{multline} 
with $t_{T,p,h}:= \frac{\int|h|^2\chi_Tr_T^{2p }r^{-1}V^{-1}dv }{\int|h|^2 r_T^{2p }r^{-1}V^{-1}dv }.$
As in equation \eqref{intbyparts},  a divergence computation yields
\begin{align}\label{intbyparts2}\frac{r_T^{2p}V^{-\frac{1}{2}}}{\sqrt{2+r^{-1}}}\hat e_1|h|^2 
= & \mathrm{div} \left(\frac{r_T^{2p}|h|^2\hat{e}_1}{\sqrt{2+r^{-1}}}\right)
 -  \frac{2|h|^2 }{\sqrt{2 }}(p\chi_T+1)\frac{r_T^{2p }}{rV}   + O\left(|h|^2  r_T^{2p }r^{-2}\right).
\end{align}
Integrating this equation and applying Cauchy-Schwartz yields for some $C>0$, 
\begin{align}\label{intbyparts2} \|\frac{r_T^{ p}h}{V^{\frac{1}{2}}r^{\frac{1}{2}}}\|_{L^2}\|\frac{r_T^{ p}r^{1/2}}{\sqrt{2+r^{-1}}}\nabla_{\hat e_1}h\|_{L^2}
\geq   \int\frac{|h|^2}{\sqrt{2 }}(pt_{T,p,h}+1)\frac{r_T^{2p }}{rV}dv   - C\int |h|^2  r_T^{2p }r^{-2}dv.
\end{align}
Hence  
\begin{align}\label{intbyparts3}  \|\frac{r_T^{ p}r^{1/2}\nabla_{\hat e_1}h}{\sqrt{2+r^{-1}}}\|_{L^2}^2 
\geq   \frac{1 }{2 } (pt_{T,p,h}+1)^2 \| \frac{r_T^{p }}{\sqrt{rV}}h\|_{L^2}^2 - \sqrt{2 }C(pt_{T,p,h}+1) \int |h|^2  r_T^{2p }r^{-2}dv .
\end{align} 
Inserting this inequality into \eqref{intelichnest}  gives 
\begin{multline}\label{intelichnest2}
\left( \frac{1 -p^2}{2 }+ \frac{p^2 }{2 } (1-  t_{T,p,h} )^2  \right)\| \frac{r_T^{p }}{\sqrt{rV}}h\|_{L^2}^2 -  \sqrt{2 }C(pt_{T,p,h}+1) \int |h|^2  r_T^{2p }r^{-2}dv   \\
\leq \int O( r^{-3}\eta_T^2|h|^2)dv+\int O(r^{1- N})\eta_T^2dv.
\end{multline} 
Taking the limit as $T\to\infty$ for $p\leq 1$, we deduce   $(1-p)\|r^{p-\frac{1}{2}}h\|_{L^2}^2$ is uniformly bounded  and therefore $(1-p)\|\Phi(r^{p-\frac{1}{2}}h)\|_{L^1}$  is uniformly bounded. Consider $x\in \IR^3$. Write  $|x| = 2R$.   By \eqref{philap}, the Lichnerowicz formula,  and Lemma \ref{degencurv}, on $B_{R}(x)$, we have  
\begin{multline}\label{above}
\frac{1}{2V}\Delta (\Phi( h)+R^{-6}) = -\Phi(\nabla h) - Q(c(F_A) h,h)= O(r^{-3}|\Pi_0h|^2)+O(r^{-N})\\
\leq C_hR^{-3}(\Phi( h)+R^{-6}), 
\end{multline}  
for some $C_h>0$ independent of $x$.  Corollary \ref{moserapp} now gives $\forall p<1$, 
\begin{align}\label{phibound} \Phi(h(x))+R^{-6}&\leq \frac{C}{R^3}\int_{B_R(x)}\Phi(h)dv+CR^{-6}\nonumber\\
&\leq \frac{C_2}{R^{ 2+2p}}\int_{B_R(x)}r^{2p-1}\Phi(h)dv+CR^{-6}
\leq \frac{\tilde C_3R^{-2-2p}}{(1-p)}.
\end{align}
Applying Lemma \ref{rote} gives 
\begin{align}\label{pbound} 
|h(x)|^2\leq \frac{C_3R^{-2-2p}}{(1-p)}.
\end{align}
To sharpen this estimate, we next employ the maximum principle.

Fix $p<2$, and suppose that $\frac{1}{2}r^p\Phi( h) $ has a  critical point at  $x_p$. Then  
\begin{align}\label{critpt}Q(\nabla_{\hat e_1}h,h)(x_p) &= -\frac{p}{2r\sqrt{V}}\Phi(h)(x_p),\text{ and  therefore} \nonumber\\
\Phi(\nabla_{\hat e_1}h )(x_p) &\geq  \frac{p^2}{4r^2V}\Phi(h)(x_p).
\end{align}
Assume now that $x_p$ is also a local maximum. Applying \eqref{philap} to $h$ and then using \eqref{critpt}  yields at $x_p$
\begin{multline}
V^{-1}\Delta_{\IR^3}\frac{1}{2}r^p\Phi( h)(x_p)  = -r^p\Phi(\nabla h)(x_p)-r^pQ(c(F_A)h,h)(x_p)\\
-\frac{p(1+p)}{2V}r^{p-2}\Phi(h)(x_p)-\frac{2pr^{p-1}}{\sqrt{V}}Q(h,\nabla_{ \hat e_1} h)(x_p)\\
= -r^p\Phi(\nabla h)(x_p)-r^pQ(c(F_A)h,h)(x_p)-\frac{p(1-p)}{2V}r^{p-2}\Phi(h)(x_p) .
\end{multline}
We now use \eqref{katoh2}  and then \eqref{critpt} to estimate $V^{-1}\Delta_{\IR^3}\frac{1}{2}r^p\Phi( h)(x_p)$ as
\begin{align}0&\leq V^{-1}\Delta_{\IR^3}\frac{1}{2}r^p\Phi( h)(x_p)\nonumber\\
&  \leq   -\frac{(3+ r^{-1})}{2+r^{-1}}r^p\Phi(\nabla_{\hat e_1}h)(x_p) -(\frac{p(1-p)}{2V}r^{p-2}   +O(r^{p-3}))\Phi(h)(x_p) +O(r^{-N}) \nonumber\\
 &\leq   (- \frac{(4-p^2)r^{p-2}}{8V}  +O(r^{p-3}))\Phi(h)(x_p)+O(r^{-N}) ,
\end{align}
where the implied constants in $O(r^{-N})$ are independent of $p\leq 2$. Hence  at a maximum, choosing  $N=8-p$, we have for some $C>0$,
\begin{align}\label{cholla}
 r(2-p)  \Phi(h)  \leq C\Phi(h)+O(r^{-5}),
\end{align}
and either $r\leq \frac{2C}{ 2-p}$ or $\Phi(h)\leq O(r^{-5})$ at $x_p$. 

 Consider a sequence $p_k:= 2-e^{-k}$.  We seek to prove that $r^{2p_k}\Phi(h)$ is bounded with bound independent of $k$. Assume this is false.  Let the maximum value of $r^{2p_k}\Phi(h)$ occur at $x_{p_k}$.   Then $\lim_{k\to\infty}r^{2p_k}\Phi(h)(x_{p_k}) = \infty$. In particular, for $k>k_0$, for some $k_0\in \IN$, by \eqref{cholla}, $r(x_{p_k})\leq 2Ce^{k}$.     On the other hand,  
\begin{align}\frac{r^{2p_k}(x_{p_k})}{r^{2p_{k-1}}(x_{p_k})}=  r(x_{p_k})^{2e^{1-k}-2e^{-k}}< (2C)^{2e^{1-k}} e^{2ke^{1-k}}=:A_k.
\end{align}
\begin{align}\Phi(r^{p_k}h)(x_{p_k})\leq A_k\Phi(r^{p_{k-1}}h)(x_{p_k}) \leq A_k\Phi(r^{p_{k-1}}h)(x_{p_{k-1}}).\end{align} 
Iterating, we see that 
\begin{align} \Phi(r^{p_k}h) \leq  \Phi(r^{p_{k_0-1}}h)(x_{p_{k_0-1}})\prod_{k=k_0 }^\infty A_j.
\end{align}   
Take the limit as $k\to\infty$ to obtain
\begin{align}\|\Phi(r^2h)\|_{L^\infty}\leq  \Phi(r^{p_{k_0-1}}h)(x_{p_{k_0-1}})\prod_{k=k_0}^\infty A_k.
\end{align} 
 The product $\prod_{k=k_0}^\infty A_k<\infty$. Hence $\Phi(r^2h)$ is bounded and by Lemma \ref{rote}, $|r^2h|$ is bounded. 
\end{proof}


\section{The Index Theorem}\label{Sec:Index}
So far we have assumed only that $A$ is an instanton with generic asymptotic holonomy, so by Thm.~\ref{asyholo} the values $\exp(2\pi i \lambda_a/\ell)$ are pairwise distinct.  Thus, by \eqref{2harm4} this implies that the connection is $\kappa$-generic outside of some compact set in $\mathbb{R}^3$ and Lemma~\ref{poincholonomy} applies.  For the rest of this paper we also assume that 
\begin{assume}\label{assume}
$\exp\left(2\pi i\frac{\lambda_a}{\ell}\right)\not = 1$ for all $a.$
\end{assume}
It implies that the hypotheses of Lemma \ref{Fred} hold and the Dirac operator $D_A:H^2_1(TN_k,S\otimes \mathcal{E})\to L^2(TN_k,S\otimes \mathcal{E})$ is therefore  Fredholm. The objective now is to compute its $L^2$-index. The argument follows   \cite{SSZ}  and \cite{SternH}: 
 
\begin{enumerate}
\item In order to simplify the analysis, we apply a conformal transformation to the original TN$_{k}$ metric, and modify the connection so that it is asymptotically abelian. This does not change the $L^2$-index of the Dirac operator. The new metric is asymptotically that of a circle bundle with shrinking fiber over an $\mathbb{R}_+\times S^2$ base. Working with this new metric and connection greatly simplifies error estimates and allows us to use techniques developed in \cite{SternH} to compute the index. With this metric,   $\eta-$invariant type spectral terms are essentially replaced by their readily computed adiabatic limits.    
\item  We express the index as a sum of  terms involving the super-trace of the heat kernel $e^{-tD^2}$.  The index can be written as a sum of two terms:  {the }bulk and the asymptotic contribution. The bulk  involves  the classical Atiyah-Singer integrand, while the asymptotic contribution depends  on the behavior at infinity of the instanton connection and on the rescaled metric. 
\item We approximate the heat kernel by a parametrix, and show that the index can be computed from the parametrix. It follows,  as in the compact case, that this approximation can be used to compute the bulk contribution. It requires somewhat more work to show that the approximation can also be used to compute the asymptotic contribution. Our conformal change of the metric facilitates this step as well. Because the local injectivity radius tends to zero exponentially fast in the new metric, we first Fourier expand in the circle fibers. This step simultaneously avoids the introduction of exponentially large errors associated with cutoff functions localizing to geodesic neighborhoods and makes error terms associated with nonzero Fourier coefficients exponentially small. 
\item Finally, we compute the index from the parametrix. 
\end{enumerate}
\subsection{Index Preliminaries}\label{Sec:Prelim}
Let $r$ denote the Euclidean distance from the origin in the $\IR^3$ base of TN$_k$. Lift the function $r$ to TN$_k$. Multiplying the TN$_k$ metric $g$ (see \eqref{metric}) by a smooth conformal factor that equals $\frac{1}{Vr^2}$  for $r$ large, yields a new metric $g'$.  For large $r$, the new metric takes the form
\begin{equation}\label{ConfMet}  
g'=\frac{1}{r^2}dr^2+g_{S^2}+\frac{1}{V^2r^2}(d\tau+\omega)^2=dy^2+g_{S^2}+e^{-2y}V^{-2}(d\tau+\omega)^2,
 \end{equation}
where $y = \ln(r)$, $\omega$ is defined in \eqref{varpi}, and $g_{S^2}$ denotes the standard round metric on the unit $2-$sphere.

Observe that if $\Psi$ is a $p-$form (possibly with values in $ad(E)$), then
\begin{align}|\Psi|_{g'} = V^{\frac{p}{2}}e^{py}|\Psi|_{g} ,
\end{align} for $y$ large. Hence in the $g'$ metric, we have 
$$|F_A| = O(1).$$

\begin{proposition}\label{sameind}
The Dirac operators associated to $g=g_{\mathrm{TN}_k}$ and $g'$ have the same $L^2$-index.
\end{proposition}
\begin{proof}
Write the conformal factor as $e^{-2u},$ then  
near infinity it has the form $e^{-2u}=\frac{1}{Vr^2},$ so the corresponding Dirac operators are related by
$$D_{g'}=e^{\frac{3u}{2}}D_g e^{\frac{-3u}{2}}.$$
(See \cite[Section 1.4]{Hitchin74}.) Define an injective map
 \begin{align}
 T:\mathrm{Ker}_{L^2_g}(D_g)&\to\mathrm{Ker}_{L^2_{g'}}(D_{g'})\nonumber\\
 h&\mapsto e^{\frac{3u}{2}}h.
 \end{align} 
 Since the $L^2_g$-solutions  of $D_gh=0$ decay exponentially in $r$ by Proposition \ref{spindecay}, $T$ takes $\mathrm{Ker}_{L^2_g}(D_g)$ to $\mathrm{Ker}_{L^2_{g'}}(D_{g'})$ and is therefore well-defined. Now, let $\phi\in \mathrm{Ker}_{L^2_{g'}}(D_{g'}),$ then by Proposition \ref{spindecay}$, e^{\frac{-3u}{2}}\phi$ decays exponentially, since it belongs to $\mathrm{Ker} (D_g)$ and $e^{-\beta r}e^{\frac{-3u}{2}}\phi\in L^2$ for all $\beta>0$. Hence $T$ is an isomorphism.
 \end{proof}
Henceforth we work in the conformally rescaled metric $g'$. In particular, from now on $D_A:= D_{A,g'}$, etc. 

Outside a compact set, the connection $A$ on $\mathcal{E}$ induces an abelian connection 
$\tilde A^{\ab}$ 
on the sum $\oplus_{a=1}^n l_a$ of eigenline bundles  of the holonomy. 
By \eqref{offdiag}, 
\begin{align}|A-\tilde A^{\ab}| = O(e^{-Ny}),\quad\forall N.\end{align}
For large $R$ (such that $\{\nu_j\}\in B_R(0)),$ define the modified connection,
$$A^{\ab}:=\,\mathrm{diag}\left(-i(\lambda_a+\frac{m_a}{2r})\frac{d\tau+\omega}{V}+\pi_k^*\eta_a\right),$$
with $\eta_a$ as in Eq.~\eqref{asyconnect}. 
This connection has the convenient property that the connection matrices are constant in the TN$_k$ fiber in every local tangent frame consisting of vectors commuting with $\p_\tau$. Moreover, by \eqref{asyconnect},
\begin{align}|A-  A^{\ab}| = O(e^{-2y}).\end{align}

Let  $\eta\in C^\infty(\IR)$ be supported in $(-\infty,1)$,  identically $1$ on $(-\infty,0]$, with $|\eta'|\leq 2$. For  $R$ large, we define a new connection $A_R:= \eta(y-R) A + (1-\eta(y-R))A^{\ab}$.  Define $$M_s:=y^{-1}([0,s]).$$ For some $s_0>0$, $M_{s_0}^c$ is a circle bundle  over a cylinder $\mathbb{R} \times S^2$ with  the $S^1$-fibers shrinking rapidly as $y\to\infty$.  $\{ M_y\}_{y} $ is an exhaustion of TN$_k$ by  compact sets.

\begin{lemma} 
$D_{A_R}$ is Fredholm, and $\mathrm{index}\,(D_{A_R}) = \mathrm{index}\,(D_A).$ 
\end{lemma}
\begin{proof}Set 
$A(t) = (1-t)A+tA_R$. We  apply  Lemma \ref{Fred} to $D_{A(t)}$, with  
$X_{j}: = M_{j+s_0}$. In the notation of Lemma \ref{Fred}, $T_j \geq c_1e^{2j}$, and $\epsilon_j\leq c_2$ for some $c_1, c_2>0$.  Assumption \ref{assume} easily implies for some $\kappa>0$, $|\mu_a-m|>\kappa,$ $\forall m\in \IZ, \forall a$. $\lim_{j\to\infty}\frac{\epsilon_j}{T_j} = 0$. Hence Lemma \ref{Fred} implies $D_{A(t)}$ is Fredholm. Clearly  $t\to D_{A(t)}$ is a continuous family of Fredholm operators. The index is a continuous integer valued function on the space of Fredholm operators, and therefore constant on curves. Hence $ \mathrm{index}\,(D_{A }) = \mathrm{index}\,(D_{A(0)})=\mathrm{index}\,(D_{A(1)}) = \mathrm{index}\,(D_{A_R}).$ 
\end{proof}
We remark that  shifting the index  computation to $D_{A_R}$ is not essential; it merely notationally simplifies certain computations by removing numerous exponentially small but nonzero commutators.  

For a subset  $U$  of $S^2$ such that the $S^1$-bundle of \eqref{ConfMet} over it is trivial, we choose a local oriented  $g'$-orthonormal frame $(e_1,e_2,e_3,e_4)$ with 
\begin{align}\label{frame}
e_1=\bar e_1&=\partial_y,&
e_2&=\bar{e}_2-\omega(\bar{e}_2)\partial_\tau,&
e_3&=\bar{e_3}-\omega(\bar{e}_3)\partial_\tau,&
e_4&= e^yV\partial_\tau,
\end{align}
where $\{\bar e_1,\bar{e}_2,\bar{e}_3\}$ is a local oriented orthonormal frame on $\IR\times U$ lifted to TN$_k$ via the above product structure. The  corresponding coframe is  $e^1=dy$, $e^2=\pi^*\bar{e}^2$, $e^3=\pi^*\bar{e}^3$ and $e^4=e^{-y}V^{-1}(d\tau+\omega)$, where  $\{\bar{e}^j\}_{j=2}^3$    is  dual to $\{\bar{e}_j\}_{j=2}^3$ and $\pi$ denotes the $S^1$-bundle projection. For example, using spherical coordinates $(\phi, \theta)$ on $U,$ with polar angle $\phi\in[0,\pi]$ and azimuthal angle $\theta\in[0,2\pi),$ and $U$  containing neither North nor South pole, one can take $\bar{e}_2=\p_\phi$ and $\bar{e}_3=\frac{1}{\sin \phi}\p_\theta$. We recall that  $c^j:=c(e^j)$ and that the chirality operator is $\gamma^5=-c(e^1e^2e^3e^4)=-c^1c^2c^3c^4.$

\begin{lemma}\label{generalind}
The $L^2$-index of $D_A^-:\Gamma(S^-\otimes \mathcal{E})\to\Gamma(S^+\otimes \mathcal{E})$ is given by 
\begin{equation}\begin{split}\label{l2ind}
{}\mathrm{ind}_{L^2} D_A^-&=-\frac{\mathrm{Rank}(\mathcal{E})}{192\pi^2}\int_M\tr_{_{\scriptscriptstyle T(\mathrm{TN}_k)}} \mathcal{R}\wedge  \mathcal{R}
+\frac{1}{8\pi^2}\int_M \tr_{_{\scriptscriptstyle \mathcal{E}}} F\wedge F\\
&{-}\lim_{R\to\infty}\lim_{y\to \infty}\frac{1}{2}\int_{e^{-(2+\delta)y}}^\infty\int_{\partial M_y}\mathrm{tr}\, c(\nu)\gamma^5D_{A_R}e^{-tD_{A_R}^2}(x,x){d\nu_y dt},
\end{split}\end{equation}
where  $\nu$ is the unit outward normal to $\p M_y$, $d\nu_y$ is the induced volume form on $\partial M_y,$ and  $\delta\in (0,1)$. 
 \end{lemma}
(In this section, $x$ will  denote a point in TN$_k$ rather than in $\IR^3$). We call the last summand of 
\eqref{l2ind} {\em the asymptotic contribution}.

\begin{proof}
Let $P^R$ denote the $L^2-$unitary projection onto  $\mathrm{Ker} (D_{A_R})$, and  let $p^R(x,x')$ denote the Schwarz kernel of $P_R$. Then 
$$ \mathrm{ind}_{L^2} D_A^-:=\mathrm{dim}\,\mathrm{Ker}D_{A_R}^--\mathrm{dim}\,\mathrm{Ker}D_{A_R}^+={-}\Tr\,\gamma^5 P^R ={-} \int_M\tr\,\gamma^5 p^R(x,x)d\nu,$$
where  $\Tr$ denotes trace over the Hilbert space of $L^2$ sections, and $\tr$ denotes the pointwise trace of endomorphisms of $S\otimes \mathcal{E}$.

Let $k^R(t,x,x')$ denote the Schwarz kernel of $e^{-tD_{A_R}^2}$.  
Observe that in the strong operator topology $P^R = \lim_{t\rightarrow \infty}e^{-tD_{A_R}^2}.$ Hence 
$p^R(x,x') = \lim_{t\rightarrow \infty}k^R(t,x,x').$ This limit is not uniform in $(x,x')$, but for any compact subset $K$,  
$\int_K\tr\,\gamma^5p^R(x,x)d\nu = \lim_{t\rightarrow\infty}\int_{K}\tr\,\gamma^5 k^R(t,x,x)d\nu.$ (See \cite[Lemma 2.2.3]{Stern1}, replacing $(\frac{ D^+D^-}{w}+1)^{-k+1}$ with $e^{-t D^+D^-}$).  Hence,
\begin{align*}
 \int_M\tr\,\gamma^5 p^R(x,x)d\nu &= \lim_{y\rightarrow\infty}\int_{ M_y}\tr\,\gamma^5 p^R(x,x)d\nu\\
&= \lim_{y\rightarrow\infty}\lim_{t\rightarrow\infty}\int_{M_y}\tr\,\gamma^5 k^R(t,x,x)d\nu.
\end{align*}
Following Callias  \cite[Proposition 1]{Callias}, we use the fundamental theorem of calculus and the divergence theorem to rewrite this as 
\begin{multline*}
 \int_M\tr\,\gamma^5 p^R(x,x)d\nu= \\
= \lim_{y\rightarrow\infty}\left(\int\limits_{e^{-(2+\delta)y}}^\infty\frac{d}{dt}\int\limits_{M_y}\tr\,\gamma^5 k^R(t,x,x)d\nu dt +  \int\limits_{M_y}\tr\,\gamma^5 k^R(e^{-(2+\delta)y},x,x)d\nu \right) \\
= \lim_{y\rightarrow\infty}\left(\frac{1}{2}\int\limits_{e^{-(2+\delta)y}}^\infty \int\limits_{M_y} e_i \tr\,c^i\gamma^5 {D_{A_R}}k^R(t,x,x)d\nu dt  +  \int\limits_{M_y}\tr\,\gamma^5 k^R(e^{-(2+\delta)y},x,x)d\nu \right)\\
= \lim_{y\rightarrow\infty}\left(\frac{1}{2}\int\limits_{e^{-(2+\delta)y}}^\infty  \int\limits_{\p M_y} \tr\,c(\nu)\gamma^5 D_{A_R}k^R(t,x,x)d\nu dt  +\int\limits_{M_y}\tr\,\gamma^5 k^R({e^{-(2+\delta)y}},x,x)d\nu \right).
\end{multline*}   
The second equality in this expression follows from 
\begin{multline}
\int_{M_{y}}\frac{d}{dt}\tr\, \gamma^5 k^R(t,x,x')|_{x=x'}d\nu=-\int_{M_{y}}\tr\, \gamma^5(D_{A_R}^2k^R)(t,x,x')|_{x=x'}d\nu\\
=- \int_{M_{y}}\tr\, \gamma^5(D_{A_R}\circ k^R\circ D_{A_R})(x,x')|_{x=x'}d\nu
=-\frac{1}{2}\int_{M_{y}}\tr\, \gamma^5((D_x+D_{x'}){D_x}k)|_{x=x'}d\nu\\
= \frac{1}{2}\int_{M_{y}}e_i \tr\, c^i\gamma^5 (D_{A_R}\circ k^R(t,x,x))d\nu.
\end{multline}
Here we have used $D_x$ and $D_{x'}$ to distinguish  2 different lifts of $D_{A_R}$ to $M\times M$. 
Following \cite{roe}, we  now construct an approximate heat kernel of the form 
$$k_N^R(t,x,x')=\eta(x,x')h_t(x,x')\sum_{j=0}^N t^j\Theta_j (x,x'),$$
where $h_t(x,x')=\frac{1}{(4\pi t)^2}e^{-\frac{d(x,x')^2}{4t}}$ (with $d(x,x')$ the $g'$ distance between $x$ and $x'$) and $\eta(x,x')=\tilde\eta(d(x,x'))$ is a cut-off function  supported in a neighborhood $\mathcal{N}$ of the diagonal, such that for $(x,x')\in \mathcal{N}$,  $x$ lies in a normal coordinate neighborhood of $x'$ and $x'$ lies in a normal coordinate neighborhood of $x$. It suffices to choose $\tilde \eta$ to be supported on $[0,\frac{1}{2\ell}e^{-y} ]$, identically one on $[0,\frac{1}{4\ell}e^{-y} ]$, and satsifying for $j=1,2$,  $|\frac{d^j}{dt^j}\tilde\eta|\leq c_j e^{jy}$,  for some $c_j>0$.   The evaluation of the trace of the approximate heat kernel is standard, and we will simply refer to  \cite{roe} for the computation of the trace asymptotics. We will also estimate the error for $t= e^{-(2+\delta)y}$ and show that the error terms in this approximation vanish as $y\to \infty$. This estimate is also standard, except for the effect of the small injectivity radius. We include this estimate in order  to clarify the effects of the small injectivity radius on the analysis. The construction of approximate heat kernel in this subsection is only suitable for   $t$ small relative to $e^{-2y}$. 

To define the $\Theta_j$ inductively, let $\Theta_0(x,x')$ denote parallel translation from $x'$ to $x$ along the radial geodesic. 
Compute (see \cite[Lemma 7.12, Lemma 7.13, and Theorem 7.15]{roe}) 
\begin{multline}\label{induct0}
(\frac{\p}{\p t} + D_{A_R}^2)h_t(x,x')\sum_{j=0}^N t^j\Theta_j (x,x')
= h_t(x,x')[\sum_{j=0}^N t^jD_{A_R}^2\Theta_j (x,x')\\
+  \sum_{j=0}^{N-1} t^jr^{-j}(\det g')^{-\frac{1}{4}}   \nabla_{\frac{\p}{\p r}} ( r^{j+1}(\det g')^{\frac{1}{4}}\Theta_{j+1} (x,x'))],
\end{multline}
with $D_{A_R}^2$ differentiating in the $x$ variable. 
Set 
$$ \nabla_{\frac{\p}{\p r}} ( r^{j+1}(g')^{\frac{1}{4}}\Theta_{j+1} (x,x')) = -(g')^{\frac{1}{4}}r^jD_{A_R}^2\Theta_j (x,x'),$$
which we solve by integrating along geodesic rays in a radially covariant constant frame. With this choice, \eqref{induct0} reduces to 
\begin{align}\label{induct1}&(\frac{\p}{\p t} + D_{A_R}^2)(h_t(x,x')\sum_{j=0}^N t^j\Theta_j (x,x'))= h_t(x,x') t^ND_{A_R}^2\Theta_N,
\end{align}
and setting $k_N^R:= \eta h_t \sum_{j=0}^N t^j\Theta_j$, we have 
\begin{align}\label{induct1}&(\frac{\p}{\p t} + D_{A_R}^2)k_N^R= \eta h_t t^ND_{A_R}^2\Theta_N + (\Delta\eta -2\nabla_{\nabla\eta})h_t \sum_{j=0}^N t^j\Theta_j.
\end{align}

We now use Duhamel's principle to estimate the error arising when we use $k_N^R$ to estimate the trace. Let 
$$\epsilon_N(t,x,x'):= (\frac{\p}{\p t} + D_{A_R}^2)k_N^R(t,x,x').$$
Then
\begin{align}\label{duhamel}&k_N^R(t,x,x') - k^R(t,x,x')   = \int_0^t\int_Mk^R(t-s,x,w)\epsilon_N(s,w,x')dwds\nonumber\\
&= \int_0^t\int_Mk_N^R(t-s,x,w)\epsilon_N(s,w,x')dwds \nonumber\\
&-   \int_0^t\int_0^{t-s}\int_M\int_Mk^R(t-s-u,x,z)\epsilon_N(u,z,w)\epsilon_N(s,w,x')dzdwduds.
\end{align}
Taking the trace, we have 
\begin{align}\label{trduhamel}
&\int_{M_y}(\tr \gamma^5k_N^R(t,x,x) - \tr \gamma^5k^R(t,x,x))dx   \nonumber\\
&= \int_{M_y}\int_0^t\int_M\tr \gamma^5k_N^R(t-s,x,w)\epsilon_N(s,w,x')dwdsdx \nonumber\\
&-   \int\limits_{M_y}\int\limits_0^t\int\limits_0^{t-s}\int\limits_M\int\limits_M\tr \gamma^5k^R(t-s-u,x,z)\epsilon_N(u,z,w)\epsilon_N(s,w,x)dzdwdudsdx.
\end{align}
The last summand can be written as 
\begin{align}\label{lastsumm}
\int_0^t\int_0^{t-s}\Tr \gamma^5e^{-(t-s)D_{A_R}^2}\epsilon_N(u)\chi_{M_{y+1}}\epsilon_N(s)\chi_{M_y},
\end{align}
where $\epsilon_N(w)$ denotes the operator with Schwarz kernel $\epsilon_N(w,x,y),$ and $\chi_X$ denotes the characteristic function of the set $X$. We have used the fact that $\epsilon_N$ is supported near the diagonal to insert an additional $\chi_{M_{y+1}}$. Using the boundedness of $e^{-sD_{A_R}^2}$ we have 
\begin{align}
\left|\Tr \gamma^5e^{-(t-s)D_{A_R}^2}\epsilon_N(u)\chi_{M_{y+1}}\epsilon_N(s)\chi_{M_y}\right|
&\leq \left\|\gamma^5e^{-(t-s)D_{A_R}^2}\epsilon_N(u)\chi_{M_{y+1}}\epsilon_N(s)\chi_{M_y}\right\|_{Tr}\nonumber\\
&\leq \| \epsilon_N(u)\chi_{M_{y+1}}\|_{HS}\|\epsilon_N(s)\chi_{M_y}\|_{HS}.
\end{align}
Here $\|\cdot\|_{Tr}$ denotes trace class norm and  $\|\cdot\|_{HS}$ denotes Hilbert-Schmidt norm. (See \cite[Section 2]{SternCHA}
for a summary of relevant properties of these norms.) 
  
The $\Theta_j$ can be computed recursively in terms of curvatures and their derivatives (see for example \cite[Section 3.3]{SternCHA}), and are therefore bounded. Hence, \eqref{induct1} gives 
\begin{multline}
|\epsilon_N(t,x,x')|\leq C_0\eta(x,x')e^{-\frac{d^2(x,x')}{4t}}t^N+C_1\left(e^{2y}+e^{y}\frac{d(x,x')}{t}\right)e^{-\frac{d^2(x,x')}{4t}}\chi_{\text{support }d\eta}\\
\leq \left[C_0 t^N+C_1\left(e^{2y}+ \frac{1}{2\ell t}\right)e^{-\frac{e^{-2y}}{64l^2t}}\right]\chi_{\mathcal{N}}.
\end{multline}  
Here we have used $d(x,x')\geq \frac{e^{-y}}{4l}$ on the support of $d\eta$. For $t\leq e^{-(2+\delta)y}$, we have 
\begin{align}|\epsilon_N(t,x,x')|\leq \left[C_0 e^{-N(2+\delta)y}+C_1\left(e^{2y}+ \frac{e^{(2+\delta)y}}{2\ell}\right)e^{-\frac{e^{\delta y}}{64\ell^2}}\right]\chi_{\mathcal{N}} = O(e^{-2Ny}).
\end{align}  
Recalling that 
$$\|\epsilon_N(s)\chi_{M_y}\|_{HS}^2 = \int_{M_{y+1}\times M_y} |\epsilon_N(s,x,x')|^2dxdx'\leq Ce^{-2Ny},$$ 
we see that the quantity in \eqref{lastsumm} is exponentially decreasing. Similarly, 
$$\left|\int_{M_y}\int_0^t\int_M\tr\, \gamma^5 k_N^R(t-s,x,y)\epsilon_N(s,y,x')dydsdx\right|\leq Ce^{-2Ny},$$ 
and we deduce 
\begin{align}\label{trduhamel2}&\int_{M_y} \tr\, \gamma^5 k^R(t,x,x)dx = \int_{M_y}  \tr\, \gamma^5 k_N^R(t,x,x)dx  +O(e^{-Ny}),
\end{align}
for $t\leq e^{-(2+\delta)y}.$

  It follows, as in the compact case, (see e.g. \cite[Chapter 12]{roe}) that 
\begin{align}
\int_{M_y} \tr\, \gamma^5k^R\left(e^{-(2+\delta)y},x,x\right)dx =\int_{M_y} \mathrm{ch}(\mathcal{E},A_R)\wedge \hat{A}(\mathrm{TN}_k)+O(e^{-Ny}).
\end{align}  
In dimension four, 
$\hat{A}=1-\frac{1}{24}p_1 =1+\frac{1}{192\pi^2}\tr\, \mathcal{R}\wedge \mathcal{R},$ where $p_1$ is the first Pontryagin form, and the Chern character is
$$\mathrm{ch}(\mathcal{E},A_R)=\mathrm{Rank}(\mathcal{E})+\frac{i}{2\pi}\tr\, F_{A_R}-\frac{1}{8\pi^2} \tr\, F_{A_R}\wedge F_{A_R}.$$ 
Therefore, 
\begin{multline}
\lim_{y\to\infty}\int_{M_y} \tr \gamma^5k^R(e^{-(2+\delta)y},x,x)dx = \int_{M }(\frac{\mathrm{Rank}(\mathcal{E})}{192\pi^2}\tr \,\mathcal{R}\wedge \mathcal{R}{-}\frac{1}{8\pi^2} \tr F_{A_R}\wedge F_{A_R})\\
= \int_{M }(\frac{\mathrm{Rank}(\mathcal{E})}{192\pi^2}\tr\, \mathcal{R}\wedge \mathcal{R}{-}\frac{1}{8\pi^2} \tr F_{A}\wedge F_{A}).
\end{multline}
The last equality follows from using  Chern-Weil theory to express $ \tr F_{A_R}\wedge F_{A_R} -  \tr F_{A}\wedge F_{A}$ as the differential of a rapidly decreasing three-form. 
 \end{proof}

\begin{lemma}\label{Pont}
The integral of the first Pontryagin form over TN$_k$ is
\begin{equation}
\frac{1}{192\pi^2}\int_{\mathrm{TN}_k}\tr\, \mathcal{R}\wedge \mathcal{R}=\frac{k}{12}.
\end{equation}
\end{lemma}
\begin{proof}
This is computed in \cite{HAWKING197781} and can also be established by direct calculation. Computing in the original $g$ metric, we have, as in \cite{Nergiz:1995qg}, that $\tr \,\mathcal{R}\wedge \mathcal{R}=\frac{1}{2}(\Delta\Delta V^{-1})dx^1\wedge dx^2\wedge dx^3\wedge d\tau$. Since $V$ is harmonic, 
$ \Delta V^{-1}=-2\frac{|\nabla V|^2}{V^3},$ 
 and $\Delta V^{-1}=-\frac{4}{r_\sigma}+O(r_\sigma^0)$ near each center $\nu_\sigma$. Applying Stokes theorem, the computation of the Pontryagin number reduces to  
$$\frac{1}{192\pi^2}\int\tr\, \mathcal{R}\wedge \mathcal{R}=-\sum_{\sigma}\frac{1}{192\pi}\int_{S^2_{\nu_\sigma}}\nabla_n\big( \frac{|\nabla V|^2}{V^3} \big) d\mathrm{Vol}_{S^2_{\nu_\sigma}},$$ where $S^2_{\nu_\sigma}$ denotes a small sphere centered at $\nu_\sigma\in\mathbb{R}^3$, and $n$ is the outward unit normal.  The integral over each center $\nu_\sigma$  contributes $\frac{1}{12}$.  Since the Pontryagin forms are invariant under conformal transformations (see \cite{CS}), the Pontryagin number is the same for the $g$ and $g'$ metrics.
\end{proof}


\subsection{Approximate Heat Kernel for arbitrary $t$, large $y$}

In Lemma \ref{generalind}, expression (\ref{l2ind}) for the index of $D_A^-$ involves two summands:  the Atiyah-Singer integrand and the asymptotic contribution. The latter is more subtle; hence we will give more details of the analysis of this term. We first specify an iterative semilocal approximation to the heat kernel. We then prove that substituting the approximation for the exact kernel computes the asymptotic contribution to the index.  

Consider an open neighborhood $\mathcal{U}\subset M_{R+2}^c$  admitting a section of the $S^1$ bundle, 
containing a point $p$. Introduce ($p$ and section dependent) coordinates to $\mathcal{U}\times\{p\}$ with 
$x:=(y, \phi,\theta,\tau)$ as in Section~\ref{Sec:Prelim}, so that $x(p) = (y(p),\frac{\pi}{2},0,0)$ and so that 
$$\pi_k^{-1}\Bigg(\bigg(\begin{smallmatrix}\cos(\theta)\sin(\phi)&-\sin(\theta)&-\cos(\theta)\cos(\phi)\\
\sin(\theta)\sin(\phi)&\cos(\theta)&-\sin(\theta)\cos(\phi)\\ 
\cos(\phi)&0& \sin(\phi)\end{smallmatrix}\bigg)\pi_k(p)\Bigg)$$ 
has coordinates $(y(p),\phi,\theta,\ast)$. Choose a frame $\{e_j\}_{j=1}^4$ defined by these coordinates as in the  paragraph following \eqref{frame}. 
Let $b=(b_1,b_2,b_3)=(y, \phi,\theta)$  so that $x = (b,\tau)$.  We use the continuous Fourier transform in the base variables and the  discrete Fourier expansion in  $\tau$. Let $v$  denote coordinates dual to the $b$ variables.  Write $u = (v,\frac{ \kappa}{{ 2\pi}})=(v_1,v_2,v_3,\frac{ \kappa}{{ 2\pi}})$, where $ \kappa\in \IZ$ labels  
the discrete Fourier modes.  We make use of the compact notation $$\int\ldots du := \sum_{ \kappa\in\IZ}\int_{\IR^3}\ldots dv.$$

For a suitable contour $C$ surrounding $[0,\infty)\subset\IR\subset\mathbb{C}$ oriented counterclockwise, we use the Cauchy integral formula to write
$$e^{-tD_{A_R}^2}=\frac{-1}{2\pi i}\oint_Ce^{-tz}(D_{A_R}^2-z)^{-1}dz{.}$$
Hence,   an approximation of $(D_{A_R}^2-z)^{-1}$ yields an approximation of $e^{-tD_{A_R}^2}$. 
We iteratively construct an approximation of the Schwartz kernel of $(D_{A_R}^2-z)^{-1}$ of the form
\begin{multline}\label{approx}
\sum_{j=0}^N\int e^{2\pi i(x-x')u}\sigma_z^{-j-1}(x)q_j(x,x')\psi(x,x')\,du\\
={\sum_{j=0}^N\sum_{ \kappa\in \IZ}\int_{\IR^3} e^{2\pi i(b-b')v}{e^{{{i}(\tau-\tau') \kappa}}}\sigma_z^{-j-1}q_j(x,x')\psi(x,x')\,dv,}
\end{multline}      
{where $\sigma_z$ is constructed below from the symbol of $D_{A_R}^2$, $\psi$ is a fiber isomorphism specified below, and the $q_j$ are defined inductively.} 
{Observe} that $$\int_{\IR^3}  e^{2\pi i(y-y')v_1}e^{2\pi i(\phi-\phi')v_2}e^{2\pi i(\theta-\theta')v_3}dv_1dv_2dv_3$$  represents the ({distributional}) integral kernel of the delta distribution with respect to the form $dy\wedge  d\phi\wedge d\theta$. 

We now specify each term in (\ref{approx}). First we  specify on $\mathcal{U}$
a frame dependent identification   $\psi(x,x')\in\textrm{Hom}(S_{x'}\otimes \mathcal{E}_{x'},S_x\otimes \mathcal{E}_x)$ of the fibers of the bundle $S\otimes \mathcal{E}$ at $x'$ and at $x$. The local frame $\{e_j = e_j(x,x')\}_{j=1}^4$ on $\mathcal{U}\times \{x'\}$  defines a section of the bundle of oriented orthonormal frames. We lift this to a local section of the principal spin bundle and use it to define a local frame $\{f_a = f_a(x,x')\}_{a=1}^4$ for the spin bundle.   Define $\psi^S(x,x')=f_a(x,x')\otimes f_a^*(x',x')$ where $\{f_a^*\}$ denotes the dual coframe.
Take a local unitary $H_p$ (holonomy) eigenframe frame $\{s_l\}$ of $\mathcal{E}$ on $\mathcal{U}$ and define $\psi^{\mathcal{E}}(x,x')=s_l(x)\otimes s^*_l(x')$, where again $\{s^*_l\}$ denotes the dual frame. Now set $\psi = \psi^S\otimes \psi^{\mathcal{E}}$.  Observe that 
\begin{align}
 \nabla^S_{e_j}\psi^S(x,x')&=\Gamma_j^S(x)\psi^S(x,x'),&
 \nabla^{ab}_{e_j}\psi^{\mathcal{E}}(x,x')&={A^{\ab}_{j}(x)\psi^{\mathcal{E}}(x,x'),} \label{con2}
\end{align}
where $e^j\otimes\Gamma_j^S(x)$ and $e^j\otimes A^{\ab}_{j}$ are the connection one-forms for the given frames.
For the computations below, we define $\psi^{\mathcal{E}}(x,x')$ as follows. Let $\tau(x)$ denote the $\tau$ coordinate of $x$ and let 
$$V_{x' }:= \{x\in \mathcal{U}: \tau(x) = \tau(x')\}$$ denote a section of the $S^1$ bundle $\mathcal{U}$ containing $x'$.     Pick a unitary eigenbasis $\{s_a(x')\}_a$, at $x'$, of the holonomy operator. Extend this basis to a frame on $V_{x'}$ by radial parallel translation along $V_{x'}$ with respect to $\nabla^{\mathrm{ab}}$. This gives a frame (diagonal with respect to the holonomy) at each point $(y,\phi,\theta, \tau(x'))$. Extend this to a frame $\{s_a(x,x')\}_a$ over all of $\mathcal{U}$   so that its coefficients in the $\{w_a\}_a$ frame constructed in the first paragraph of section \ref{Sec:Asymp} are constant in each circle fiber.   Restricted to $V_{x'}$, this frame satisfies 
\begin{align} \nabla^{\ab}_{\bar e_m}s_a(x) =  \frac{1}{2}F^{\ab}(r_V\frac{\p}{\p r_V},\bar e_m )s_a(x,x')+O(r_V(x,x')^2), \forall x\in V_{x'},
\end{align}
where $r_V(x,x')$ denotes the distance function on $V_{x'}$ from $x$ to $x'$. (See, for example, \cite[Lemma 3.18]{SternCHA}.) 
By \eqref{odeabc},  
\begin{align}\p_\tau\langle \nabla_{e_j}^{\ab}s_a,s_a\rangle = O(e^{-2y}). 
\end{align}
Hence $\forall x\in \mathcal{U}$, 
\begin{align}\label{connest} \nabla_{\bar e_m}^{\ab}s_a =  \frac{1}{2}F^{\mathrm{ab}}_A(r_V\frac{\p}{\p r_V},\bar e_m )s_a+O(r_V(x,x')^2) +O(e^{-2y}).
\end{align}
We also have from Theorem \ref{asyholo}
\begin{align}\label{conhol} \nabla_{ e_4}s_a  =  -i(  e^y\lambda_a+\frac{m_a}{2})s_a  +O(e^{-y}).
\end{align}
 
Write {\begin{equation}\label{cdu}
e^{{-{i} \tau  \kappa}}e^{-2\pi ib\cdot v}D_{A_R} e^{2\pi ib\cdot v}e^{{{i} \tau  \kappa}}= D_{A_R} +c(d u),
\end{equation}
thus determining $c(du)$ to be  
\begin{equation}\label{cdudef}
c(d u)=2\pi iv_1c^1+\big(2\pi iv_2-\omega(\p_\phi){{i} \kappa}\big)c^2
+
\frac{2\pi iv_3-\omega(\p_\theta){{i} \kappa}}{\sin\phi}c^3+e^yV{{i} \kappa} c^4.
\end{equation}
In order to exploit the holonomy, we  modify the decomposition \eqref{cdu}. As in Section~\ref{Sec:Asymp}, we have $\mu={i{A^{\ab}}{(\p_\tau)}}=\mathrm{diag}(\frac{\lambda_a}{\ell}+\frac{1}{2}e^{-y}\vartheta_a+O(e^{-2y}))$, and we set $\Lambda = \text{diag}(\frac{\lambda_a}{\ell}).$

 Write analogously to \eqref{cdu} 
\begin{equation}\label{dhat}
e^{{-{i} \tau  \kappa}}e^{-2\pi ib\cdot v}D_{A_R} e^{2\pi ib\cdot v}e^{{{i} \tau  \kappa}}=\hat{D}+c(\delta u),
\end{equation}
with   
\begin{multline}\label{cdeludef}
c(\delta u)=2\pi iv_1c^1+\big(2\pi iv_2-\omega(\p_\phi){{i}( \kappa- \Lambda)}\big)c^2\\
+\frac{1}{\sin\phi}\big(2\pi iv_3-\omega(\p_\theta){{i}( \kappa- \Lambda)}\big)c^3
+e^yV{{i}( \kappa- \Lambda)}c^4\\
=2\pi iv_1c^1+ 2\pi iv_2 c^2 +
\frac{2\pi i}{\sin\phi} v_3 c^3+e^yV{{i}( \kappa- \Lambda)}(c^4- \frac{e^{-y}}{V}c(\omega)),
\end{multline}
and $\hat D$ is defined by  equation (\ref{dhat}). This definition shifts the unbounded part of the connection matrix to the $c(\delta u)$ term. It follows  that
\begin{multline}\label{cdu2}
c(\delta u)^2=(2\pi v_1)^2+{e^{2y}V^2( \kappa- \Lambda)^2}\\
+\left(2\pi v_2-{{\omega(\p_\phi)}}( \kappa- \Lambda)\right)^2
+\frac{1}{\sin^2\phi}\left(2\pi v_3-{{\omega(\p_\theta)}}( \kappa- \Lambda)\right)^2.
\end{multline}

For $z\in\mathbb{C},$ set
\begin{equation}\label{sigmadef}\sigma_z=c(\delta u)^2-z.\end{equation}}   
Assumption~\ref{assume} implies that for $R$ large no diagonal entry of $( \kappa- \Lambda)$ is zero, therefore $|\sigma_z|\geq ce^{2y}-|z|$, for some $c\in \IR $.

We introduce the differential operator  
\begin{align}L:&=\hat{D}^2+\{\hat{D},c(\delta u)\} \nonumber\\
&= \hat{D}^2+L_1+  c(d\delta u)+d^*\delta u,
\end{align}
where 
\begin{align}\label{L1def}
L_1:=&-4\pi i v_1\nabla_{e_1} -2(2\pi iv_2-i\omega(\bar e_2)(\kappa-\Lambda)) (\nabla_{e_2}-i\omega(\bar e_2)\Lambda)\nonumber\\
&
- \frac{2}{\sin\phi}(2\pi i v_3- i\omega( \p_\theta)(\kappa-\Lambda)) (\nabla_{e_3}-i\omega(\bar e_3)\Lambda)\nonumber\\
& -2e^yV i( \kappa- \Lambda)(\nabla_{e_4}+ie^yV\Lambda),
\end{align} 
and where $\delta u$ is the covector defined by \eqref{cdeludef}. Then 
\begin{align}\label{Dsigmaj}
  e^{-2\pi i(x-x')u}(D_{A_R}^2-z)e^{2\pi i(x-x')u} &=L+\sigma_z.
\end{align}
Let $\tilde \zeta(t)$ be a smooth function supported in $(-2,2)$ and identically $1$ on $[-1,1]$, with derivative $|d\tilde\zeta|\leq 2.$ Set $ \zeta:M\times M\to \IR$ to be 
\begin{align}\label{zeta}
\zeta(p,q) = \tilde \zeta\left(e^{\frac{3}{4}y(q)}d(\pi_k (p),\pi_k(q))\right).
\end{align}
 Here the distance $d(\cdot,\cdot)$ is with respect to the new metric $g'.$ (In the following construction, the $e^{\frac{3}{4}y(q)}$ factor in the definition of $\zeta$ can be replaced by $e^{\alpha y(q)}$ for any $\alpha\in (0,1)$.  If $\alpha \geq 1$, stationary phase arguments no longer imply the error terms associated with derivatives of the cutoffs are rapidly decreasing. The larger we choose $\alpha$, however, the better our subsequent bounds on connection matrices.) On the support of $\zeta(x,x')$,  the connection matrices in the frame $\{s_a(x,x')\}_a$  satisfy  
\begin{align}\label{connest2} A_R(\bar e_m) =  \frac{1}{2}F_A^{\ab}(r_V\frac{\p}{\p r_V},\bar e_m ) +O(e^{-\frac{3}{2}y})
=O(e^{-\frac{3}{4}y}).
\end{align}

Then  
\begin{multline*}
 (D_{A_R}^2-z)\int e^{2\pi i(x-x')\cdot u}\zeta(x,x')\sigma_z^{-j-1}q_j(x,x')\psi(x,x')du =\\ 
\shoveleft\int e^{2\pi i(x-x')\cdot u}(L+\sigma_z)\zeta(x,x')\sigma_z^{-j-1}q_j(x,x')\psi(x,x')du=\\ 
\shoveleft\zeta(x,x')\int e^{2\pi i(x-x')\cdot u} (\sigma_z^{-j}q_j(x,x')\psi(x,x')  +  L\sigma_z^{-1-j}q_j(x,x')\psi(x,x'))du \\
\shoveright \hfill+[D_{A_R}^2,\zeta(x,x')]\int e^{2\pi i(x-x')\cdot u}\sigma_z^{-j-1} q_j(x,x')\psi(x,x'))du.
\end{multline*}
Set $q_0(x,x') \in \mathrm{End}(S_x\otimes \mathcal{E}_x)$ to be the identity, and define $q_j(x,x')\in \mathrm{End}(S_x\otimes \mathcal{E}_x)$  by setting  
\begin{align}\label{defineqj}q_{j }  = -   \left(\sigma_z^{j }L\sigma_z^{-j }q_{j-1}\psi\right)\psi^{-1}
=    (-1)^{j}\sigma_z^{j}\left((L\sigma_z^{-1})^{j}\psi\right)\psi^{-1},\end{align}
where the $\sigma_z^{-1}$ in $(L\sigma_z^{-1})^{j }$ denotes the operation of left multiplication by $\sigma_z^{-1}$. In particular, $(L\sigma_z^{-1})^{j }$ denotes a $2j$ order partial differential operator which acts on the  $\psi$ factor but not  the $\psi^{-1}$. 
 We define an approximate resolvent kernel by 
\begin{align}q_z^N(x,x'):= \sum_{j=0}^N\zeta(x,x')\int e^{2\pi i(x-x')\cdot u}\sigma_z^{-j-1}q_j(x,x')\psi(x,x')du,
\end{align}
and let $Q^N_z$ denote the operator with Schwartz kernel $q^N_z$. 
Then
\begin{multline}\label{epsdef}
(D_{A_R}^2-z) q_z^N(x,x')\\
= \zeta(x,x')\int e^{2\pi i(x-x')\cdot u}\psi(x,x')du
+ \zeta(x,x')\int e^{2\pi i(x-x')\cdot u}L\sigma_z^{-N-1}q_N\psi(x,x')du\\
\shoveright{+\sum_{j=0}^N[D_{A_R}^2,\zeta(x,x')]\int e^{2\pi i(x-x')\cdot u}\sigma_z^{-j-1}q_j(x,x')\psi(x,x')du}\\
\shoveleft{= I+ \zeta(x,x')\int e^{2\pi i(x-x')\cdot u}L\sigma_z^{-N-1}q_N\psi(x,x')du}\\
\shoveright \hfill+\sum_{j=0}^N[D_{A_R}^2,\zeta(x,x')]\int e^{2\pi i(x-x')\cdot u}\sigma_z^{-j-1}q_j(x,x')\psi(x,x')du.
\end{multline}
We now define our approximate heat kernel, 
\begin{align}
k^R_N(t,x,x'):=   \frac{-1}{2\pi i}\int\limits_Ce^{-tz}q_z^N(x,x') dz,
\end{align}
and let $K_{t,N}$ denote the operator with Schwarz kernel $k^R_N(t,\cdot,\cdot)$.

\begin{lemma}\label{verification}
Let $0\leq \tilde\eta\leq 1$ be a cutoff function supported in $M_{T}^c$, $T\geq R+4$ large.
For any  $t{>} 0$ and {$N$ large},
\begin{align}|\Tr\gamma^5 \eta e^{-tD_{A_R}^2} - \Tr\gamma^5 \eta K_t^N|
=\begin{cases}
 O(   e^{-(2N-3 )T} ),& 
 t\geq e^{-2T}\\
  O(   t^{\frac{N}{2}-1}e^{3T}),& 
  t< e^{-2T}
\end{cases}. 
\end{align}
\begin{align}
|\Tr\gamma^5 \eta c(dy)D_{A_R}e^{-tD_{A_R}^2} - \Tr\gamma^5 \eta c(dy)D_{A_R}K_t^N|=
\begin{cases} 
O(   e^{-(2N-4 )T} ),&
t\geq e^{-2T}\\
 O(   t^{\frac{N}{2}-2}e^{3T}),&
 t< e^{-2T}
\end{cases}. 
\end{align}
\end{lemma}
\begin{proof}
For any $\phi\in H_1^2(S\otimes \mathcal{E})$, one has $\mathrm{Im} ((D_{A_R}^2-z)\phi,\phi)=-\mathrm{Im} (z)\|\phi\|^2$. Thus, 
$\|(D_{A_R}^2-z)\phi\|_{L^2}\geq |\mathrm{Im}(z)|\|\phi\|$. This implies that if   $\mathrm{Im}(z)\neq 0$, then $D_{A_R}^2-z$ is injective with closed range, and $\|(D_{A_R}^2-z)^{-1}\|_{\mathrm{op}}\leq |\mathrm{Im}(z)|^{-1}$.

Consider a counterclockwise oriented curve $C_t$ surrounding the spectrum of $D_{A_R}^2$ defined as follows: $C_t$ is the union of a semicircle $\{ z : |z|=1/t, \mathrm{Re}\, z\leq 0 \}$ and two horizontal half-lines $y = \pm \frac{1}{t}$, $x\geq 0$.  
Observe that $|e^{-tz}|\leq e$ for any  $z\in C_t$.  Moreover, for all  $z\in C_t$,  
$$\|(D_{A_R}^2-z)^{-1}\|_{\mathrm{op}}\leq t.$$

Let $\epsilon_z^N:=I-(D_{A_R}^2-z)Q^N_z$.  
 Then  
\begin{align*}e^{-tD_{A_R}^2} - K_t^N &=   \frac{-i}{2\pi }\int_{C_t}e^{-tz}((D_{A_R}^2-z)^{-1} - Q^N_z)dz\\
&=   \frac{-i}{2\pi }\int_{C_t}e^{-tz}(D_{A_R}^2-z)^{-1}\epsilon_z^Ndz\\
&=   \frac{-i}{2\pi }\int_{C_t}e^{-tz}Q_z^N  \epsilon_z^N dz+\frac{-i}{2\pi }\int_{C_t}e^{-tz}(D_{A_R}^2-z)^{-1}(\epsilon_z^N)^{2}dz.
\end{align*}
Taking the trace yields 
\begin{align*}
\Tr\gamma^5 \eta (e^{-tD_{A_R}^2} - K_t^N) &=  \Tr\gamma^5 \eta \frac{-i}{2\pi }\int_{C_t}e^{-tz}Q^N_z\epsilon_z^Ndz \\
&+  \Tr\gamma^5 \eta \frac{-i}{2\pi }\int_{C_t}e^{-tz}(D_{A_R}^2-z)^{-1}(\epsilon_z^N)^2dz. 
\end{align*}
Hence
\begin{align}\label{errorinf}
|\Tr\gamma^5 \eta (e^{-tD_{A_R}^2} - K^N_t)|&\leq 
\left| \Tr\gamma^5 \eta \frac{1}{2\pi }\int_{C_t}e^{-tz}Q^N_z\epsilon_z^Ndz\right| \\
&+      \frac{1}{2\pi }\int_{C_t}e^{-t\mathrm{Re}\,{z}}t\left|\sqrt{\eta}\epsilon_z^N \right|_{HS}
\left|\epsilon_z^N\sqrt{\eta}\right|_{HS}
d|z|.\nonumber
\end{align}  
By \eqref{epsdef} we can expand the Schwarz kernel 
$\epsilon_z^N(x,x')$ of $\epsilon_z^N$ 
as 
\begin{align}\epsilon_z^N(x,x') =   \epsilon_{z,1}^N(x,x')+\epsilon_{z,2}^N(x,x'),
\end{align}
with  
 $$\epsilon_{z,1}^N(x,x')= \zeta(x,x')\int e^{2\pi i(x-x')\cdot u}L\sigma_z^{-N-1}q_N(x,x')\psi(x,x')du,$$
and
$$\epsilon_{z,2}^N(x,x') := \sum_{j=0}^N[D_{A_R}^2,\zeta(x,x')]\int e^{2\pi i(x-x')\cdot u}\sigma_z^{-j-1}q_j(x,x')\psi(x,x')du.$$
We are thus left to consider three terms (after applying the arithmetic/geometric mean to eliminate cross terms) to estimate  the right-hand side of \eqref{errorinf}: the contributions from $|\epsilon_{z,1}^N|_{HS},$ $|\epsilon_{z,2}^N|_{HS},$ and $| \Tr\gamma^5 \eta \frac{1}{2\pi }\int_{C_t}e^{-tz}Q^N_z\epsilon_z^Ndz|$. 
We will estimate the first term in detail. The remaining terms are similar, except  some terms in the trace integral require integration by parts in $z$ before estimating, and some terms in the $|\epsilon_{z,2}^N|_{HS}$ contribution require integration by parts in the phase space variable.

Define the shifted variables
$$v(j) = \left(2\pi v_1, 2\pi v_2- \omega(\p_\phi) ( \kappa-\frac{\lambda_j}{\ell}),\frac{2\pi v_3- \omega(\p_\theta) ( \kappa-\frac{\lambda_j}{\ell})  }{\sin(\phi)}\right).$$
Observe that $L$ has coefficients that are polynomials of degree at most one in the variables $v$ and $e^y( \kappa-\frac{\lambda_j}{\ell})$. Also $\sigma_z^{-1}$  is a rational function of degree $-2$ in the same variables. Hence $(L\sigma_z^{-N-1}q_N \psi)\psi^{-1}= (-1)^N((L\sigma_z^{-1})^{N+1}\psi)\psi^{-1}$ is a diagonal matrix whose j-th entry is a sum of terms of the form 
\begin{align}\sum_{ \kappa\in\IZ}  p_a(v(j),e^y( \kappa-\frac{\lambda_j}{\ell})) \left[|v(j)|^2+e^{2y}V^2( \kappa-\frac{\lambda_j}{\ell})^2-z\right]^{-N_a} .
\end{align}
Here $N+1\leq  N_a\leq   3N+3,$ and each $p_a$ is a Clifford-bundle-valued  polynomial in $v(j)$ and $e^y( \kappa-\frac{\lambda_j}{\ell})$ with bounded coefficients and of degree  $d_a<N_a$, with $ N_a-d_a\geq N+1$.
Let 
$$\epsilon_0:= \frac{1}{2}\min_{ \kappa,j}\left\{( \kappa-\frac{\lambda_j}{\ell})^2\right\}.$$
Using the arithmetic geometric mean inequality, $\prod_js_j^{a_j}\leq \sum_j\frac{a_j}{\sum_ia_i}s_j^{\sum_ia_i}$, we can estimate   $|\epsilon_{z,1}^N(x,x')|$ when $\mathrm{Re}\, z<\epsilon_0e^{2y}$ for some $C_N,\tilde C_N>0$ and for any $a\in[0,\frac{N}{2}-2)$ by  
\begin{multline*}
 |\epsilon_{z,1}^N(x,x')|\leq \sum_j\zeta(x,x')C_N\sum_{ \kappa\in\IZ} t^a\int_{\IR^3}\frac{dv}{(|v(j)|^2+e^{2y}( \kappa-\frac{\lambda_j}{\ell})^2)^{\frac{N+1}{2}-a}}\\
= 
\zeta(x,x')\sum_j\tilde C_N\sum_{ \kappa\in\IZ}\frac{t^ae^{-(N-2-2a)y}}{| \kappa-\frac{\lambda_j}{\ell}|^{(N-2-2a)}}.
\end{multline*}
Hence, for $N\geq 4+2a$, and $\mathrm{Re}\, z<\epsilon_0e^{2y}$, 
\begin{align}\label{ezest}\|\sqrt{\eta}\epsilon_{z,1}^N \|_{L^2(M\times M)}= O(t^ae^{-(N-\frac{3}{2}-2a)T}).\end{align}
Here we have used the fact that (in the $g'$ metric) $\mathrm{Vol}(M)<\infty$ with $\mathrm{Vol}(M_T^c) = O(e^{-T}).$  
The $t^a$ factor comes from the estimate 
 $|(|v(j)|^2+e^{2y}( \kappa-\frac{\lambda_j}{\ell})^2-z)^{-a}|< t^a$ on $C_t$. This is only useful for $t\leq O(e^{-2y})$. Otherwise we will choose $a=0$. 

Hence,
\begin{align}\label{ezcont}
\int\limits_{C_t:\mathrm{Re}\,(z) <\epsilon_0e^{2T}}e^{-tw}t\|\sqrt{\eta}\epsilon_{z,1}^N \|_{HS}\|\epsilon_{z,1}^N \sqrt{\eta}\|_{HS}|dz|\leq  O(t^{2a-1}e^{-(2N-3-4a)T}).\end{align}

When $w:=\mathrm{Re}\, z\geq \epsilon_0e^{2y}$, we have the weaker estimate : 
$$  |\epsilon_{z,1}^N(x,x')|\leq  \sum_j\zeta(x,x')\tilde C_N\sum_{ \kappa\in\IZ} \int_{\IR^3}\frac{(|v|^{N_j-N-1}+|e^{y}( \kappa-\frac{\lambda_j}{\ell})|^{N_j-N-1})dv}{([|v|^2+e^{2y}( \kappa-\frac{\lambda_j}{\ell})^2-w]^2+t^{-2})^{\frac{N_j+1}{4}}}.$$
We  split the $ \kappa$ sum  into terms where $w\leq \frac{1}{2} e^{2y}( \kappa-\frac{\lambda_j}{\ell})^2$, and terms where  $w\geq \frac{1}{2} e^{2y}( \kappa-\frac{\lambda_j}{\ell})^2$. The first set is infinite and can be treated exactly as the case of $|z|<\epsilon_0e^{2y}$. It's contribution to the Hilbert-Schmidt norm is  again $O(e^{-(N-\frac{1}{2})T})$. We estimate the remaining terms as follows.  
\begin{multline}
\sum_j\zeta(x,x')\hat C_N\sum_{e^{2y}| \kappa-\frac{\lambda_j}{\ell}|^2\leq 2w}\int_{\IR^3}\frac{(|v|^{N_j-N-1}+|e^{y}( \kappa-\frac{\lambda_j}{\ell})|^{N_j-N-1})dv}{([|v|^2+e^{2y}( \kappa-\frac{\lambda_j}{\ell})^2-w]^2+t^{-2})^{\frac{N_j+1}{4}}}\\
\leq \sum_j\zeta(x,x')\hat C_N\sum_{e^{2y}| \kappa-\frac{\lambda_j}{\ell}|^2\leq 2w}\frac{1}{w^{\frac{N-1}{2}}}\int_{\IR^3}\frac{(|v|^{N_j-N-1}+ 1)dv}{([|v|^2+\frac{e^{2y}( \kappa-\frac{\lambda_j}{\ell})^2}{w}-1]^2+t^{-2}w^{-2})^{\frac{N_j+1}{4}}}\\
\leq \sum_j\zeta(x,x')  C_N'\sum_{| \kappa-\frac{\lambda_j}{\ell}|^2\leq 2e^{-2y}w} t^{\frac{N_j+1}{2}}w^{\frac{N_j-N+2}{2}}\\
\leq \sum_j\zeta(x,x')  C_N' e^{- y} t^{\frac{N_j+1}{2}}w^{\frac{N_j-N+2}{2}}.
\end{multline}
Hence, for $N\geq 4$ and $\mathrm{Re}\,(z)\geq\epsilon_0e^{2T}$,  
$$\|\sqrt{\eta}\epsilon_{z,1}^N \|_{L^2(M\times M)}^2= O(e^{-3T} t^{ N_j+1 }w^{ N_j-N+2 }).$$
These terms contribute to \eqref{errorinf} with a term bounded by 
\begin{multline}\label{cont2}
O(e^{-3T})\int_{\epsilon_0e^{2T}}^\infty e^{-tw}t^{ N_j+2 }w^{ N_j-N+2 }dw
 = O\left( t^{ N_j+2 }e^{-t\epsilon_0e^{2T}} e^{ (2N_j-2N+1)T} \right)\\
\leq O\left( \max\{t^{ 3N+5 }e^{-t\epsilon_0e^{2T}} e^{ (4N+7)T }, t^{ N +3}e^{-t\epsilon_0e^{2T}} e^{ 3T}\}\right).
\end{multline}

Combining \eqref{ezcont} and \eqref{cont2} and the corresponding results for the other two contributions then yields 
\begin{multline}\label{errorinf2}
   \left| \Tr\gamma^5 \eta \frac{1}{2\pi }\int_{C_t}e^{-tz}Q^N_z\epsilon_z^Ndz\right|
+ \frac{1}{2\pi }\int_{C_t}e^{-t\mathrm{Re}\,{z}}|t|\left|\sqrt{\eta}\epsilon_z^N \right|_{HS}
\left|\epsilon_z^N\sqrt{\eta}\right|_{HS}
d|z|\\
 =
\begin{cases}
  O(   e^{-(2N-3 )T} ), &\text{ for }t\geq e^{-2T}\\
  O(   t^{\frac{N}{2}-1}e^{3T}), &\text{ for }t< e^{-2T}.
  \end{cases}
\end{multline}  
When we  consider $D_{A_R}e^{-tD_{A_R}^2}-D_{A_R}K_t^N$ in place of $ e^{-tD_{A_R}^2}- K_t^N$, the additional $D_{A_R}$ increases the allowed homogeneity of the polynomials $p_a$ by $1$. The proof then proceeds as before. 
\end{proof}

Since $e^1$ is the unit outward normal to $M_y$, we are left to compute 
\begin{align}\label{step1}
&{-}\lim_{y\to \infty}\frac{1}{2}\int\limits_{e^{-(2+\delta)y}}^\infty\int\limits_{\partial M_y}\tr c^1\gamma^5D_{A_R}e^{-tD_{A_R}^2}(x,x){d\nu_x  dt}=\nonumber\\
& \lim_{y\to \infty}\int\limits_{e^{-(2+\delta)y}}^\infty\int\limits_{\partial M_y}\frac{1}{4\pi i}\int\limits_{C_t}\int \sum_{j=0}^Ne^{-tz}\tr c^1\gamma^5(\hat{D}+c(\delta u))\sigma_z^{-j-1}q_j\psi|_{x=x'}\,du\,dz\, d\nu_x dt. 
\end{align}
The following section is devoted to simplifying  \eqref{step1}.
 
\subsection{Reduction}
The expression \eqref{step1} contains  an enormous number of summands. In this section we show that only those summands corresponding to $q_0$ and $q_1$ contribute to the index formula. 
Write 
\begin{align}\label{comL}L\sigma_z^{-1} = \sigma_z^{-1}L - \sigma_z^{-2}[L,c(\delta u)^2] + \sigma_z^{-3}[[L,c(\delta u)^2],c(\delta u)^2].
\end{align}
The last term in expression \eqref{comL} is zero order and commutes with $\sigma_z^{-1}$. Hence, we can write 
\begin{align}\label{pabc}&\sigma_z^{-j-1}q_j\psi  = (-1)^j\sigma_z^{-1} (L \sigma_z^{-1})^j\psi\nonumber\\
&=   (-1)^j\sum_{a+b+c = j}\sigma_z^{-1-a-2b-3c}p_{a,b,c}(L,[L,c(\delta u)^2],[[L,c(\delta u)^2],c(\delta u)^2])\psi,
\end{align}
where each $p_{a,b,c}(X_1,X_2,X_3)$ is a polynomial homogeneous of degree $a,b,c$ respectively in the noncommutative variables $X_1,X_2,$ and $X_3.$ Then \\$p_{a,b,c}(L,[L,c(\delta u)^2],[[L,c(\delta u)^2],c(\delta u)^2])$ defines a partial differential operator of order $\leq 2a+b$. Inserting expression \eqref{pabc} into \eqref{step1} and performing the contour integration yields
\begin{align}\label{step11}
&\frac{1}{4\pi i}\int\limits_{C_t}\int e^{-tz}\tr c^1\gamma^5(\hat{D}+c(\delta u))\sigma_z^{-j-1}q_j\psi|_{x=x'}\,du\,dz\nonumber\\
&= \frac{1}{2}\int \tr c^1\gamma^5(\hat{D}+c(\delta u))e^{-tc(\delta u)^2}\sum_{a+b+c = j}\frac{(-1)^{j+1}p_{a,b,c}t^{a+2b+3c}}{(a+2b+3c)!}\psi|_{x=x'}\,du, 
\end{align}
where $p_{a,b,c} = p_{a,b,c}(L,[L,c(\delta u)^2],[[L,c(\delta u)^2],c(\delta u)^2])$. Observe that as a polynomial in $u$, $p_{a,b,c}$ has degree at most $a+3b+4c$, and this highest degree is obtained only for the summand $p_{a,b,c}(L_1,[L_1,c(\delta u)^2],[[L,c(\delta u)^2],c(\delta u)^2])$, where $L_1$ is defined in \eqref{L1def}.  

\begin{lemma}\label{elembound}
For $\chi\geq 0 $, $p>-\frac{1}{2}$ and $a\not\in\mathbb{Z}$, $$\sum_{ \kappa\in \IZ}\int_\chi^\infty t^{{p}}e^{-te^{2y}V^2{( \kappa+a)^2}}dt\leq C_pe^{-2(p+1)y},$$ for some $C_p>0$.   For $0<\delta<1$ and  $b$ odd 
$$\sum_{k\in \IZ}\int_{e^{-(2+\delta)y}}^\infty t^{q}\left[e^{y}V( \kappa+a)\right]^be^{-te^{2y}V^2 ( \kappa+a)^2}dt = O\left(e^{-(2q+2-b) y}\right),$$
 For $b$ even we have 
\begin{multline*}
\sum_{ \kappa\in \IZ}\int_{e^{-(2+\delta)y}}^\infty t^{q}\left[e^{y}V ( \kappa+a)\right]^be^{-te^{2y}V^2( \kappa+a)^2}dt\\
 = O\left(e^{-(2q+2-b) y}\right) +   O\left( e^{\frac{\delta}{2} y-(2q+2-b)(1+\frac{\delta}{2}) y}\right).
\end{multline*}
\end{lemma}
\begin{proof}
We break the integration interval into two pieces $[e^{-(2+\delta)y},e^{-2y}]$ and $[e^{-2y},\infty).$ For the latter interval
\begin{multline*}
\sum_{ \kappa\in \IZ}\int_{e^{- 2y}}^\infty t^{q}\left[e^{y}V{( \kappa+a)}\right]^be^{-te^{2y}V^2{( \kappa+a)^2}}dt\\ 
= \sum_{ \kappa\in \IZ}e^{- (2+2q-b)y}\int_{1}^\infty t^{q}\left[ V{( \kappa+a)}\right]^be^{-t V^2{( \kappa+a)^2}}dt 
= O\left(e^{- (2+2q-b)y}\right).
\end{multline*}
To  estimate  the remaining integral, we first transform the sum via the Poisson summation formula:
\begin{multline*}
\sum_{ \kappa\in\IZ}[e^{y} ( \kappa+a)]^be^{-te^{2y}V^2{( \kappa+a)^2}}\\
 = {\sqrt{\pi}}e^{(b-1)y}V^{-1}t^{-1/2}\sum_{p\in \IZ}e^{2\pi ipa}\left(\frac{i}{2\pi}\frac{\p}{\p p}\right)^be^{- t^{-1}{\pi^2}  e^{-2y}V^{-2}p^2}.
\end{multline*}
In such formulas, we differentiate before evaluating at $p$ integer.
Hence,   
\begin{align*}
&\sum_{ \kappa\in \IZ}\int_{e^{-(2+\delta)y}}^{e^{-2y}} t^{q}[e^{y} ( \kappa+a)]^be^{-te^{2y}V^2{( \kappa+a)^2}}dt\\
 &= \sum_{p\in \IZ}e^{2\pi ipa}\left(\frac{i}{2\pi}\frac{\p}{\p p}\right)^b\int_{e^{-(2+\delta)y}}^{e^{-2y}} t^{q-\frac{1}{2}}{\sqrt{\pi}}e^{(b-1)y}V^{ -1}e^{- t^{-1}{\pi^2 } e^{-2y}V^{-2}p^2  }dt\\
&= \sum_{p\in \IZ}e^{2\pi ipa}\left(\frac{i}{2\pi}\frac{\p}{\p p}\right)^b\int_1^{e^{ \delta y}}  e^{-(2q+2-b) y}t^{-q-\frac{3}{2}}{\sqrt{\pi}}V^{ -1}e^{- t {\pi^2}V^{-2}p^2  }dt. 
\end{align*}
For $b$ odd, $p=0$ does not contribute, and this yields
$$\sum_{ \kappa\in \IZ}\int_{e^{-(2+\delta)y}}^{e^{-2y}} t^{q}[e^{y} ( \kappa+a)]^be^{-te^{2y}V^2{( \kappa+a)^2}}dt 
= O\left(e^{-(2q+2-b) y}\right).$$
For $b$ even, $p=0$ contributes and we have 
\begin{multline}
\sum_{ \kappa\in \IZ}\int_{e^{-(2+\delta)y}}^{e^{-2y}} t^{q}[e^{y} ( \kappa+a)]^be^{-te^{2y}V^2{( \kappa+a)^2}}dt\\ 
 = O\left(e^{-(2q+2-b) y}\right)
+   O\left( e^{\frac{\delta}{2} y-(2q+2-b)(1+\frac{\delta}{2}) y}\right) .
\end{multline}
\end{proof}

We apply this lemma to  eliminate most terms in the expression \eqref{step1}. 
\begin{lemma}\label{vange2}
$$ \lim_{y\to \infty}\int\limits_{e^{-(2+\delta)y}}^\infty\int\limits_{\partial M_y}\frac{1}{4\pi i}\int\limits_{C_t}\int \sum_{j>2}^Ne^{-tz}\tr c^1\gamma^5(\hat{D}+c(\delta u))\sigma_z^{-j-1}q_j\psi|_{x=x'}\,du\,dz\, d\nu_x dt=0. 
$$
\end{lemma}
\begin{proof}
Make the change of variables $u= (\kappa,v)\to (\kappa, t^{-1/2}v)$. After this change of variables, $p_{a,b,c}$ as a polynomial in $(e^y\kappa,t^{-\frac{1}{2}})$ has degree at most $a+3b+4c$, with this maximal degree only obtained in the summand\\ $p_{a,b,c}(L_1,[L_1,c(\delta u)^2],[[L,c(\delta u)^2],c(\delta u)^2])$, and $c(\delta u)$ becomes a polynomial of degree at most one in these variables. Hence substituting expression \eqref{step11} into \eqref{step1}, making the change of variables and integrating with respect to $v$ leaves us with integrals given by the product of the trace of  bounded endomorphisms and  integrals of the following form.
\begin{align}\label{jpc} &\int_{e^{-(2+\delta)y}}^\infty t^{a+2b+3c-\frac{3}{2}-m-\frac{\epsilon}{2}}\left[e^{y}V ( \kappa-\Lambda)\right]^{a+3b+4c-2m}e^{-te^{2y}V^2( \kappa-\Lambda)^2}dt  \nonumber\\
&= O\left(e^{-(j+c -1 - \epsilon ) y}\right)
+   O\left( e^{\frac{\delta}{2} y-(j+c-1 - \epsilon )(1+\frac{\delta}{2}) y}\right),
\end{align}   
 with $a+b+c=j$, and   $\epsilon = 1$ for the $c(\delta u)$ summand and $0$ for the $\hat D$ summand in the $\tr\gamma^5c(\nu)(\hat D+ c(\delta(u))$ term. Moreover, the $O\left( e^{\frac{\delta}{2} y-(j+c-1 - \epsilon )(1+\frac{\delta}{2}) y}\right)$ term only appears when $a+b$ is even. In particular, the  terms in \eqref{step1} (as expanded in \eqref{step11}) with $j+c>1+\epsilon$ are exponentially decreasing. 
\end{proof} 
To obtain further vanishing we require an algebraic lemma. 
\begin{lemma}[\cite{roe},\cite{BGV}]\label{cliffvan}  If $\alpha$ is a  $p-$form, with $p<4$, then $\tr \gamma^5c(\alpha) = 0$. 
\end{lemma}
We need  more information about the  connection components $\gamma_{jl}^m:= \langle \nabla_{e_j}e_l,e_m\rangle$ in order to exploit the preceding lemma. For the convenience of the reader, we record all these terms to the requisite accuracy. We have 
\begin{align}\label{gammam}
\gamma^1_{44} &= 1-\frac{ k}{4\ell}e^{-y}+O(e^{-2y}),&\gamma^2_{33} &= -\cot(\phi) +O(e^{-2y}) =O(e^{-\frac{3y}{4}}), \nonumber\\
\gamma_{23}^4 &= \frac{ k}{4\ell}e^{-y}+O(e^{-2y}),&
\gamma_{32}^4 &= -\frac{ k}{4\ell}e^{-y}+O(e^{-2y}),\\
\gamma_{42}^3 &= -\frac{ k}{4\ell}e^{-y}+O(e^{-2y})
 .\nonumber 
\end{align}
Here $k$ is the number of centers in TN$_k$. 
All other terms not related to the above by the relation $\gamma_{ij}^m = -\gamma_{im}^j$ are $O(e^{-2y})$. 
\begin{lemma}
\begin{align}\label{q0cont}& \lim_{y\to \infty}\int\limits_{e^{-(2+\delta)y}}^\infty\int\limits_{\partial M_y}\frac{1}{4\pi i}\int\limits_{C_t} e^{-tz}\tr c^1\gamma^5(\hat{D}+c(\delta u))\sigma_z^{-1}q_0\psi|_{x=x'}\,du\,dz\, d\nu_x dt\nonumber\\
&=   \lim_{y\to \infty}\int\limits_{e^{-(2+\delta)y}}^\infty\int\limits_{\partial M_y}  e^{-tc(\delta u)^2}\tr 
\gamma^5\frac{ke^{-y}}{16\ell}c(e^1\wedge e^4\wedge dVol_{S^2})   \,du \, d\nu_x dt. 
\end{align}
\end{lemma}
\begin{proof}By Lemma \ref{cliffvan}, $\tr \gamma^5 c^1c(\delta u)= 0$, since $c^1c(\delta u)$ contains no term which can be written as Clifford multiplication by a $4-$form. Writing 
\begin{align}\label{twistD}\hat D = c^m(e_m+A_m^{ab}-\gamma_{ii}^m)+ic^4\Lambda +\frac{ke^{-y}}{8\ell}c(e^4\wedge dVol_{S^2}) + O(e^{-2y}),
\end{align}
we see that the only term Clifford multiplication by a $4$-form $c^1\hat D$ contributes is $\frac{ke^{-y}}{8\ell}c(e^1\wedge e^4\wedge dVol_{S^2}) + O(e^{-2y})$, and the result follows. 
\end{proof}

\begin{lemma}
$$ \lim_{y\to \infty}\int\limits_{e^{-(2+\delta)y}}^\infty\int\limits_{\partial M_y}\frac{1}{4\pi i}\int\limits_{C_t}\int \sum_{j=2}^Ne^{-tz}\tr c^1\gamma^5(\hat{D}+c(\delta u))\sigma_z^{-j-1}q_j\psi|_{x=x'}\,du\,dz\, d\nu_x dt=0. 
$$
\end{lemma}
\begin{proof} In the proof of Lemma \ref{vange2}, we showed that terms with $j+c>1+\epsilon$ (in the notation of that proof) are exponentially decreasing. Hence for $j=2$, the only possible nonvanishing terms are $c=0$ and $\epsilon = 1$ arising in the summand
$-\tr \gamma^5 c^1c(\delta u)p_{a,b,0}(L_1,[L_1,c(\delta u)^2],1)$. By Lemma \ref{cliffvan}, $L_1$ must provide two additional Clifford terms for the trace to be nonzero. These terms can only come from the spin connection. We write 
$\nabla_{e_j} = e_j + A_j +\frac{1}{4}\gamma_{jk}^mc(e^k\wedge e^m)$. In our frame and neighborhood, the only $\gamma_{jk}^m$ which are not exponentially decreasing are $\gamma_{44}^1 = -\gamma_{41}^4 = 1+O(e^{-y}).$ Hence the only nonexponentially decreasing Clifford term contributed by $L_1$ is $c(e^1\wedge e^4)$. Both $c^1c(\delta u)c(e^1\wedge e^4)$ and $c^1c(\delta u)c(e^1\wedge e^4)^2$ can be written as the sum of a scalar and Clifford multiplication by a 2-form. Hence by Lemma \ref{cliffvan}, $-\tr \gamma^5 c^1c(\delta u)p_{a,b,0}(L_1,[L_1,c(\delta u)^2],1)$ is exponentially decreasing, and the lemma follows. 
\end{proof}

\begin{lemma}
$$ \lim_{y\to \infty}\int\limits_{e^{-(2+\delta)y}}^\infty\int\limits_{\partial M_y}\frac{1}{4\pi i}\int\limits_{C_t}\int e^{-tz}\tr c^1\gamma^5 \hat{D} \sigma_z^{-2}q_1\psi|_{x=x'}\,du\,dz\, d\nu_x dt=0. 
$$
\end{lemma}
\begin{proof} 
 By \eqref{jpc}, the contribution of  $\tr c^1\gamma^5 \hat{D}\sigma_z^{-2}q_1\psi $ to \eqref{step11}  is the trace of a bounded endomorphism times $O(e^{- cy})$. Hence  terms with $c>0$ are exponentially decreasing. Moreover all terms other than the contribution from $\tr c^1\gamma^5 \hat{D}\sigma_z^{-2}p_{a,b,0}(L_1,[L_1,c(\delta u)^2],1)$, with $a+b=1$ are exponentially decreasing.  The only nonexponentially decreasing Clifford term contributed by $L_1$ is $c(e^1\wedge e^4)$. Hence $\hat D$ must contribute  Clifford multiplication by a 3-form in order to have nonzero trace. By \eqref{twistD}, the only 3-form $\hat D$ contributes is exponentially decreasing, and the result follows. 
\end{proof}
For the contribution of the $q_1$ term to \eqref{step1}, using \eqref{comL}, we are left with 
\begin{align}\label{beforediff}&J_1:=   \lim_{y\to \infty}\int\limits_{e^{-(2+\delta)y}}^\infty\int\limits_{\partial M_y}\frac{1}{4\pi i}\int\limits_{C_t}\int  e^{-tz}\tr c^1\gamma^5 c(\delta u)  (\sigma_z^{-2}L - \sigma_z^{-3}[L,c(\delta u)^2])\psi|_{x=x'}\,du\,dz\, d\nu_x dt\nonumber\\
& =  \lim_{y\to \infty}\int\limits_{e^{-(2+\delta)y}}^\infty\int\limits_{\partial M_y}\int
 \tr c^1\gamma^5 c(\delta u)e^{-tc(\delta u)^2}  ( \frac{t}{2}L + \frac{t^2}{4}[L,c(\delta u)^2])\psi|_{x=x'}\,du \, d\nu_x dt.
\end{align}
\begin{lemma}
\begin{align}\label{afterdiff0}  &J_1= \lim_{y\to \infty}\int\limits_{e^{-(2+\delta)y}}^\infty\int\limits_{\partial M_y}\int
 \tr\gamma^5 c^1  e^{-tc(\delta u)^2}   {\frac{k{e^{-y}}}{8\ell}}c^2c^3c^4\,du \, d\nu_x dt\nonumber\\
&+\lim_{y\to \infty}\int\limits_{e^{-(2+\delta)y}}^\infty\int\limits_{\partial M_y}\int
 \tr c^1\gamma^5   c^4e^yV{i( \kappa- \Lambda)}e^{-tc(\delta u)^2}   \frac{t}{2}c^2c^3(F_A(e_2,e_3) {+}\frac{ik}{4}(\kappa-\Lambda))\,du \, d\nu_x dt.
\end{align}
\end{lemma}
\begin{proof}
From the definition of $c(\delta u)$ it follows that
\begin{align}\label{diffid}tc(\delta u)e^{-tc(\delta u)^2} = [-\frac{i}{4\pi}\Big(c^1\partial_{v_1}+ c^2\partial_{v_2}+\sin\phi\, c^3 \partial_{v_3}\Big) +tc^4e^yV{i( \kappa- \Lambda)}]e^{-tc(\delta u)^2}.
\end{align}

 Inserting \eqref{diffid} into \eqref{beforediff} and integrating by parts in $v$ gives
\begin{align}\label{afterdiff}  &J_1= \lim_{y\to \infty}\int\limits_{e^{-(2+\delta)y}}^\infty\int\limits_{\partial M_y}\int
 \tr \gamma^5 c^1  e^{-tc(\delta u)^2}   \frac{1}{2}(  \sum_{j=1}^3c^j(\frac{1}{4}\gamma_{jl}^mc^lc^m+A(e_j) -  \omega(\bar e_j)i\Lambda)
\,du \, d\nu_x dt\nonumber\\
&+\lim_{y\to \infty}\int\limits_{e^{-(2+\delta)y}}^\infty\int\limits_{\partial M_y}\int
 \tr \gamma^5 c^1  e^{-tc(\delta u)^2}  (\frac{\cos(\phi)c^2}{2\sin(\phi)}+ \frac{t}{4}[(   c^1\nabla_{e_1} +   c^2  \nabla_{e_2} 
+     c^3  \nabla_{e_3} ),c(\delta u)^2])\psi|_{x=x'}\,du \, d\nu_x dt\nonumber\\
&+\lim_{y\to \infty}\int\limits_{e^{-(2+\delta)y}}^\infty\int\limits_{\partial M_y}\int
 \tr c^1\gamma^5   c^4e^yV{i( \kappa- \Lambda)}e^{-tc(\delta u)^2} ( \frac{t}{2}L + \frac{t^2}{4}[L,c(\delta u)^2])\psi|_{x=x'}\,du \, d\nu_x dt.
\end{align}
By lemma \ref{cliffvan}, 
$$\tr \gamma^5 c^1  e^{-tc(\delta u)^2} (\frac{\cos(\phi)c^2}{2\sin(\phi)}+   \frac{t}{4}[(   c^1\nabla_{e_1} +   c^2  \nabla_{e_2} 
+     c^3  \nabla_{e_3} ),c(\delta u)^2])=0,$$ 
and
$$\tr \gamma^5 c^1  e^{-tc(\delta u)^2}   \frac{1}{2}(  \sum_{j=1}^3c^j(A(e_j) -  \omega(\bar e_j)i\Lambda)=0.$$

By \eqref{gammam}, we have 
$$ \sum_{j=1}^3c^j \frac{1}{2}\gamma_{jl}^mc^lc^m =   (\frac{k{e^{-y}}}{2\ell}+O(e^{-2y}))c^2c^3c^4
-  \gamma_{33}^m c^m.$$
In order to simplify the last line of \eqref{afterdiff}, we remove all the terms in $L$ which are odd in $v_1,$ $2\pi  v_2-\omega(\bar e_2) (\kappa-\Lambda)$, $ 2\pi  v_3-\omega(\p_\theta) (\kappa-\Lambda)$, or $\phi$, as they integrate to zero. Among the remaining terms, we then identify all the terms in $L$ which have a $c^2c^3$ factor and no $c^1$ or $c^4$ factor. By Lemma \ref{cliffvan} these are the only terms which do not trace to zero in 
$\tr c^1\gamma^5   c^4  ( \frac{t}{2}L + \frac{t^2}{4}[L,c(\delta u)^2])\psi|_{x=x'}$. Using \eqref{gammam} we compute that the remaining $c^2c^3$ term in $L\psi_{x=x'}$ is
$$ c^2c^3\left(F_A(e_2,e_3) {+}\frac{ik}{4}(\kappa-\Lambda)+O(e^{-y})\right), $$
 which vanishes in the commutator with $c(\delta u)^2$. 
The result follows. 
\end{proof}

\subsection{The Asymptotic Contribution}
After the reductions obtained in the previous section, the computation of the index reduces to the computation of the sum and integral of \eqref{q0cont} and \eqref{afterdiff0}. The elements $\frac{\lambda_a}{\ell}$ of the diagonal matrix $ \Lambda$ only enter our computations as $\frac{\lambda_a}{\ell}+\IZ$. We denote by $\{{\lambda_a}/{\ell}\}$ the unique representative of $\frac{\lambda_a}{\ell}+\IZ$ in the interval $[0,1)$, and we let $\{ \Lambda\}$ denote the diagonal matrix with entries $\{{\lambda_a}/{\ell}\}.$ We denote $F_A(e_2,e_3)$ by $F_{23}$, and let $F_{23}^0$ denote the zeroth Fourier coefficient of $F_{23}$ in the given frame. 
\begin{theo}
The asymptotic contribution to the index equals
\begin{multline*}
\lim_{y\to \infty}\frac{1}{2}\int_{e^{-(2+\delta)y}}^\infty\int_{\partial M_y}\tr\, c(\nu)({-}\gamma^5)De^{-tD^2}(x,x)\\
=\frac{k}{2}\tr_{_{\scriptscriptstyle \mathcal{E}}}\left(\{{ \Lambda}\}^2-\{{ \Lambda}\}+\frac{1}{6}\right)+\frac{1}{2}\int_{S^2_\infty} \tr_{_{\scriptscriptstyle \mathcal{E}}}\left(\frac{\{{ \Lambda}\}}{\pi}iF_{23}^0-\frac{iF_{23}^0}{2\pi}\right)d\mathrm{Vol}_{S^2}.
\end{multline*}
\end{theo}
\begin{proof}
We must evaluate the sums and integrals corresponding to each one of the three summands in \eqref{q0cont} and \eqref{afterdiff0}. Start with 
\begin{align*}
&{\int\limits_{e^{-(2+\delta)y}}^\infty\int\limits_{\partial M_y}\int
 \tr c^1\gamma^5   c^4e^yV{i( \kappa- \Lambda)}e^{-tc(\delta u)^2}   \frac{t}{2}c^2c^3F_{23} \,du \, d\nu_x dt.}\\
&={\sum_ \kappa 2} \int\limits_{e^{-(2+\delta)y}}^\infty\int_{\mathbb{R}^3}  \tr_{_{\scriptscriptstyle \mathcal{E}}} e^{-t {{e^{2y}V^2}}( \kappa- \Lambda)^2}e^{-t 4\pi^2(v_1^2+v_2^2+\frac{v_3^2}{\sin^2\phi} )}    e^yV {{i}( \kappa- \Lambda)}   tF_{23} dv dt \\
&={\sum_ \kappa 2 }\int\limits_{e^{- \delta y}V^2}^\infty  \tr_{_{\scriptscriptstyle \mathcal{E}}} \sin \phi\, e^{- {{t}( \kappa- \Lambda)^2}}(4\pi )^{-3/2}   {{i}( \kappa- \Lambda) }  F_{23} t^{-1/2} dt .
\end{align*}  
The Poisson summation formula implies 
$$\sum_{ \kappa\in \mathbb{Z}}( \kappa+a)e^{-4\pi^2s( \kappa+a)^2} = \sum_{p=1}^\infty 2p\sin(2\pi pa)(4\pi s)^{-3/2}e^{-p^2/4s}.$$ 
This transforms the last integral to
\begin{multline*}
{{-\frac{1}{4}}}\tr_{_{\scriptscriptstyle {\mathcal{E}}}}\sum_{p=1}^\infty2p\sin(2\pi p \Lambda)\sin \phi \int_{e^{- \delta y}V^2}^\infty  e^{- \frac{ {\pi^2p^2} }{t} }   {{i}}   F_{23} t^{-2} dt\\
= {{-\tr_{_{\scriptscriptstyle \mathcal{E}}}\,  \frac{i}{4} } } F_{23} \sum_{p=1}^\infty2p\sin(2\pi p \Lambda)\sin \phi \int_0^{e^{ \delta y}V^{-2}} 
  e^{{- t  \pi^2p^2}  }   dt.
 \end{multline*}
In the limit as $y\to\infty$ this reduces to 
\begin{align}\label{bndrya}
{-\frac{1}{2}\tr_{_{\scriptscriptstyle \mathcal{E}}}}\sum_{p=1}^\infty \frac{\sin(2\pi p \Lambda)}{ \pi^2p }i F_{23}^0\sin \phi.
\end{align}  Recalling the Fourier expansion of the Bernoulli polynomials \cite[Sec.1.13]{BatemanProject} for $x\in(0,1):$ 
\begin{align}\label{Bernoulli}
{\frac{1}{n!}B_n(x)=-\sum_{p\neq 0}\frac{e^{2\pi i p x}}{(2\pi i p)^n},}
\end{align}
we have 
$\frac{1}{2}-\{a\}= \sum_{p=1}^\infty \frac{\sin(2\pi pa)}{\pi p },$ where $a\in\mathbb{R}\setminus\mathbb{Z},$ and we denote by $\{a\}\in [0,1)$  the unique representative of $a+\IZ$ in that interval.   The sum \eqref{bndrya} then reduces, under the  Assumption~\ref{assume}, to
\begin{align}\label{sinphi}
{\frac{1}{2}}\tr_{\mathcal{E}}\left(\frac{\{{ \Lambda}\}}{\pi}-\frac{I}{2\pi}\right) iF_{23}^0 \sin \phi. \end{align}

We recall that in our parametrix construction, we had replaced $d\mathrm{Vol}_{S^2}$ by $ d\phi\wedge d\theta$. The factor of $\sin \phi$ in (\ref{sinphi}) restores the usual volume form, and the contribution of  the final summand of  (\ref{afterdiff0}) to the index is 
 $${\frac{1}{2}} \int_{S^{2}_\infty}\tr_{\mathcal{E}}\left(\frac{\{{ \Lambda}\} iF_{23}^0}{\pi}-\frac{iF_{23}^0}{2\pi}\right) d\nu.
$$

Now we consider the remaining terms. 
\begin{align}\label{remainders}
&{\int\limits_{e^{-(2+\delta)y}}^\infty\int\limits_{\partial M_y}\int
 \tr_\mathcal{E} e^{-tc(\delta u)^2}\bigg(     {\frac{k{e^{-y}}}{{4}\ell}}
-
e^yV  \frac{ikt}{2}(\kappa-\Lambda)^2\bigg)\,du \, d\nu_x dt.}
\end{align}
The Poisson  summation formula implies
\begin{align*}
\sum_{ \kappa\in \mathbb{Z}}e^{-4\pi^2s( \kappa+a)^2} &= \sum_{p\in \mathbb{Z}}(4\pi s)^{-1/2}e^{-p^2/4s}e^{2\pi ipa},\text{ and}\\
\sum_{ \kappa\in \mathbb{Z}}e^{-4\pi^2s( \kappa+a)^2}4\pi^2s( \kappa+a)^2 &= \sum_{p\in \mathbb{Z}}(4\pi s)^{-1/2}e^{-p^2/4s}e^{2\pi ipa}\left(\frac{1}{2}-\frac{p^2}{4s}\right).
\end{align*}
Computing the integrals and applying Poisson summation formulas as before simplifies \eqref{remainders} to
\begin{multline*}
\frac{e^{-2\pi ip\Lambda}k \sin\phi}{16\pi}\sum\limits_{p\neq 0}\int\limits_{e^{-\delta y}V^2}^\infty t^{-2}e^{-\frac{\pi^2 p^2}{t}}\frac{\pi^2 p^2}{t}\,dt
=\frac{k}{8\pi}\sin\phi\sum_{p\neq 0}\frac{\cos {2\pi p\Lambda} }{2\pi^2p^2}\\
=\frac{k\sin\phi}{4\pi}\frac{1}{2}\big(\{\Lambda\}^2-\{\Lambda\}+\frac{1}{6}\big),
\end{multline*}
where we have used the Bernoulli polynomial Taylor expansion \eqref{Bernoulli}, 
$$ \sum_{p> 0}^\infty  \frac{\cos(2\pi pa)}{\pi^2 p^2} = B_2(a) = \{a\}^2-\{a\}+\frac{1}{6},$$ 
\end{proof}
Assembling the above results we obtain
 \begin{multline}\label{indprelim}
  \mathrm{ind}_{L^2} D^-=\frac{k}{2}\tr(\{{\Lambda}\}^2-\{{\Lambda}\})
  +\frac{1}{2}\tr\int_{S^2_\infty}\left(\frac{\{{\Lambda}\}}{\pi}iF_{23}^0
 -\frac{iF_{23}^0}{2\pi}\right)d\mathrm{Vol}_{S^2}\\
+\frac{1}{8\pi^2}\int \tr F\wedge F.
\end{multline}
From the asymptotic form of the connection of Theorem~\ref{asyholo} we easily evaluate the boundary contribution, given by the first line of \eqref{indprelim}.  Letting $M=\mathrm{diag}(m_a),$  we have  
$\frac{i}{2\pi}\int_{S^2_\infty}F_{23}^0\,d\mathrm{Vol}_{S^2}
=M-k\Lambda.$  Thus we obtain
\begin{theo}\label{bigind}
 The index of the Dirac operator $D_A$ equals 
 \begin{multline}\label{indfor}
 \mathrm{ind}_{L^2} D^-=
 \mathrm{tr}\left(\frac{k}{2}\{\Lambda\}^2 -\frac{k}{2}\{\Lambda\}-\{\Lambda\}(k\Lambda-M)+\frac{1}{2}(k\Lambda-M)\right)
 \\
+\frac{1}{8\pi^2}\int \tr F\wedge F.
\end{multline}

\end{theo}

Let's apply this formula to the abelian instanton on TN$_1$. It is given by $A=-i\frac{s}{2V}(d\tau+\omega)$, with curvature $F=dA=-i(d(s/2V)\wedge (d\tau+\omega)+(s/2V)d\omega)$. Therefore, $\Lambda=s/\ell , M=0,$ and $\frac{1}{8\pi^2}\int F\wedge F=\frac{1}{2}(s/\ell )^2$. The above index formula reduces to $\lfloor s/\ell \rfloor(\lfloor s/\ell \rfloor+1)/2,$  in agreement with \cite{Pope2}, where the solutions of the Dirac operator on TN$_1$ in this background were explicitly found. These solutions were  studied more recently in  \cite{Jante:2013kha, Jante:2015xra}.

As another illustration, consider a Whitney sum $\oplus_{j=1}^n L_j$ of line bundles with  abelian connection one-forms $-i a_j=-i \frac{H_j}{V}(d\tau+\omega)+\pi^*_k(\eta_j)$ on $L_j,$ with $H_j=\lambda_j+\sum_\sigma\frac{v_{j\sigma}}{2 r_\sigma}$ and $d\eta_j=*_3dH_j.$  Thus, we have an instanton connection one-form $A=-i\mathrm{diag}(a_j).$ Its second Chern character value is
$
\frac{1}{8\pi^2}\int F_A\wedge F_A=\frac{1}{2}\left(
k\Lambda^2-2\Lambda M+\mathrm{diag}(\sum_\sigma (v_{j\sigma})^2)
\right),
$
giving
\begin{align}\label{last}
 \mathrm{ind}_{L^2} D^-=\sum_{j=1}^n\sum_{\sigma=1}^k
  \frac{\left(\lfloor\lambda_j/\ell \rfloor-v_{j\sigma}\right)\left(\lfloor\lambda_j/\ell \rfloor-v_{j\sigma}+1\right)}{2}.
\end{align}

\section*{Acknowledgements}
We thank the anonymous referee for his or her numerous suggesions for improving our exposition. 
SCh is grateful to the Institute for Advanced Study, Princeton and to the Institute des Hautes \'Etudes Scientifiques for their hospitality and support during various stages of this work; he also thanks the Berkeley Center for Theoretical Physics for hospitality during its completion.  The work of SCh was partially supported by the Simons Foundation grant {\footnotesize\#}245643. The work of MAS was partially supported by the Simons Foundation grant {\footnotesize\#}353857 and NSF grant DMS 1005761. The work of ALH was supported by an NSF Alliance Postdoctoral Fellowship.

\bibliographystyle{amsalpha}

\providecommand{\bysame}{\leavevmode\hbox to3em{\hrulefill}\thinspace}
\providecommand{\MR}{\relax\ifhmode\unskip\space\fi MR }
\providecommand{\MRhref}[2]{%
  \href{http://www.ams.org/mathscinet-getitem?mr=#1}{#2}
}
\providecommand{\href}[2]{#2}

\end{document}